%% file: neck-pinch.tex
\documentclass[12pt,oneside]{amsart}
\usepackage{latexsym}
\usepackage{amsmath,amssymb,amsthm}
\usepackage{amscd}
\usepackage[dvips]{graphics}
\newcommand{\rp}{\mathbb{RP}}
\newcommand{\re}{\mathbb{R}}
\newcommand{\co}{\mathbb {C}}

\newcommand{\pgl}[1]{\mathbf{PGL}(#1,\mathbb{R})}
\newcommand{\psl}[1]{\mathbf{PSL}(#1,\mathbb{R})}
\newcommand{\Sl}[1]{\mathbf{SL}(#1,\mathbb{R})}

\newcommand{\na}{\nabla}
\newcommand{\sfrac}[2]{{\textstyle \frac{#1}{#2}}}

\newtheorem{prop}{Proposition} 
\newtheorem{cor}[prop]{Corollary}
\newtheorem{thm}{Theorem}

\newtheorem{lem}[prop]{Lemma}

\theoremstyle{remark}
\newtheorem*{rem}{Remark}

\begin{document}
\title[Convex $\rp^2$ Structures Under Neck Separation]{Convex $\rp^2$ Structures and Cubic Differentials Under Neck Separation}
\author{John Loftin}

\maketitle

\begin{abstract}
Let $S$ be a closed oriented surface of genus at least two. Labourie and the author have independently used the theory of hyperbolic affine spheres to find a natural correspondence between convex $\rp^2$ structures on $S$ and pairs $(\Sigma,U)$ consisting of a conformal structure $\Sigma$ on $S$ and a holomorphic cubic differential $U$ over $\Sigma$.  The pairs $(\Sigma,U)$, for $\Sigma$ varying in moduli space, allow us to define natural holomorphic coordinates on the moduli space of convex $\rp^2$ structures.
We consider geometric limits of convex $\rp^2$ structures on $S$ in which the $\rp^2$ structure degenerates only along a set of simple, non-intersecting, non-homotopic loops $c$.  We classify the resulting $\rp^2$ structures on $S-c$ and call them regular convex $\rp^2$ structures.  We put a natural topology on the moduli space of all regular convex $\rp^2$ structures on $S$ and show that this space is naturally homeomorphic to the total space of the vector bundle over $\overline{\mathcal M}_g$ each of whose fibers over a noded Riemann surface is the space of regular cubic differentials. In other words, we can extend our holomorphic coordinates to bordify the moduli space of convex $\rp^2$ structures along all neck pinches. The proof relies on previous techniques of the author, Benoist-Hulin, and Dumas-Wolf, as well as some details due to Wolpert of the geometry of hyperbolic metrics on conformal surfaces in $\overline{\mathcal M}_g$.
\end{abstract}

\section{Introduction}
A (properly) convex $\rp^2$ surface is given as a quotient $\Gamma\backslash \Omega$, where $\Omega$ is a bounded convex domain in $\re^2\subset \rp^2$ and $\Gamma$ is a discrete subgroup of $\pgl3$ acting discretely and properly discontinuously on $\Omega$. We assume our convex $\rp^2$ surfaces are oriented, and it is natural in this case to lift the action of $\pgl3$ to an action of $\Sl3$ acting on the convex cone over $\Omega$ in $\re^3$. Labourie and the author independently showed that a marked convex $\rp^2$ structure on a closed oriented surface $S$ of genus $g$ at least two is equivalent to a pair $(\Sigma,U)$, where $\Sigma$ is a marked conformal structure on $S$ and $U$ is a holomorphic cubic differential \cite{labourie97,labourie07,loftin01}.  This result relies on the geometry of hyperbolic affine spheres, in particular on results of C.P.\ Wang \cite{wang91} and deep geometric and analytic results of Cheng-Yau \cite{cheng-yau77,cheng-yau86}. This provides a complex structure on the deformation space $\mathcal G_S$ of marked convex $\rp^2$ structures on $S$.  Moreover, we can mod out by the mapping class group to find that the moduli space $\mathcal R_S$ of unmarked oriented convex $\rp^2$ structures is given by the total space of the bundle of holomorphic cubic differentials over the moduli space $\mathcal M_g$.  One may naturally extend the bundle of holomorphic cubic differentials to a (V-manifold) holomorphic vector bundle whose fiber is the space of \emph{regular cubic differentials} over each noded Riemann surface $\Sigma$ in the boundary divisor in the Deligne-Mumford compactification $\overline{\mathcal M}_g$.  In \cite{loftin02c}, for each pair $(\Sigma,U)$ of noded Riemann surface $\Sigma$ and regular cubic differential $U$ over $\Sigma$, we construct a corresponding $\rp^2$ structure on the nonsingular locus $\Sigma^{\rm reg}$, and specify the geometry near each node by the residue of the cubic differential there.  In this way, we may define a \emph{regular convex $\rp^2$ structure} on $\Sigma$.  There is a standard topology on the space of regular cubic differentials, and we define a topology on the space of regular convex $\rp^2$ structures under which geometric limits are continuous and which is similar in spirit to Harvey's use of the Chabauty topology to describe the Deligne-Mumford compactification \cite{harvey77}. Our main result is then

\begin{thm} \label{main-thm}
Let $S$ be a closed oriented surface of genus $g\ge2$.
The total space of the bundle of regular cubic differentials over the moduli space of $\overline{\mathcal M}_g$ is homeomorphic to the moduli space $\mathcal R^{\rm aug}_S$ of regular convex $\rp^2$ structures on $S$.
\end{thm}

In \cite{loftin02c}, we constructed regular convex $\rp^2$ structures corresponding to regular cubic differentials over a noded Riemann surface, and gave some local analysis of the families of regular convex $\rp^2$ structures in the limit. In passing from regular convex $\rp^2$ structures to regular cubic differentials, Benoist-Hulin show that finite-volume convex $\rp^2$ structures correspond to regular cubic differentials of residue zero \cite{benoist-hulin13}.   Quite recently, as this paper was being finalized, Xin Nie has classified all convex $\rp^2$ structures corresponding to meromorphic cubic differentials on a Riemann surface \cite{nie15}.  In the present work, we only consider cubic differentials of pole order at most three (as these are the only ones which appear under neck separation), and the $\rp^2$ geometry of each end is determined by the residue $R$, where $U= R \frac{dz^3}{z^3} + \cdots$. It should be interesting to determine how Nie's higher-order poles relate to degenerating $\rp^2$ structures.

For poles of order 3, there are three cases to consider, as determined by the residue $R$. If $R=0$, then the ends are parabolic, which is locally the same structure as a parabolic element of a Fuchsian group.  If ${\rm Re}\,R\neq0$, then the holonomy of the end is hyperbolic, while if $R\neq0$ but ${\rm Re}\,R=0$, the holonomy is quasi-hyperbolic.  These two cases are not present in the theory of Fuchsian groups. In particular, the associated Blaschke metric is complete and asymptotically flat, and has finite diameter at these ends. The geometry of $\rp^2$ surfaces  contains both flat and hyperbolic geometry as  limits.

The proof of the main theorem involves several analytic and geometric prior results.  First of all, the principal new estimates in the proof are to find sub- and super-solutions to an equation of C.P.\ Wang \cite{wang91} which are uniform for convergent families $(\Sigma_j,U_j)$ of noded Riemann surfaces and regular cubic differentials.  These will allow us to take limits along the families.  A uniqueness result of Dumas-Wolf \cite{dumas-wolf14} then shows that the limits we find are the ones predicted in \cite{loftin02c}.  To analyze limits of regular $\rp^2$ structures, we use a powerful technique of Benoist-Hulin, which shows that natural projectively-invariant tensors on convex domains converge in $C^\infty_{\rm loc}$ when the domains converge in the Hausdorff topology \cite{benoist-hulin13}.  We also use many details about the structure of the Deligne-Mumford compactification $\overline{\mathcal M}_g$, and in particular, the analytic framework due to Masur and refined by Wolpert to relate the hyperbolic metric and with the plumbing construction near the singular curves in $\partial \mathcal M_g$.

\subsection{$\rp^2$ structures and higher Teichm\"uller theory}
Goldman \cite{goldman88a} and Hitchin \cite{hitchin87} prove that the Teichm\"uller space of conformal structures on a closed oriented marked surface $S$ of genus at least 2 is homeomorphic to a connected component of the space of representations $\pi_1S\to\psl2$ modulo conjugation in $\psl2$. Consider representations of $\pi_1S$ to higher-order Lie groups is then known as higher Teichm\"uller theory.  Choi-Goldman \cite{choi-goldman93} show that the deformation space $\mathcal G_S$ of convex $\rp^2$ structures on $S$ is homeomorphic to the Hitchin component of representations $\pi_1S\to\psl3$ \cite{hitchin92}.  Goldman provided in \cite{goldman90a} the analog of Fenchel-Nielsen coordinates on $\mathcal G_S$.  Fenchel-Nielsen coordinates play an important role in analyzing $\overline{\mathcal M}_g$ going back to Bers \cite{bers74} and Abikoff \cite{abikoff80}.  In particular, Wolf-Wolpert \cite{wolf-wolpert92} determine the real-analytic relationship between the between complex-analytic coordinates on $\overline {\mathcal M}_g$ as given by Masur \cite{masur76} and the Fenchel-Nielsen coordinates.  In our present work, we have related the complex-analytic data of regular cubic differentials to the projective geometry of the convex $\rp^2$ structures, but we have not addressed Goldman's Fenchel-Nielsen coordinates.  It should be possible to do so, as Marquis has already extended Goldman's coordinates to pairs of pants with non-hyperbolic holonomy \cite{marquis10}.

There have also been many other works on limits of convex $\rp^2$ structures.  Anne Parreau has analyzed limits of group representations into Lie groups in terms of group actions on $\re$-buildings \cite{parreau00,parreau12}. Parreau thus provides an analog of Thurston's boundary of Teichm\"uller space.  Inkang Kim \cite{inkang-kim05} applies Parreau's theory to construct a compactification of the deformation space of convex $\rp^2$ structure $\mathcal G_S$.  Limits of cubic differentials were related to Parreau's picture in \cite{loftin06} and recently in \cite{nie15a}.  Dumas-Wolf have recently studied polynomial cubic differentials on $\co$ \cite{dumas-wolf14}, and they show  that the space of polynomial cubic differentials up to holomorphic equivalence is isomorphic via the affine sphere construction to the space of bounded convex polygons in $\re^2\subset \rp^2$ up to to projective equivalence. Their construction has been used by Nie \cite{nie15} to analyze the $\rp^2$ geometry related to higher-order poles of cubic differentials, and should be useful in other contexts as well. Benoist-Hulin have also studied cubic differentials on the Poincar\'e disk, and have shown that the Hilbert metric on a convex domain is Gromov-hyperbolic if and only if it arises from a cubic differential on the Poincar\'e disk with bounded norm with respect to the Poincar\'e metric \cite{benoist-hulin14}.

Tengren Zhang has considered degenerating families of convex $\rp^2$ structures with natural constraints on Goldman's parameters \cite{zhang13}.  Ludovic Marquis has studied convex $\rp^2$ structures and their ends from a different point of view from this paper \cite{marquis10,marquis12}. Recently Choi has analyzed ends of $\rp^n$ orbifolds in any dimension \cite{choi15a,choi15b}.

Fix a conformal structure $\Sigma$ on a closed oriented surface of $S$ of genus at least two. Let $G$ be  a split real simple Lie group with trivial center and rank $r$.  Hitchin uses Higgs bundles to parametrize the Hitchin component of the representation space from $\pi_1S$ to $G$ by the set of $r$ holomorphic differentials, which always includes a quadratic differential \cite{hitchin92}.  For $\psl3$, Hitchin specifies a quadratic and a cubic differential.  If the quadratic differential vanishes in this case, then Labourie has shown that we can parametrize the Hitchin component by the affine sphere construction $(\Sigma,U)$ for $U$ Hitchin's cubic differential (up to a constant factor) \cite{labourie07}.  Labourie has also recently shown that Hitchin representations for other split real Lie groups of rank 2 ($\mathbf{PSp}(4,\re)$ and split real $G_2$) can be parametrized by pairs $(\Sigma,V)$, where $\Sigma$ varies in Teichm\"uller space and $V$ is a holomorphic differential of an appropriate order \cite{labourie14}.  It would be very interesting to analyze these Hitchin representations similarly as $\Sigma$ approaches a noded Riemann surface and $V$ is a regular differential.  The relationship between the Higgs bundles and the relevant geometric structures is not as well developed in this case. See
\cite{guichard-wienhard08,baraglia-thesis}.

\subsection{Outline}

Section \ref{top-section} begins by  defining the topological space of regular convex $\rp^2$ structures.  First of all, we recount Goldman's theory of building convex $\rp^2$ surfaces by gluing together simpler surfaces along principal geodesic boundary components. Then we prove a few general topology lemmas about spaces of orbits of homeomorphisms under the quotient topology. These lemmas are used to show that the topology of regular convex $\rp^2$ structures is first countable, and thus can be described in terms of convergent sequences.  Then we define regular separated necks and show in Theorem \ref{geom-limits-regular} that these regular separated necks encompass all geometric limits of convex $\rp^2$ structures on $S$ which degenerate to convex $\rp^2$ structures on $S-\ell$, for $\ell$ a simple non-peripheral loop in $S$. Next, we define the augmented Goldman space of marked regular convex $\rp^2$ structures on $S$, which, similarly to augmented Teichm\"uller space, is a non-locally-compact bordification of the Goldman space (the deformation space of marked convex $\rp^2$ structures).  The definition is based on pairs $(\Omega,\Gamma)$ and encodes both the Hausdorff limits of convex domains of $\Omega_j$  and also the convergence of representations $\Gamma_j$ of subgroups of the fundamental group, all modulo a natural action of $\Sl3$.  Then we take a quotient by the mapping class group to define the augmented moduli space of convex $\rp^2$ structures $\mathcal R_S^{\rm aug}$.

In the final part of Section \ref{top-section}, we recall the plumbing construction for neighborhoods of noded Riemann surfaces in the boundary of moduli space, largely following Wolpert, and its relation to the complete hyperbolic metric on the regular part of each surface.  We then use these constructions to construct the standard topology on the total space of the bundle of regular cubic differentials over $\overline{\mathcal M}_g$. Roughly, we define a metric $m$ on each noded Riemann surface $\Sigma$ which is equal to the hyperbolic metric on the thick part of $\Sigma$ and a flat cylindrical metric on the collar and cusp neighborhoods making up the thin part of $\Sigma^{\rm reg}$. Then convergence of a sequence $(\Sigma_j,U_j)$ is defined as convergence of $\Sigma_i$ in $\overline{\mathcal M}_g$, together with $L^\infty_{m_j,{\rm loc}}$ convergence of the cubic differentials $U_j$.

In the next section, we discuss the basics of hyperbolic affine spheres \cite{cheng-yau77,cheng-yau86}. Let $\mathcal H$ be a hyperbolic affine sphere, which is a surface in $\re^3$ asymptotic to a cone over a bounded convex domain $\Omega$. $\mathcal H$ is diffeomorphic to $\Omega$ under projection to $\rp^2$, and any projective action on $\Omega$ lift to a special linear action on $\mathcal H$.  Two basic invariant tensors, the Blaschke metric and the cubic tensor, thus descend to $\Omega$.  The Blaschke metric induces an invariant conformal structure on $\Omega$, and thus on the quotient $\Gamma\backslash \Omega$.  The cubic tensor induces a holomorphic cubic differential $U$.

Starting from a pair $(\Sigma,U)$, we can recover the picture of $(\Omega,\Gamma)$ by introducing a background metric $g$ and solving Wang's integrability condition (\ref{u-eq}) for a conformal factor $e^u$.  Then $e^ug$ is the Blaschke metric, and if it is complete, we recover the global hyperbolic affine sphere $\mathcal H$ and $(\Omega,\Gamma)$.

The hyperbolic affine sphere over $\Omega$ can be defined as the radial graph of $-\frac1v$ for  $v$ a convex solution to the Monge-Amp\`ere equation
$$ \det v_{ij} = \left(-\frac1v\right)^4$$ with zero Dirichlet boundary condition at $\partial \Omega$. We recall and give a new proof of Benoist-Hulin's result that the Blaschke metrics and cubic tensors converge in $C^\infty_{\rm loc}$ on convex domains converging in the Hausdorff sense \cite{benoist-hulin13}.  We also prove a new result, Proposition \ref{separate-O}, concerning sequences of pairs of points $x_j,y_j\in\Omega_j$, and show that if the Blaschke distance between $x_j$ and $y_j$ diverges to infinity, then any Benz\'ecri limits of the pointed space $\rho_j(\Omega_j,x_j)$ and $\sigma_j(\Omega_j,y_j)$ for sequences $\rho_j,\sigma_j\in\Sl3$, must be disjoint (in a sense made precise below).

We also discuss the two-dimensional case, due to Wang \cite{wang91}, and recall how to solve an initial-value problem to produce a hyperbolic affine sphere $\mathcal H$ from the data $(\Sigma,U,h)$, where $h$ is a complete Blaschke metric.

Finally, we begin the proof of Theorem \ref{main-thm} in Section \ref{cubic-to-rp2-sec}. In this section, we show that a convergent regular sequence $(\Sigma_j,U_j)\to(\Sigma_\infty,U_\infty)$ of pairs of noded Riemann surfaces and regular cubic differentials produces regular convex $\rp^2$ structures which converge in the limit to the convex $\rp^2$ structure corresponding to $(\Sigma_\infty,U_\infty)$.  The proof follows by the method of sub- and super-solutions.  We produced a locally bounded family of sub- and super-solutions to (\ref{u-eq}) uniform over the regular parts of $\Sigma_j^{\rm reg}$. This allows us to solve the equation (\ref{u-eq}) and to take the limit of Blaschke metrics as $j\to\infty$. A uniqueness theorem of Dumas-Wolf \cite{dumas-wolf14} shows that this limit is the Blaschke metric on $(\Sigma_\infty^{\rm reg}, U_\infty)$. We then use techniques of ordinary differential equations to show the holonomy and developing maps converge.

In the final Section \ref{rp2-to-cubic-sec}, we prove the remaining part of the main theorem, by showing that we can pass from convergent sequences of regular convex $\rp^2$ structures to convergent sequences of regular cubic differentials over noded Riemann surfaces.  The proof depends on the thick-thin decomposition of hyperbolic surfaces. In particular, we use Proposition \ref{separate-O} and lower bounds on the Blaschke metric in terms of the hyperbolic metric, to rule out Benz\'ecri sequences in which the point approaches the thin part in moduli.  Conversely, if we have points converging in the same component of the thick part of moduli, we use the uniform bounds on the diameter and the ODE theory from Section \ref{cubic-to-rp2-sec} to show the limit of the domains must be the same up to an $\Sl3$ action.

\subsection{Acknowledgments}

I am grateful to acknowledge stimulating conversations with Zeno Huang, Jane Gilman, Sara Maloni, and Bill Abikoff, and would like to give special thanks to Mike Wolf and David Dumas for many illuminating discussions, and to Scott Wolpert for many interesting discussions related to this work.  This work was supported in part by a Simons Collaboration Grant for Mathematicians 210124; U.S. National Science Foundation grants DMS 1107452, 1107263, 1107367 ``RNMS:
GEometric structures And Representation varieties" (the GEAR Network); and the IMS in at the National University of  Singapore, where part of this work was completed.

\section{Definitions and topology} \label{top-section}

\subsection{Goldman's attaching across a neck}

We recount some of the basic facts about $\rp^n$ manifolds. An $\rp^n$ manifold is a defined by a maximal atlas of coordinate charts in $\rp^n$ with gluing maps in $\pgl{n+1}$; in other words, there is an $(X,G)$ structure in the sense of Thurston and Ehresmann for $X=\rp^n$ and $G=\pgl{n+1}$.  A \emph{geodesic} in an $\rp^n$ manifold is a path which is a straight line segment in each $\rp^n$ coordinate chart.

See e.g.\ Goldman \cite{goldman90a} for details. An $\rp^n$ structure on an $n$-manifold $M$ can be described in terms of the development-holonomy pair. Choose a basepoint $p\in M$. The developing map is a local diffeomorphism from ${\rm dev}\!:\tilde M\to \rp^n$, while the holonomy ${\rm hol}\!:\pi_1M \to \pgl{n+1}$. Dev and hol are related by the following equivariance condition: if $\gamma\in\pi_1M$, then $${\rm dev}\circ \gamma = {\rm hol}(\gamma)\circ {\rm dev}.$$  The developing map is defined in terms of a choice of $\rp^n$ coordinate chart around $p\in M$. First lift this chart to a neighborhood in $\tilde M$, and then analytically continue to define dev on all of $\tilde M$.  For any other choice of coordinate chart and/or basepoint, there is a map $g\in\pgl{n+1}$ so that $$ {\rm dev}' = {\rm dev} \circ g, \qquad {\rm hol}'(\gamma) = g^{-1} \circ {\rm hol}(\gamma) \circ g.$$

An $\rp^n$ manifold $X$ is called \emph{convex} if the image of the developing map is a convex domain $\Omega$ in an inhomogeneous $\re^n\subset \rp^n$, and $X$ is a quotient  ${\rm hol}(\pi_1X) \backslash \Omega$. $X$ is \emph{properly convex} if $\Omega$ is in addition bounded in an inhomogeneous $\re^n\subset \rp^n$. All the manifolds we study in this paper are properly convex, and we often simply call them convex.

On any closed oriented convex $\rp^2$ surface of genus at least 2, the $\rp^2$ holonomy around any nontrivial simple loop is hyperbolic, in that it is conjugate to a diagonal matrix $D(\lambda,\mu,\nu)$, where $\lambda>\mu>\nu>0$ and $\lambda\mu\nu=1$. Choose coordinates in $\rp^2$ so that this holonomy action is given by $H=D(\lambda,\mu,\nu)$.  The three fixed points of this action are the attracting fixed point $[1,0,0]$, the repelling fixed point $[0,0,1]$, and the saddle fixed point $[0,1,0]$. Define the \emph{principal triangle} $T$ as the projection onto $\rp^2$ of the first octant in $\re^3$. The \emph{principal geodesic} $\tilde\ell$ associated to this holonomy matrix is the straight line in the boundary of $T$ from the repelling to the attracting fixed point.  Let $\bar T$ denote the triangle given by the reflection of $T$ across the principal geodesic given by the matrix $J=D(1,-1,1)$. The vertices of the principal triangle are the fixed points of the holonomy matrix.   The quotient $T/\langle H\rangle$ is called the \emph{principal half-annulus}, while the quotient of $(T\sqcup\tilde\ell\sqcup\bar T)/\langle H \rangle$ is called the \emph{$\pi$-annulus}.

We recall Goldman's theory of attaching $\rp^2$ surfaces across a principal geodesic. On a compact properly convex $\rp^2$ surface $S^a$ with principal geodesic boundary, an annular neighborhood of a principal geodesic boundary component $\ell$ is called a \emph{principal collar neighborhood}. We may choose coordinates so that a lift $\tilde\ell$ of $\ell$ is the standard principal geodesic mentioned above.  Assume the image $\Omega^a$ of the developing map is then a subset of the principal triangle $\bar T$.  The principal collar neighborhood then develops to be a neighborhood $\mathcal N$ of $\tilde\ell$ in $\bar T$ which is invariant under the action of the holonomy matrix $H=D(\lambda,\mu,\nu)$.  Let $\gamma\in\pi_1S^a$ represent the loop $\ell$.  Then the quotient $\mathcal N^a= \langle H\rangle \backslash \mathcal N$ is the principal collar neighborhood.

Now consider a second convex $\rp^2$ surface $S^b$ with principal geodesic boundary, together with a principal geodesic boundary component. Choose local $\rp^2$ coordinates so that the lift of this geodesic boundary loop is $-\tilde \ell$ (the minus sign denoting the opposite orientation), and the image $\Omega^b$ of the developing map of $S^b$ is contained in $T$. If the holonomy around $-\tilde\ell$ is $H^{-1}$, then $H$ acts on $\Omega^a\sqcup \tilde\ell \sqcup \Omega^b$.  (In order to glue the surfaces along $\ell$, we need to glue across all the lifts of $\ell$, which we may describe as $ {\rm hol}(\beta) \circ \tilde\ell$ for $\beta$ in the coset space $\pi_1(S^a)/\langle \gamma\rangle$, where $\gamma$ is the element in $\pi_1$ determined by the loop $\ell$.)

\begin{thm}[Goldman] \label{glue-rp2}
Let $M$ be a properly convex $\rp^2$ surface with principal geodesic boundary. (In particular, we assume the boundary is compact and each component of the boundary is the quotient of a principal geodesic under a holonomy action of hyperbolic type). $M$ is not assumed to be either connected or compact. Let $B_1,B_2$ be two boundary components, and assume that they have principal collar neighborhoods $N_1,N_2$  respectively which are projectively isomorphic under an orientation-reversing map across the boundary. This induces a projective structure on a full neighborhood of the geodesic formed from gluing $B_1$ and $B_2$. The resulting $\rp^2$ surface $\bar M$ is also convex. Moreover, $\bar M$ is properly convex except in the case of gluing two principal half-annuli together to make a $\pi$-annulus.
\end{thm}

\begin{rem}
Goldman states this theorem (Theorem 3.7 in \cite{goldman90a}) a bit differently, in that the hypothesis is that $M$ is compact as an $\rp^2$ surface with boundary. However, the compactness condition is not used in Goldman's proof.  What is used is the condition of being properly convex.  The image of the developing map of a properly convex $M$ with a hyperbolic holonomy along the principal boundary geodesic must be properly contained in a principal triangle.    \emph{Proper} convexity is essential for us, as the $\pi$-annulus, which has principal boundary geodesics and is convex but not properly convex, cannot be glued to another $\rp^2$ surface while maintaining convexity.  In fact, Choi \cite{choi94a,choi94b} has cut non-convex closed $\rp^2$ surfaces along principal geodesics into a union of properly convex pieces and $\pi$-annuli. Choi-Goldman use this construction to classify all closed $\rp^2$ surfaces \cite{choi-goldman97}.
\end{rem}

\begin{rem}
To stay within the category of properly convex $\rp^2$ surfaces, we should avoid principal half-annuli.  This is not a serious restriction on us, as principal half-annuli are prime in the sense that they cannot be formed by gluing together two smaller properly convex $\rp^2$ surfaces along a principal geodesic boundary. Moreover, all the surfaces we consider have negative Euler characteristic.
\end{rem}

\begin{cor}
Let $M$ be a properly convex $\rp^2$ surface with principal geodesic boundary.  Assume the hyperbolic holonomies along two boundary components $B_1,B_2$ are, up to conjugation, inverses of each other.  Then we may glue $B_1$ to $B_2$ as above to make the $\rp^2$ surface $\bar M$ properly convex.
\end{cor}

\begin{proof}
Choose local $\rp^2$ coordinates near $B_1,B_2$ so that there is a lift of each to the standard principal geodesic $\tilde \ell$, so that the neighborhoods of $B_1,B_2$ are respectively on opposite sides of the $\tilde\ell$, and so that the holonomies around $B_1,B_2$ are diagonal. Since they are both in canonical form, they must actually be inverses of each other, say $H$ and $H^{-1}$.  Now for a point $p$ in $T$ close enough to the interior of $\tilde \ell$, we may form a neighborhood of $\ell$ by moving $p$ by the one-parameter group $H^t$ corresponding to the holonomy.  The region between $\{H^tp\}$ and $\ell$ is then a principal collar neighborhood $\mathcal N^a$.  But now we can do the same thing on the other side of $\ell$ to find a principal collar neighborhood $\mathcal N^b$.  Since their holonomies are inverses of each other, we see that, after possibly shrinking the collar neighborhoods, $\mathcal N^a = J \mathcal N^b $ with the holonomy equivariant under the action of $J$. This means that $J$ descends to the quotient, and the hypotheses in Theorem \ref{glue-rp2} are satisfied.  Note $J$ commutes with the holonomy matrix.
\end{proof}

Goldman's construction in  Theorem \ref{glue-rp2} involves a choice of a orientation-reversing projective map $J$ to glue the collar neighborhoods across the principal geodesic boundary components. If standard coordinates are chosen on $\re^3$ as above, then we may choose $J=D(1,-1,1)$.  Note $J$ commutes with the holonomy $H$.  But there are other possible choices determined by generalized twist parameters $\sigma,\tau$. For $M_{\sigma,\tau} = D(e^{-\sigma-\tau},e^{2\tau},e^{\sigma-\tau})$, consider the projective involution $J_{\sigma,\tau}=M_{\sigma,\tau}J$, which still commutes with $H$. $\tau$ is called the \emph{twist parameter}, as it corresponds to the usual twist parameter on a hyperbolic surface.  We call $\sigma$ the \emph{bulge parameter}.  (In \cite{loftin02c}, these are called the horizontal and vertical twist parameters respectively.) Our choice of $J$ is not canonical, as it depends on a choice of coordinates; the twist and bulge parameters then are relative to this choice of $J$.  The results and proofs in this paper do not depend on the choice of $J$, as we are primarily interested in the cases when the bulge parameters go to $\pm\infty$.

On a neck with hyperbolic holonomy, with coordinates on $\re^3$ so that the holonomy is given by the diagonal matrix $H=D(\lambda,\mu,\nu)$ with $\lambda>\mu>\nu>0$ and $\lambda\mu\nu=1$.  A \emph{Dehn twist} is a generalized twist which corresponds exactly to the holonomy along the geodesic loop.  The Dehn twist is transverse to the family of bulge parameters, but is not typically strictly a twist parameter as defined above.

\subsection{General topology of orbit spaces}

\begin{lem} \label{first-countable-quotient}
Let $X$ be a first countable topological space, and let $\Phi$ be a set of homeomorphisms acting on $X$. Then the quotient space $\Phi\backslash X$ is first countable with respect to the quotient topology.
\end{lem}

\begin{proof}
Let $x\in X$ and denote the projection to the quotient by $f$.  Since $X$ is first countable, we may choose $V_i$ a countable collection of neighborhoods of $x$ for which any open set $U$ containing $x$ satisfies $U\supset V_i$ for some $i$.

Let $\mathcal O$ be an open neighborhood of $f(x)$ in the quotient space.  The inverse image $f^{-1}(\mathcal O)$ is open in $X$, and so there is a $V_i$ so that $f^{-1}(\mathcal O) \supset V_i$.  Now we claim that $f(V_i)$ is open and that $\mathcal O\supset f(V_i)$. The openness of $f(V_i)$ is equivalent to the openness of $f^{-1}(f(V_i))$, which is the union of orbits $\Phi V_i = \bigcup_{\phi\in\Phi}\phi V_i$. The openness of $f^{-1}(V_i)$ follows since each $\phi$ is a homeomorphism.  Finally, $\mathcal O \supset f(V_i)$ is equivalent to $f^{-1}(\mathcal O) \supset f^{-1}(f(V_i)) = \Phi V_i$. This follows since $f^{-1}(\mathcal O) = \Phi f^{-1}(\mathcal O) \supset \Phi V_i$.  Thus the open sets $f(V_i)$ are the required open neighborhoods around $f(x)$.
\end{proof}

\begin{lem}\label{orbit-sequence}
Let $X$ be a first countable topological space, and let $\Phi$ be a set of homeomorphisms acting on $X$, and let $f$ be the projection to the quotient space. Then $y_i\to y$ in $\Phi\backslash X$ if and only if there is a sequence $x_i\to x$ in $X$ with $y_i=f(x_i)$ and $y=f(x)$.
\end{lem}

\begin{proof}
Let $y_i\to y$ in $\Phi\backslash X$. Let $x\in f^{-1}(y)$. Let $V_n$ be a neighborhood basis for $x\in X$. The proof of the previous lemma implies that $f(V_n)$ is a neighborhood basis for $y\in \Phi\backslash X$. Then $y_i\to y$ implies that for all $n$, there is an $I$ so that if $i\ge I$, then $y_i\in f(V_n)$. This is equivalent to $f^{-1}(y_i) \subset f^{-1}(f(V_n)) = \Phi(V_n)$, which in turn means that for all $x_i\in y_i$ there is a $\phi_i\in\Phi$ so that $x_i\in\phi_i V_n$. In turn, $\phi_i^{-1} x_i \in V_n$ and so $\phi_i^{-1}x_i\to x$ in $X$.
\end{proof}

\subsection{Markings on convex $\rp^2$ surfaces}\label{mark-subsec}

In this subsection, we consider a connected oriented surface $S$ which admits a complete hyperbolic metric. A marked convex $\rp^2$ structure on $S$ is given by the quotient $\{(\Omega,\Gamma)\}/\sim$, where $\Omega$ is a properly convex domain in $\rp^2$ and, for a basepoint $x_0\in S$,
$$\Gamma\!:\pi_1(S,x_0)\to \Sl3$$
is a discrete embedding which acts on $\Omega$ so that $\Gamma\backslash\Omega$ is diffeomorphic to $S$. The equivalence relation $\sim$ is given by $(\Omega,\Gamma)\sim (A\Omega,A\Gamma A^{-1})$ for $A\in\Sl3$.  Note that $\Omega$ is the image of a developing map for this $\rp^2$ structure on $S$, and $\Gamma$ is the corresponding holonomy representation.

\begin{lem} \label{indep-basepoint}
A change of the basepoint does not change the marked convex $\rp^2$ structure $\{(\Omega,\Gamma)\}/\sim$.
\end{lem}
\begin{proof}
If $x_0,x\in S$, then each homotopy class of paths from $x_0$ to $x$ induces an isomorphism of $\pi_1(S,x_0)$ and $\pi_1(S,x)$. If $P$ is a path from $x_0$ to $x$, consider a lift of $x_0$ to $\Omega$, and consider the developing map along a lift of $P$.  In this case, the image $\Omega$ of the developing map $\Omega$ is unchanged, and the induced isomorphism along $P$ identifies $\Gamma_{x_0}$ and $\Gamma_{x,P}$. The images of $\Gamma_{x_0}$ and $\Gamma_{x,P}$ are still the same, and thus the quotient $\Gamma\backslash \Omega$ is unchanged.

Our definition is potentially ambiguous if we consider two non-homotopic paths $P$ and $Q$ from $x_0$ to $x$. In this case, $\Gamma_{x,P}$ and $\Gamma_{x,Q}$ are related by conjugating by the element $\Gamma_x(P^{-1}\cdot Q)\in\Sl3$. Since $\Gamma_x(P^{-1}\cdot Q)$ acts on $\Omega$, $(\Omega,\Gamma_{x,P})\sim (\Omega,\Gamma_{x,Q})$.
\end{proof}

Define the \emph{Goldman space} of $S$ by $\mathcal G_S = \{(\Omega,\Gamma)\}/\sim$ with the following topology.  For convex domains in $\rp^2$, consider the Hausdorff distance with respect to the Fubini-Study metric in $\rp^2$.  For the representation $\Gamma$, use the product topology of one copy of $\Sl3$ for each element of $\Gamma(\gamma)$ for $\gamma\in\pi_1(S,x_0)$ (note we consider only surfaces of finite type, for which $\pi_1(S,x_0)$ is finitely generated). Since $\Gamma$ is countable, this topology is first countable.
Now the equivalence relation $\sim$ represents the orbits of a group action of $\Sl3$, which acts by homeomorphisms on the space of all $(\Omega,\Gamma)$.   Then Lemma \ref{first-countable-quotient} shows the quotient topology on $\mathcal G_S$ is also first countable.

In the case that $S$ is a closed surface of genus $g\ge2$, we may define Goldman space $\mathcal G_g =\mathcal G_S$. This deformation space is the analog of Teichm\"uller space for convex $\rp^2$ structures on $S$.  Goldman \cite{goldman90a} proved that $\mathcal G_g$ is homeomorphic to $\re^{16g-16}$.  For augmented Goldman space, which we define below, we will need the case of noncompact $S$ as described above.

It will be useful for us to allow the case in which $S=\sqcup_{i=1}^n S_i$ has finitely many connected components.  In this case, define $\mathcal G_S$ as the product of the $\mathcal G_{S_i}$ with the product topology.

\begin{rem}
The topology we consider is related to the Chabauty topology consider by Harvey \cite{harvey77}. See also Wolpert \cite{wolpert08}.   Harvey's work is concerned solely with limits of Fuchsian groups under the Chabauty topology.  In particular, the image of the developing map for a Fuchsian group is always the hyperbolic plane, while our analogous domains $\Omega_j$ can and do vary.  Moreover, for noncompact convex $\rp^2$ surfaces which natural appear as limits in our case (regular convex $\rp^2$ surfaces), the holonomy representations $\Gamma_j$ do not determine the geometry.  Distinct pairs of convex $\rp^2$ surfaces with a isomorphic holonomy  naturally arise: consider a surface with an end, put hyperbolic holonomy on the end, and vary the bulge parameter from $-\infty$ to $+\infty$.
(For compact convex $\rp^2$ surfaces, a rigidity theorem for the holonomy spectrum holds \cite{cooper-delp10,inkang-kim01,inkang-kim10}.)  Our definition is in a sense a little less general than Harvey's, as we specify the loops along which the degeneration occurs.

Our topology is also analogous to the geometric topology on hyperbolic manifolds (see e.g.\ \cite{benedetti-petronio92}), in which sequences of pairs of points whose hyperbolic distance diverges to infinity cannot reside in the same geometric limit space.  Although our definitions are phrased differently, we do see below in Proposition \ref{separate-O} that a similar property holds with respect to the projectively-invariant Blaschke metrics on $\Omega$.
\end{rem}

\subsection{Separated necks and the pulling map} \label{sep-neck-subsection}
Let $S$ be a connected oriented surface of finite hyperbolic type.
Define $C(S)$ to be the set whose elements consist of sets of free homotopy classes of simple loops on $S$ so that each loop is nonperipheral and no two loops intersect ($C(S)$ may be identified with the set of simplices of the complex of curves on $S$). Let $c\in C(S)$.  Denote the connected components of $S- c$ by $S_1,\dots,S_n$.  Note that each surface $S_i$ admits a finite-area hyperbolic metric.  (In the notation below, we will not be careful to distinguish between $c\in C(S)$ as a collection of nonintersecting loops in $S$ as opposed to a collection of homotopy classes.)

If  $c\in C(S)$, we define the \emph{pulling map}
$$ {\rm Pull}_c\!: \mathcal G_S \to \mathcal G_{S-c}$$
as follows.  Let $S-c = \sqcup_{i=1}^n S_i$.  Then for $X\in \mathcal G_S$, take a representative $(\Omega,\Gamma)$ in the equivalence relation for $X$.  Then represent ${\rm Pull}_c(X)$ by the ordered $n$-tuple with $i^{\rm th}$ element represented by $(\Omega,\Gamma^i)$, where $\Gamma^i=\Gamma|_{S_i}$ is a sub-representation of $\Gamma$ corresponding to $\pi_1(S_i,x_i)$ for a basepoint $x_i$. (Recall Lemma \ref{indep-basepoint} shows the marked convex $\rp^2$ structure is unchanged if the choice of basepoint changes.) To be precise, for the subsurface $S_i\subset S$, the fundamental group of $S_i$ is naturally a conjugacy class of subgroups of $\pi_1(S,x_i)$. We choose $\Gamma^i$ to be the composition of the injection $\pi_1(S_i,x_i)\to \pi_1(S,x_0)$ with $\Gamma$. The element of $\mathcal G_S$ is independent of the conjugacy class though: Let $\eta=\Gamma(\gamma)$ for $\gamma\in\pi_1(S,x_i)$, and consider $(\Omega,\Gamma^i)\sim\eta(\Omega,\Gamma^i) = (\Omega,\eta\Gamma^i\eta^{-1})$.  Thus choosing a particular sub-representation in the conjugacy class to identify as $\Gamma^i$ is harmless.

Note for each $\rp^2$ surface $X_i=\Gamma^i\backslash \Omega$, the developing map still has image equal to all of $\Omega$, while the group of deck transformations $\Gamma^i(\pi_1S_i)$ is smaller than $\Gamma(\pi_1S)$.  In terms of Goldman's attaching map, the domain attached across the principal geodesic still remains attached under the pulling map.  This map is called \emph{pulling} in part because it not simply \emph{cutting} along the principal geodesic.  Instead, one can imagine a viscous liquid being pulled apart, and the material on either side of the principal geodesic remains attached to the other side after the neck is pulled.

Now consider a collection $\{S_i,e_{i1},\dots,e_{ij_i}\}$, where each $S_i$ is an oriented surface with genus $g_i$ and $j_i$ ends with $2g_i+j_i\ge3$ (so that $S_i$ admits a complete hyperbolic metric). Assume the number of ends is even, and partition the set of ends $\{e_{11},\dots, e_{nj_n}\}$ into pairs.  We will study the situations in which these pairs of ends are in the closure of the image of the pulling map.
A pairs of ends in $\{e_{ij}, e_{k\ell}\}$ is said to be a \emph{trivial separated neck} if it is (locally) isomorphic as an $\rp^2$ surface to a neck which is the image of the pulling map across a loop $\ell$.  A pair of ends forms a \emph{regular separated neck}  if it is one of these three cases:
\begin{itemize}
\item The holonomy around each end is parabolic.
\item The holonomy around each end is quasi-hyperbolic, and the oriented holonomies around each end are, up to conjugation in $\Sl3$, inverses of each other.
\item The holonomy around each end is hyperbolic; the oriented holonomies around each end are, up to conjugation, inverses of each other; and the $\rp^2$ structure about one of the two ends has bulge parameter $+\infty$, while the other end has bulge parameter $-\infty$.
\end{itemize}
A simple end of a convex $\rp^2$ surface is \emph{regular} if it forms half of a regular separated neck.

\begin{thm} \label{geom-limits-regular}
Let $S$ be a compact surface, and let $c\in C(S)$.  Then under the topology defined above, the closure of the image ${\rm Pull}_c(\mathcal G_S)$ consists of convex $\rp^2$ structures for which the neck across each loop in $c$ is either a regular or a trivial separated neck.
\end{thm}

\begin{proof}
Since $S$ is compact, the holonomy around each loop  $\ell\in c$ is hyperbolic.

Recall that a hyperbolic element in $\Sl3$ has three distinct positive eigenvalues.  Any nonhyperbolic limit $A$ of such holonomies must still have all positive eigenvalues.  Moreover, it must have maximal Jordan blocks (the other cases are ruled out by Choi \cite{choi94b} and the author \cite{loftin04}; see also \cite{nie15}).  These nonhyperbolic limits are exactly the quasi-hyperbolic and parabolic cases, which are regular.  For the quasi-hyperbolic case, the inverse property follows from the fact that the holonomy around the two ends of a neck are inverses of each other (since these loops are freely homotopic in $S$ with opposite orientations).

Now we consider the cases of hyperbolic limits, and show that any limits which are not trivial must have infinite bulge parameter.  In order to do this, consider a sequence $X_i\in \mathcal G_S$.

There are two cases to consider.  First of all, assume that $S-\ell = S_1\sqcup S_2$ is disconnected.  Then the hypothesis shows that there are sequences $(\Omega_k,\Gamma_k)$ and $\rho_k,\sigma_k\in \Sl3$ so that $\rho_k(\Omega_k,\Gamma_k|_{S_1}) \to (\mathcal O, H)$ and $\sigma_k(\Omega_k,\Gamma_k|_{S_2})\to (\mathcal U, G)$. The quotient $H\backslash \mathcal O$ gives the $\rp^2$ structure on $S_1$, while $G\backslash \mathcal U$ gives the structure on $S_2$.  Now pick a based loop $\ell_0$ in $S$ which is freely homotopic to $\ell$, and let $\gamma_k$ be the corresponding element  $\Gamma_k(\ell_0)$. Assume $\ell$ is oriented in the same direction as the boundary of $S_1$.  This implies $\ell$ is oriented in the opposite direction to the boundary of $S_2$.  Let $\gamma_{\mathcal O}$ and $\gamma_{\mathcal U}$ be the limits of  $\rho_k\gamma_k \rho_k^{-1}$ and $\sigma_k\gamma_k\sigma_k^{-1}$ respectively.  We may choose coordinates (and modify $\rho_k$ and $\sigma_k$) so that the limiting hyperbolic holonomies around $\ell$ satisfy $$\gamma_{\mathcal O} = D(\lambda,\mu,\nu), \quad \gamma_{\mathcal U} = \gamma_{\mathcal O}^{-1},\quad \lambda>\mu>\nu>0,$$
where $D$ represents the diagonal matrix.
In other words, for $\gamma_{\mathcal O}$, the principal geodesic is the line segment from $[1,0,0]$ to $[0,0,1]$ with nonnegative entries.  Denote this principal geodesic by $\tilde\ell$.

We also can make a further normalization to assume that
\begin{equation}\label{conjug-diagonal}
 \sigma_k \gamma_k \sigma_k^{-1} = \rho_k \gamma_k \rho_k^{-1} = D(\lambda_k,\mu_k,\nu_k).
\end{equation}
Proof: The eigenvalues $\lambda_k,\mu_k,\nu_k$ of $\sigma_k\gamma_k\sigma_k^{-1}$ approach $\lambda,\mu,\nu$. For $k$ large enough, $\lambda_k,\mu_k,\nu_k$ are uniformly bounded, positive and separated from each other.  We may choose a matrix $\phi_k$ of eigenvectors of $\sigma_k \gamma_k \sigma_k^{-1}$ which approaches the identity matrix as $k\to\infty$ (for example, we may choose eigenvectors of unit length; note the identity matrix is a matrix of unit eigenvectors of the limit $\gamma_{\mathcal O}$ of $\sigma_k\gamma_k\sigma_k^{-1}$).  Then $(\sigma_k\Omega_k,\sigma_k\Gamma_k\sigma_k^{-1}) \to (\mathcal O,H)$ if and only if $\phi_k^{-1}(\sigma_k\Omega_k,\sigma_k\Gamma_k\sigma_k^{-1}) \to (\mathcal O,H)$.  Note our construction implies
$$\phi_k^{-1}\sigma_k\gamma_k\sigma_k^{-1}\phi_k = D(\lambda_k,\mu_k,\nu_k).$$
Thus we may replace $\sigma_k$ by $\phi_k^{-1}\sigma_k$, and we may assume (\ref{conjug-diagonal}).

Now Equation (\ref{conjug-diagonal}) implies the diagonal matrix $D(\lambda_k,\mu_k,\nu_k)$ commutes with $\sigma_k\rho_k^{-1}$, and so $\sigma_k\rho_k^{-1}$ is diagonal as well. Define $\alpha_k=\sigma_k \rho_k^{-1}$.  Thus we write
$\alpha_k = D(\lambda_k^{t_k},\mu_k^{t_k},\nu_k^{t_k})\cdot D(e^{-s_k},e^{2s_k},e^{-s_k})$ as a product of holonomy and bulge matrices. Since the neck is being separated, Dehn twists do not affect the geometry of $S-\ell$.  Now if $\alpha_k=\sigma_k\rho_k^{-1}$ has a subsequential finite limit $\alpha$ modulo Dehn twists, we have $\mathcal O = \lim_{j\to\infty} \alpha_{k_j} \rho_{k_j}(\Omega_{k_j})= \alpha \mathcal U$. Moreover, if we let $\hat\alpha_k$ be the matrix  given by the product of $\alpha_k$ by an integral power of a Dehn twist so that $$\hat\alpha_k = D(\lambda_k^{\hat t_k},\mu_k^{\hat t_k},\nu_k^{\hat t_k})\cdot D(e^{-s_k},e^{2s_k},e^{-s_k})$$ for $\hat t_k\in[0,1)$, then $\sigma_{k_j}\Gamma_{k_j}\sigma_{k_j}^{-1} = \hat\alpha_{k_j}\rho_k\Gamma_{k_j}\rho_k^{-1}\hat\alpha_{k_j}^{-1}$ converges to a limit $L$.  Then we apply Goldman's Theorem \ref{glue-rp2} above to show  the neck is trivial.  Thus we may assume the $s_k$ converge to $\pm\infty$.

(Here is how to apply Theorem \ref{glue-rp2} to the present situation.  Consider $\hat S_1$ as the convex $\rp^2$ surface homeomorphic to $S_1$ formed by cutting along the principal geodesic at the end $\ell$.  This can be constructed by letting $\hat{\mathcal O}$ be the convex domain formed by cutting along $h\tilde \ell$ for all $h\in H(\pi_1S_1)$. Then $\hat S_1$ is the quotient $H\backslash \hat{\mathcal O}$.  We can similarly form $\hat S_2$ with image $\hat{\mathcal U}$ of the developing map.  Then since $\alpha^{-1}\mathcal U =\alpha \mathcal O$, the domains $\hat {\mathcal O}$ and $\alpha^{-1}\hat{\mathcal U}$ are disjoint subsets of $\mathcal O$ with common boundary segment $\tilde\ell$.  Since the group actions also match up on all of $\mathcal O$, there are principal collar neighborhoods in $\hat{\mathcal O}$ and $\alpha^{-1}\hat{\mathcal U}$ which are invariant under holonomy along $\tilde\ell$ and are projectively equivalent via the principal reflection across $\tilde \ell$.  Thus Theorem \ref{glue-rp2} applies, and we may glue $\hat S_1$ and $\hat S_2$ together via this identification to form a convex $\rp^2$ surface.  This surface must be identical to the quotient $L\backslash \mathcal O$ by analytic continuation, and therefore uniqueness, of the developing map.)

Assume without loss of generality that $s_k \to +\infty$. To show the bulge parameter must be infinite in the limit, we recall the principal triangles with principal geodesic $\tilde\ell$. Let $T$ be the open triangle in $\rp^2$ given by the projection of the first octant in $\re^3$, and let $\bar T$ be the reflection of $T$ across the principal geodesic.  Since the surface $S$ is separated along this principal geodesic, we may assume that the universal covers $\mathcal O$ and $\mathcal U$ of $S_1$ and $S_2$ respectively are in part on opposite sides of $\tilde\ell$. Without loss of generality, assume that $\mathcal O\cap \bar T\neq\emptyset$ and $\mathcal U\cap T\neq\emptyset$.

Let $q\in\mathcal U \cap T$. Then for large $k$, $q\in \rho_k\Omega_k$ and so $\alpha_k q \in \sigma_k\Omega_k$.  This shows that the limit of $\alpha_kq$ is in the closure of $\mathcal O$.  But since we know $\alpha_k$ has bulge parameter $s_k$ going to $+\infty$, $\alpha_kq\to [0,1,0]$ from within $T$. Since the principal geodesic $\tilde\ell\subset\bar{\mathcal O}$, we see by convexity that $T\subset \mathcal O$, and thus that the bulge parameter of this end of $S_1$ is $+\infty$.

On the other hand, the same argument shows that if there is a $p\in\mathcal U \cap \bar T$, then $\bar T\subset \mathcal O$. This is impossible, as then $\mathcal O \supset \bar T \cup \tilde\ell\cup T$, which contains the coordinate line with infinite point $[0,1,0]$.  This contradicts the proper convexity of $\mathcal O$. Thus $\mathcal U \cap \bar T = \emptyset$, which means the bulge parameter of this end of $S_2$ is $-\infty$.  The limit then satisfies the condition for a separated neck with hyperbolic holonomy to be regular.

For the second case, assume $S-\ell = S_1$ is connected.   In this case, we have a sequence of marked $\rp^2$ structures $(\Omega_k,\Gamma_k)$ and distinguished hyperbolic elements $\gamma_k,\delta_k\in\Gamma_k(\pi_1S)$, together with attaching maps $T_k\in {\rm Aut}\,\Omega_k$ so that $T_k\gamma_k T_k^{-1} = \delta_k^{-1}$ (see for example Harvey \cite{harvey77}).  The hyperbolic elements $\gamma_k$ and $\delta_k$ both represent the holonomy (with opposite orientations) of the neck to be separated. We assume that $(\Omega_k,\Gamma_k|_{S_1}) \to (\mathcal O,H)$.  (In this case, since there is a single limit domain, we absorb the $\rho_k\in\Sl3$ into the definitions of $\Omega_k$ and $\Gamma_k$.  In the present case, $T_k$ will diverge instead of $\rho_k$.) See Figure \ref{attach-nonsep-figure}.

\begin{figure}
\begin{center}
\scalebox{.3}{\includegraphics{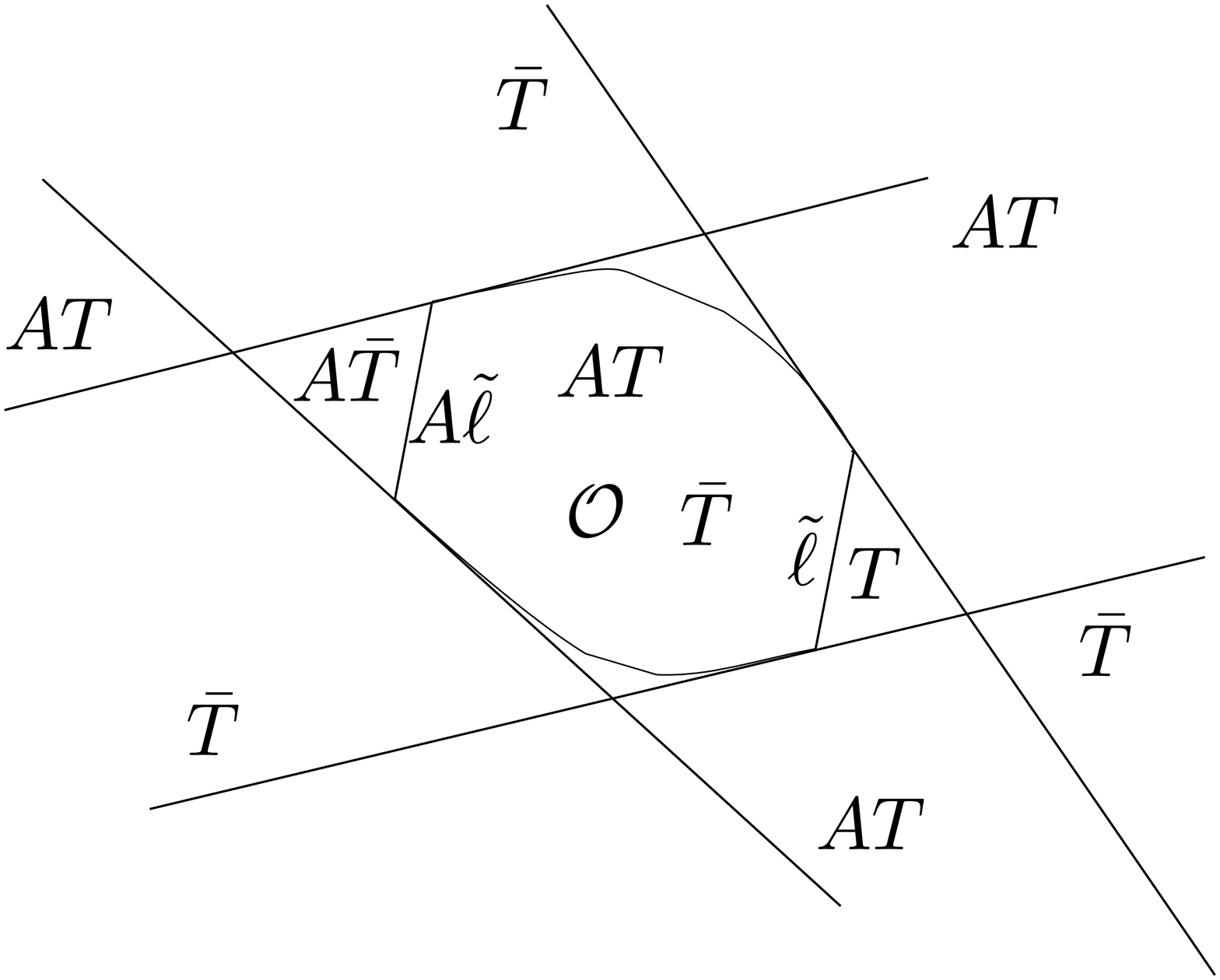}}
\end{center}
\caption{}
\label{attach-nonsep-figure}
\end{figure}

If $\gamma=\lim\gamma_k$ is parabolic or quasi-hyperbolic, then the holonomy is regular, and the holonomy type of $\delta_k$ is the inverse of that of $\gamma_k$.

If on the other hand, $\gamma$ is hyperbolic, we proceed as above.  Choose coordinates so that $\gamma = D(\lambda,\mu,\nu)$.  By the same arguments as above, we may slightly modify the $\Omega_k$ for $k$ large so that $\gamma_k = D(\lambda_k,\mu_k,\nu_k)$ as well. Let $\delta=\lim\delta_k$ and fix $A\in\Sl3$ so that $A\gamma A^{-1} = \delta^{-1}$.  Since $\delta_k\to\delta$, we find
$$ A^{-1}T_k\cdot D(\lambda_k,\mu_k,\nu_k) \cdot T_k^{-1} A \to D(\lambda,\mu,\nu).
$$
As above, there is a matrix $\phi_k$ of eigenvectors of $A^{-1}T_k \cdot D(\lambda_k,\mu_k,\nu_k) \cdot  T_k^{-1} A$ so that $\phi_k\to I$ and $$ \phi_k^{-1} A^{-1}T_k \cdot D(\lambda_k,\mu_k,\nu_k) \cdot  T_k^{-1} A \phi_k =  D(\lambda_k,\mu_k,\nu_k) .
$$
Thus, as above, $T_k^{-1}A\phi_k$ is diagonal, and may be written as the product of Dehn twist and bulge matrices $D(\lambda_k^{t_k},\mu_k^{t_k},\nu_k^{t_k}) \cdot D(e^{-s_k}, e^{2s_k}, e^{-s_k})$.  If $s_k$ has a finite limit, then we can show, as in the case above in which $S-\ell$ is not connected, that the separated neck is trivial. (Theorem \ref{glue-rp2} applies in a similar way.) The remaining cases to analyze are $s_k\to\pm\infty$.  Assume without loss of generality that $s_k\to+\infty$.
Recall the definitions of the principal geodesic $\tilde\ell$ and principal triangles $T,\bar T$ as above. Let $\tilde\ell$ be the line segment from $[1,0,0]$ to $[0,0,1]$ with nonnegative coordinate entries, and let $T$ be the open triangle with vertices $[1,0,0]$, $[0,1,0]$, and $[0,0,1]$ all of whose coordinates are nonnegative.  Finally, let $\bar T$ be the reflection of $T$ across $\tilde\ell$.  By the limiting attaching map $A$, we may assume that $\mathcal O \cap \bar T$ and $\mathcal O \cap AT$ are not empty.  For $q\in \mathcal O \cap \bar T$, we see that for large $k$, $q\in \Omega_k\cap \bar T$. Then $\phi_k A^{-1} T_k q\to [0,1,0]$ from within $\bar T$, which shows $T_k q \to A[0,1,0]$ from within $A\bar T$. Since $\bar{\mathcal O}\supset A\tilde\ell$, we see by convexity that $\mathcal O \supset A\bar T$, and so the bulge parameter of this end is $+\infty$.

On the other hand, if there is a $p\in \mathcal O\cap T$, then the same argument shows $\mathcal O\supset
AT$ as well.  This contradicts the proper convexity of $\mathcal O$, and so we see that $\mathcal O \cap T = \emptyset$, and thus the bulge parameter of this end is $-\infty$. This picture satisfies the definition of a regular separated neck.
\end{proof}

\subsection{Augmented Goldman space}

In order to introduce the augmented Goldman space, as a warmup, we first form a bordification of $\mathcal G_S$ by attaching singular $\rp^2$ structures which degenerate only along a single simple closed nonperipheral curve $\ell$.
Define $\mathcal G^\ell_S$ as the set of all properly convex $\rp^2$ structures on $S-\ell$ which form regular separated necks across $\ell$. We produce a topology on the bordification $$\mathcal G_S \sqcup \mathcal G_S^\ell.$$
First of all, if $X\in \mathcal G_S$, then we declare all open neighborhoods in $\mathcal G_S$ to form a neighborhood basis in the bordification.

Now let $X\in\mathcal G^\ell_S$.
First of all, consider open sets among $\rp^2$ structures on $S-\ell$ which form separated necks across $\ell$. Each such open set $\mathcal O\subset \mathcal G_{S-\ell}$ contains both regular and trivial necks across $\ell$.  Now we construct from $\mathcal O$ a subset $\tilde{\mathcal O}$ of $\mathcal G_S\sqcup \mathcal G_S^\ell$.  Let $\mathcal O_{\rm reg}$ consist of the $\rp^2$ structures with regular necks across $\ell$, and let $\mathcal O_{\rm triv}= \mathcal O - \mathcal O_{\rm reg}$ consist of the $\rp^2$ structures with trivial necks across $\ell$. Then define
$$\tilde{\mathcal O} = \mathcal O_{\rm reg} \sqcup {\rm Pull}_{S,\ell}^{-1}\mathcal O_{\rm triv}.$$
In other words, for each $\rp^2$ structure with a trivial neck in a neighborhood of $X$, we attach the neck by taking the inverse image of the pulling map.  All such $\tilde{\mathcal O}$ form a neighborhood basis for the topology of augmented Goldman space near $X$.   Note that this topology on $\mathcal G_S \sqcup \mathcal G^\ell_S$ is not locally compact, since the pulling map is unchanged under each Dehn twist around $\ell$.  But, by the arguments above in Subsection \ref{mark-subsec}, we may choose a countable collection of such neighborhoods, and so the topology is first countable.

If $c\in C(S)$, define $\mathcal G^c_S$ to be the set of all properly convex $\rp^2$ structures on $S-c$ with regular necks across each loop in $c$.  (If $c=\emptyset$, $\mathcal G^\emptyset_S=\mathcal G_S$.) As a set, augmented Goldman space
$$\mathcal G^{\rm aug}_S = \bigsqcup_{c\in C(S)} \mathcal G^c_S.$$
If $X\in \mathcal G^{\rm aug}_S$, then there is a unique $c\in C(S)$ so that $X\in\mathcal G^c_S$, and thus $X$ has a regular separated neck across each loop in $c$. As we deform $X$, some of these necks may remain separated, while others may be glued together.  As above, let $\mathcal O$ be a neighborhood of $X$ in the subset of $\mathcal G_{S- c}$ consisting of those $\rp^2$ structures which have separated necks across each loop in $c$.  Each $Y\in\mathcal O$ has either trivial or regular separated necks across each loop in $c$. Then $$\mathcal O = \bigsqcup_{d\subset c} \mathcal O_{{\rm triv},d},$$ where $Y\in\mathcal O_{{\rm triv},d}$ if and only $d$ is the set of loops across which the separated neck is trivial in $Y$ (thus the necks across the loops in $c- d$ are the regular separated necks).  Now define $$\mathcal {\tilde O} = \bigsqcup_{d\subset c} {\rm Pull}^{-1}_d \mathcal O_{{\rm triv},d},$$ where ${\rm Pull}_\emptyset$ is the identity map.  The set of such $\tilde{\mathcal O}$ forms a neighborhood basis for the topology of $\mathcal G^{\rm aug}_S$ around $X$.

\begin{rem}
It is instructive to compare the construction of $\mathcal G^{\rm aug}_S$ to the construction of augmented Teichm\"uller space; see e.g. \cite{abikoff80}.  Given a free simple loop $c$ in a surface closed $S$ of genus at least two, we may take the hyperbolic length parameter around $c$ to be zero.  Then no neighborhood of this point in the augmented Teichm\"uller space has compact closure, as the associated twist parameters around $c$  take all real values in the neighborhood.

As we must keep track of the developing map of a surface pulled across a loop $c$, each point in $\mathcal G_{S-c}$ \emph{a priori} has a neighborhood in $\mathcal G^{\rm aug}_S$ which contains all integral powers of Dehn twists along $c$.  This shows $\mathcal G^{\rm aug}_S$ is not locally compact.  We can say more, however. For the regular cases, which are of primary interest,  one may check that each neighborhood of an $\rp^2$ structure all of whose separated necks are regular includes $\rp^2$ structures on the glued necks which represent all real values of the twist parameters.

We have defined augmented Goldman space essentially in terms of the dev-hol pair $(\Omega,\Gamma)$ of convex $\rp^2$ structures, as opposed to the Fenchel-Nielsen parameters commonly used in study of augmented Teichm\"uller space.  It should be interesting to try to use Goldman's analog of Fenchel-Nielsen parameters on convex $\rp^2$ structures \cite{goldman90a} to put coordinates on augmented Goldman space.  Goldman's parameters have been extended by Marquis to the cases of parabolic and quasi-hyperbolic holonomy \cite{marquis10}.
\end{rem}

\subsection{Augmented moduli space}

Our main space of interest is in the quotient of augmented Goldman space by the mapping class group, which we call the augmented moduli space of convex $\rp^2$ structures.  Recall the mapping class group is the group of orientation-preserving homeomorphism modulo diffeomorphisms isotopic to the identity $MCG(S)={\rm Diff}^+(S)/ {\rm Diff}^0(S)$.

Consider a diffeomorphism $\phi$ of $S$.  If $x_0\in S$ is a basepoint, then $\phi$ induces a map $\phi_*\!: \pi_1(S,x_0) \to \pi_1(S,\phi(x_0))$.  We fix a holonomy representation $\Gamma\!: \pi_1(S,x_0)\to\Sl3$.  We assume that for the image of the developing map $\Omega$, that the quotient $\Gamma\backslash \Omega$ is diffeomorphic to $S$.

\begin{prop}
$MCG(S)$ acts on $\mathcal G_S$ by homeomorphisms.
\end{prop}

\begin{proof}
Let $[(\Omega,\Gamma)]\in \mathcal G_S$, where $[\cdot]$ denotes the equivalence class under the $\Sl3$ action.
In order to consider the action of the diffeomorphism $\phi$ on $\Gamma$, each homotopy class of paths in $S$ from $x_0$ to $\phi(x_0)$ determines an isomorphism from $\pi_1(S,x_0)$ to $\pi_1(S,\phi(x_0))$.  As in Lemma \ref{indep-basepoint} above, different choices of paths lead to representations equivalent under the $\Sl3$ action.  This shows that the action of $\phi$ on $(\Omega,\Gamma)$ produces an equivalence class in $\mathcal G_S$.  Moreover, the actions of the diffeomorphism group and $\Sl3$ commute with each other, and so  the diffeomorphism group acts on $\mathcal G_S$.

All that remains is to show the diffeomorphisms isotopic to the identity act trivially.  The argument in the previous paragraph shows that we may assume such a diffeomorphism preserves the basepoint $x_0$.  In this case, an isotopy of diffeomorphisms induces a homotopy of loops based at $x_0$, and so the elements of $\pi_1(S,x_0)$ are fixed by diffeomorphisms isotopic to the identity.

It is clear from the definition of the topology on $\mathcal G_S$ that this action is by homeomorphisms.
\end{proof}

We denote the quotient $MCG(S)\backslash \mathcal G_S$ by $\mathcal R_S$.

To extend this proposition to $\mathcal G^{\rm aug}_S$, we must extend our marking to the case of separated necks.  Each $c\in C(S)$  represents a set of separated necks, and $S-c$ has a number of connected components $S_1,\dots,S_n$.  First of all, consider the case that $S_1=S-c$ is connected.  Then the action of $\Gamma(\pi_1(S,x_0))$ is restricted on $S_1$ to include only those homotopy classes of loops in $\pi_1(S_1,x_0)$, which are exactly those homotopy classes  which have representative loops which do not intersect $c$.  In other words,
$$\Gamma|_{S^1}(\pi_1(S_1,x_0))\subset \Gamma(\pi_1(S,x_0))\subset\Sl3$$
in this case.

In the case $S-c=\sqcup_{i=1}^n S_i$ is not connected, then we consider $\pi_1(S_i,x_i)$ for basepoints $x_i\in S_i$.  First of all, we may relate $\pi_1(S,x_0)$ to $\pi_1(S,x_i)$ for each $i$ by choosing a path from $x_0$ to $x_i$.  As above in Lemma \ref{indep-basepoint}, different choices of a homotopy class of each such path lead to holonomy representations conjugate by elements of $\Sl3$.  But our definition of $\mathcal G_S^c$ allows for conjugating by one element of $\Sl3$ for each connected component $S_i$ of $S-c$, and so our argument works independently of the choice of paths.  Now we restrict only to those elements in $\pi_1(S,x_i)$ which do not intersect $c$: these are exactly the elements of $\pi_1(S_i,x_i)$.

Now the following proposition follows without much difficulty from the definitions laid out above.  The assertion about $\mathcal R^{\rm aug}_S$ being first countable follows from Lemma \ref{first-countable-quotient}.

\begin{prop}
For every $c\in C(S)$, a diffeomorphism $\phi$ of $S$ induces a homeomorphism of $\mathcal G_S^{\rm aug}$ which sends each stratum $\mathcal G_S^c$ homeomorphically onto $\mathcal G_S^{\phi(c)}$. The mapping class group acts by homeomorphisms on $\mathcal G_S^{\rm aug}$, and the quotient topology on $\mathcal R^{\rm aug}_S\equiv MCG(S)\backslash \mathcal G_S^{\rm aug}$ is first countable.
\end{prop}

\subsection{Plumbing coordinates and the topology of regular cubic differentials} \label{plumbing-subsec}

In this section, we define the topology of the space of regular cubic differentials.  We start with a heuristic picture of the main construction.  Recall that on the regular part $\Sigma^{\rm reg}$ of a compact noded Riemann surface $\Sigma$, there is a unique complete conformal finite-area hyperbolic metric $k$.  Each hyperbolic surface can then naturally be decomposed into the thick part and the thin part, as Margulis's Lemma shows that there is a universal positive constant $\tilde c$ so that the set of points with injectivity radius less than $\tilde c$ is a disjoint union of annular cusp and collar neighborhoods.  The noded Riemann surface is smooth if and only if there are no cusp neighborhoods. Each cusp neighborhood is isometric to every other, and a single parameter, the length $l$ of the core geodesic is the only hyperbolic invariant for collar neighborhoods.  Each cusp or collar neighborhood is metrically rotationally invariant.  Allowing $l\to0$ changes a collar neighborhood to a pair of cusp neighborhoods, and heuristically provides a path to the boundary of the moduli space of Riemann surfaces.  In this setting, we would like to define a related conformal metric $ m$ on each Riemann surface given by replacing the hyperbolic metric on the thin part by conformal flat cylindrical metrics of circumference $2\tilde c$ (so that the resulting metric, hyperbolic on the thick part and flat on the thin part, is continuous). See Figure \ref{k-m-metrics}.  Then we may define regular cubic differentials as holomorphic cubic differentials on $\Sigma^{\rm reg}$ which are bounded in the $L^\infty_m$ norm and whose residues match up appropriately.  Convergence of families of regular cubic differentials over a sequence of noded Riemann surface $\Sigma_i$ which converge to $\Sigma_\infty$ in $\overline{\mathcal M}_g$ is then defined to be convergence in $L^\infty_{\rm loc}$ with respect to the $m_i$ coordinates.

\begin{figure}
\begin{center}
\scalebox{.3}{\includegraphics{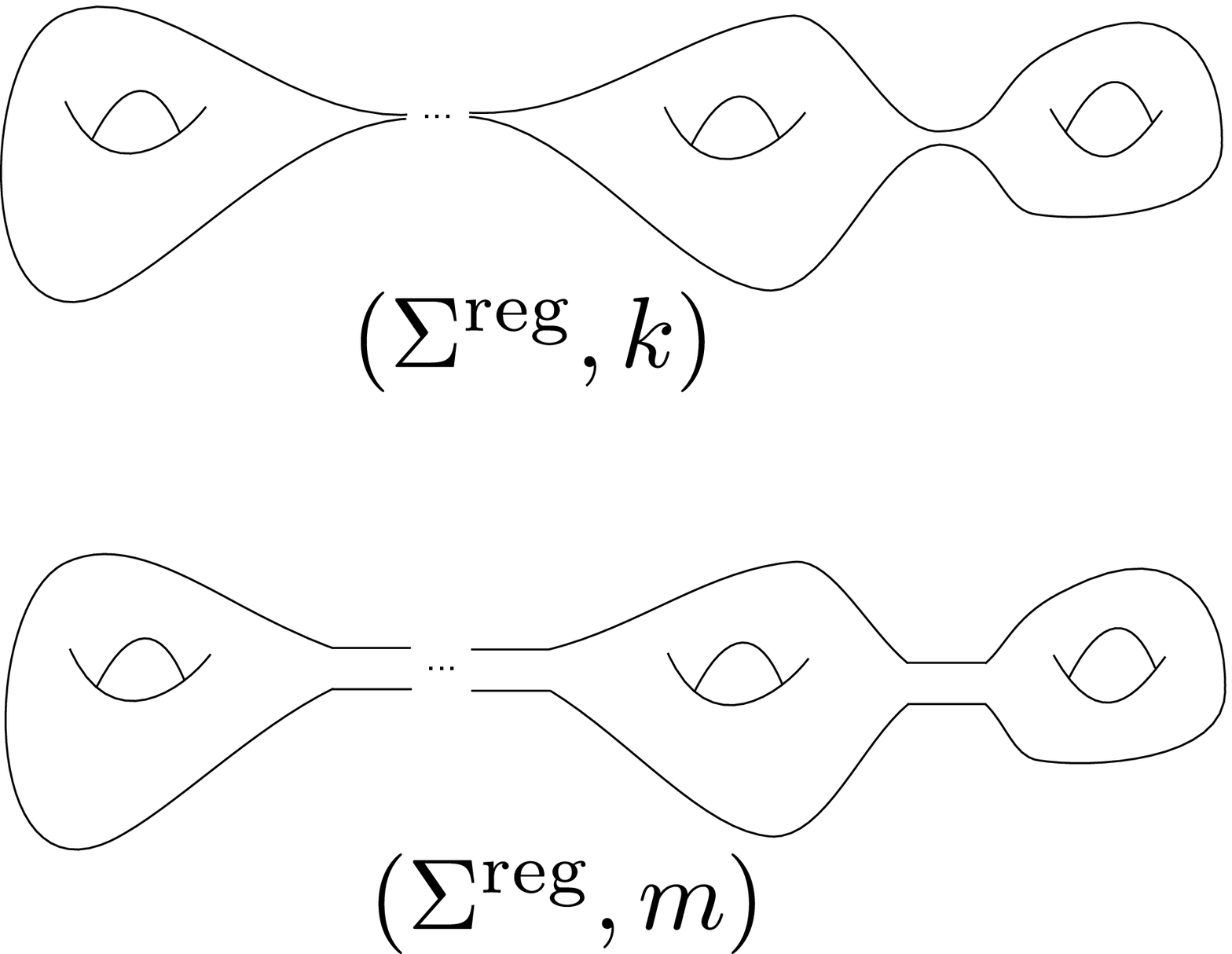}}
\end{center}
\caption{}
\label{k-m-metrics}
\end{figure}

This heuristic picture is imprecise, and using $l$ as a parameter in moduli is not well-suited to the geometry of holomorphic objects such as regular cubic differentials.  Instead, we consider Wolpert's hyperbolic metric plumbing coordinates, which describe the holomorphic moduli of noded Riemann surfaces, but also are constructed to be closely related to the hyperbolic metrics.

Consider local V-manifold cover coordinates on $\overline{\mathcal M}_g$ near a nodal curve.  These are due to Masur \cite{masur76} and refined by Wolpert.  See Wolpert \cite{wolpert10} for an overview and references.  Consider a stable noded Riemann surface $\Sigma$ with $n$ nodes.  We think of $\Sigma$ as representing a point in the boundary of the Deligne-Mumford compactification of the moduli space of closed Riemann surfaces of genus $g$.  For the $i^{\rm th}$ node there is a small \emph{cusp neighborhood} $N_i$ so that:
\begin{itemize}
\item The closures of the $N_i$ are disjoint in $\Sigma$.
\item There are coordinates $z_i,w_i$ on each part of $N_i\cap \Sigma^{\rm reg}$ and a uniform constant $c$ so that $$N_i^{\rm reg} \equiv N_i\cap\Sigma^{\rm reg} = \{|z_i|\in(0,c)\}\sqcup \{|w_i|\in(0,c)\}$$ and the complete hyperbolic metric on $\Sigma^{\rm reg}$ restricts to $N_i$ as \begin{equation} \label{hyp-met-cusp}
     \frac{|dx|^2}{(|x|\log|x|)^2}, \qquad x = w_i,z_i.
    \end{equation}
     The coordinates $z_i,w_i$ are called hyperbolic cusp coordinates.
\end{itemize}
Moreover, Wolpert \cite{wolpert92b} has constructed a real-analytic family of Beltrami differentials $\nu(s)$ on $\Sigma^{\rm reg}$ for $s$ in a neighborhood of the origin in $\co^{3g-3-n}$ so that
\begin{itemize}
\item $\nu(0)=0$.
\item The support of each $\nu(s)$ is disjoint from the closure of each cusp neighborhood $N_i$.
\item Each $\nu(s)$ is $C^\infty$, and the family of all $\nu(s)$ vary in a complex vector space of dimension $3g-3-n$.
\item There is an induced diffeomorphism of Riemann surfaces $\chi^s\!:\Sigma^{\rm reg}\to \Sigma^{s,{\rm reg}}$ satisfying $\bar\partial \chi = \nu(s)\partial \chi$.
\item On each $N_i$, $\chi^s$ restricts to a rotation (and thus a hyperbolic isometry).
\end{itemize}
Each node contributes an additional complex parameter via the plumbing construction. First of all, each cusp neighborhood $N_i$ is biholomorphic as a complex-analytic set to $\{z_iw_i=0, |z_i|<c,|w_i|<c\}\subset \co^2$.  To open the node, let $|t_i|<c^2$ and consider the annulus $$N_i^{t_i} = \{z_iw_i=t_i, |z_i|, |w_i| \in(\sfrac{|t_i|}c,c)\}\subset \co^2.$$  If we choose $t=(t_1,\dots,t_n)$ as above, we may replace $N_i$ with $N_i^{t_i}$ (by using the same $z_i,w_i$ coordinates) in order to form $\Sigma^t$.  Since the Beltrami differentials are constructed so that the hyperbolic cusp coordinates are essentially preserved, we have
\begin{itemize}
\item $(s,t)$ near $(0,0)$ form local V-manifold coordinates, the \emph{hyperbolic metric plumbing coordinates}, for $\overline{\mathcal M}_g$.
\end{itemize}

Given these hyperbolic metric plumbing coordinates, we recall Wolpert's grafting metric $g^{s,t}$.  Let $k^{s,t}$ be the complete hyperbolic metric on $\Sigma^{s,t,{\rm reg}}$. We will not use the construction of the grafting metric, but only the following properties \cite{wolpert90}:
\begin{itemize}
\item $g^{s,t}$ is a complete conformal metric on $\Sigma^{s,t,{\rm reg}}$.
\item If $t_i=0$, then $g^{s,t}=k^{s,t}$ on $N_i\cap \Sigma^{s,t,{\rm reg}}$.
\item For $t_i\neq0$, then $g^{s,t}$ is equal to
 \begin{equation} \label{hyp-met-collar}
 \left(\frac{\pi}{|z_i|\log |t_i|} \csc\left(\pi \frac{\log|z_i|}{\log|t_i|} \right)\right)^2 \, |dz_i|^2
 \end{equation}
    on $N_i^{t_i}$.
\item Away from $z_i=w_i=0$, the metrics $g^{s,t}$ on $N_i^{t_i}$ vary real-analytically in $\frac1{\log|t_i|}$ for all $|t_i|<c^2$.
\item There is a uniform constant $C$ so that
\begin{equation} \label{compare-g-k}
\left|\frac {g^{s,t}}{k^{s,t}} - 1\right| \le C \left| \sum_{i=1}^n (\log|t_i|)^{-2}\right|.
\end{equation}
\item There is a uniform constant $C'$ so that the curvature $\kappa_{g^{s,t}}$ satisfies
\begin{equation} \label{graft-curv}
\| \kappa_{g^{s,t}} + 1\|_{C^0} \le C' (\log|t_i|)^{-2}.
\end{equation}
\end{itemize}

Recall a regular cubic differential over a noded Riemann surface $\Sigma$ is given by a holomorphic cubic differential $U$ on $\Sigma^{\rm reg}$ with the following behavior at the nodes:  Let each node be given by $z_iw_i=0$ in local coordinates.  In terms of the $z_i$ and $w_i$ coordinates, we first of all require  $U$ to have a pole of order at most 3 at the origin.  For $x=z_i,w_i$, the \emph{residue} of $U$ is defined to be the $dx^3/x^3$ coefficient of $U$.  The residue does not depend on the choice of local conformal coordinate.  The second condition is that the residues of the $z_i$ and $w_i$ coordinates for each node sum to zero.

It will be useful for us to describe the convergence of cubic differentials in terms of a family of metrics constructed by modifying the grafting metrics $g^{s,t}$.  Our construction is to replace the (locally) hyperbolic grafting metrics in the thin part of each surface by a flat conformal cylindrical metric of uniformly constant diameter.  For $t_i=0$, we will replace each hyperbolic cusp end by a complete flat cylinder, while for $t_i\neq0$ small, we replace the hyperbolic collar by a flat collar.  The details are presented below.

Since the boundary of $\mathcal M_g$ is compact, it can be covered by a finite number of hyperbolic metric plumbing coordinate neighborhoods $V^\alpha$, $\alpha=1,\dots,M$ centered at nodal curves on the boundary.  Define the set $V^0$ to be an open set containing $\overline{\mathcal M}_g\setminus \cup_\alpha V^\alpha$ whose closure does not intersect $\partial{\mathcal M}_g$. $V^0$ lies in the thick part of the moduli space, as it excludes a neighborhood of the boundary.  Consider the universal curve $\pi\!: \overline{\mathcal C}_g \to \overline {\mathcal M}_g$.  For each noded Riemann surface $\Sigma^{s,t}$ in $\pi^{-1}V^\alpha$, $\alpha =0,\dots,M$, define the metric $m^{\alpha,s,t}$ as follows
\begin{itemize}
\item Let $m^{0}$ be the hyperbolic metric on the (necessarily nonsingular and closed) Riemann surface $\Sigma$.
\item For a noded Riemann surface $\Sigma=\Sigma^{s,t}$ in $U^\alpha$, define $m^{\alpha,s,t}$ to be equal to $g^{s,t}$ on $\Sigma\setminus\cup_i N_i^{s,t}$.
\item On $N_i^{t_i}$, consider the quasi-coordinate $\ell = \log x$ for $x=z_i,w_i$. Then for $t_i\neq0$, $$g^{s,t} = \left(\frac{\pi}{\log |t_i|} \csc \left(\pi\frac{{\rm Re}\,\ell}{\log|t_i|} \right) \right)^2 \,|d\ell|^2,$$ for  $\log|t_i|-\log c \le {\rm Re}\,\ell \le \log c.$
    For $t_i=0$, $g^{s,t} = ({\rm Re}\, \ell)^{-2}\,|d\ell|^2$ for ${\rm Re}\,\ell \le \log c$.
\item For the $t_i=0$ case, consider the half-cylinder $\{{\rm Re}\,\ell\le 2\log c\}$ with flat metric $f=(2\log c)^{-2}\,|d\ell|^2$.
\item For $t_i\neq0$, let $$ K = \frac{\log|t_i|}\pi \arcsin \left(\frac{\pi}{\log|t_i|} \cdot 2\log c\right),$$
    consider the annulus $\{{\rm Re}\,\ell \in[\log|t_i|-K, K] \}$ with flat metric $f=(2\log c)^{-2}\,|d\ell|^2$.  Note this metric is equal to $g^{s,t}$ on the boundary of the annulus.
\item Now on each surface we interpolate between the two metrics.  Let $\eta$ be a smooth nonnegative function of ${\rm Re} \,\ell$ which is equal to 1 for ${\rm Re}\,\ell \le 2\log c$ and equal to 0 for ${\rm Re}\,\ell \ge \log c$.  On each connected component of $N_i^{\rm reg}$ for $t_i=0$, define the metric $m^{\alpha,s,t}$ to be $(g^{s,t})^{1-\eta({\rm Re}\,\ell)}\cdot f^{\eta({\rm Re}\,\ell)}$.
\item We make a similar definition for $t_i\neq0$, shifting the interpolating factor and adjusting for the fact that $N_i^{t_i}$ is connected.  Let $$\phi({\rm Re}\,\ell) = \eta({\rm Re}\,\ell - K + 2\log c) \cdot \eta(2\log c +\log|t_i|-K-{\rm Re}\,\ell ).$$
    Then the metric $m^{\alpha,s,t}$ restricted to $N_i^{t_i}\subset\Sigma=\Sigma^{s,t}$ is defined to be $(g^{s,t})^{1-\phi({\rm Re}\,\ell)}\cdot f^{\phi({\rm Re}\,\ell)}$.
\item Note that $m^{\alpha,s,t}$ is always a complete conformal metric on $\Sigma^{\rm reg}$. It is always equal to $g^{s,t}$ outside cusp and collar neighborhoods, and well inside these small neighborhoods, the metric is flat cylindrical of uniform circumference.  These two regions are glued together along annual regions using a uniform partition of unity, and so the metric $m^{\alpha,s,t}$ on these annular regions is smooth and has uniform geometry.  In particular, the $m^{\alpha,s,t}$ metrics have uniformly bounded Gauss curvature.
\item Moreover, there is a uniform positive constant $C>0$ so that for on every Riemann surface $\Sigma=\Sigma^{s,t}$ represented in $V^\alpha$,
    \begin{equation}
    \frac{m^{\alpha,s,t}}{g^{s,t}}\ge C. \label{bound-m-g}
    \end{equation}
    By our construction, $m^{\alpha,s,t}/g^{s,t}\ge 1$ in the region where $m^{\alpha,s,t}$ is flat.  The existence of such a bound $C$ on the region of the interpolation follows from compactness considerations, while outside these two regions, the two metrics are equal.
\end{itemize}

We will also use a basic description of the thick-thin decomposition of hyperbolic surfaces and of the universal curve. See e.g.\ \cite{wolpert10}. For positive $\epsilon$ small enough, the locus of points Thin$_\epsilon$ on a complete hyperbolic surface with injectivity radius less that $\epsilon$ is called the thin part of the moduli space, while the complement is the thick part Thick$_{\epsilon}$.  The thin part is a disjoint union of punctured disks (cusps) and annuli (collars).  Margulis's Lemma shows there is a fixed $\epsilon_0>0$ so that this is true for all $0<\epsilon<\epsilon_0$, while Mumford's Compactness Theorem shows Thick$_\epsilon$ is compact. We will need to relate this to the hyperbolic metric plumbing coordinates.  In particular, in each $V^\alpha$ neighborhood, (\ref{compare-g-k}) shows that for any sequence of points in $\overline{\mathcal C}_g$, the injectivity radius of the hyperbolic metric goes to zero if and only if the plumbing coordinates $z_j,w_j$ for the appropriate collar go to zero.
\begin{lem} \label{thick-lemma}
For any convergent sequence in $\overline{\mathcal C}_g$, either it converges to a node (in which some $z_j,w_j$ coordinates are 0) or there is an $\epsilon>0$ so that all but a finite number of elements of the sequence lie in Thick$_\epsilon$.
\end{lem}

Now we describe the convergence of sequences in the total space of the bundle of regular cubic differentials over $\overline{\mathcal M}_g$. See e.g.\ \cite{wolpert12}. Since this space is a V-manifold vector bundle (and so naturally has a first-countable topology), convergence of sequences determines the topology.  Consider a sequence $(\Sigma_j,U_j)$ of pairs of noded Riemann surfaces $\Sigma_j$ and regular cubic differentials $U_j$ on $\Sigma_j$.  Then we say $(\Sigma_j,U_j)\to(\Sigma,U)$ if and only if
\begin{itemize}
\item $\Sigma_j\to\Sigma$ in $\overline{\mathcal M}_g$.
\item For $j$ large, there is an $\alpha$ so that $\Sigma_j,\Sigma$ are in the neighborhood $V^\alpha$. Thus there are hyperbolic plumbing coordinates $(s,t)$ for $\Sigma$, $(s_j,t_j)$ for $\Sigma_j$, and $(s_j,t_j)\to (s,t)$.  Also consider the neighborhoods $N_i^{s_j,t_j}$ as above.
\item On $\Sigma\setminus\cup_i N_i$, we assume $[(\chi^{s_j})^{-1}]^* U_j \to [(\chi^s)^{-1}]^* U$ uniformly with respect to the hyperbolic metrics.
\item On each $N_i$, consider the hyperbolic cusp coordinates $z_i,w_i$ on the collar $N_i^{s_j,t_j}$. Recall that $z_i$ and $w_i$ have domains $|z_i|,|w_i| \in (\frac{|t_j|}c,c)$.  Then for $x=z_i,w_i$, write $U_j = \hat U_j\,dx^3$.  The convergence condition is that for all $x\in\{z_i,w_i\}$,  $\hat U_j\to \hat U$ normally in the domains $\{|x|\in (\frac{|t_j|}c,c)\}$.
\end{itemize}

We have the following description of the convergence of families of regular cubic differentials.
\begin{lem}
$(\Sigma_i,U_i)\to (\Sigma,U)$ if and only if $\Sigma_i\to\Sigma$ in $\overline{\mathcal M}_g$ and (with respect to the appropriate member $V^\alpha$ of the cover of $\overline{\mathcal M}_g$) $[(\chi^{s_j})^{-1}]^* U_j \to [(\chi^s)^{-1}]^* U$ in $L^\infty$ with respect to the $m^{\alpha,\Sigma_i}$ metrics with the additional caveat that on the $N_i^{t_i}$ regions, the convergence in the hyperbolic metric plumbing coordinates is normal convergence on the domains $\{|x|\in (\frac{|t_j|}c,c)\}$, which vary in size as $t_j$ varies.
\end{lem}
\begin{proof}
This lemma follows fairly easily from the considerations above except for one point, namely that the convergence with respect to the flat metrics is $L^\infty$ for regular cubic differentials.  For the case of $t_j\neq0$, in the $\ell$ quasicoordinate, consider $U_j = \tilde U_j \,d\ell^3 = \tilde U_j \,\frac{dx^3}{x^3}$ for $x=z_j,w_j$.  Our convergence assumption implies that in annular bands near $|x|=c$, $\tilde U_j$ converges uniformly (and so the $\tilde U_j$ are uniformly bounded there).  Wolpert calls this condition \emph{band bounded} \cite{wolpert12}.  We have these bounds at both ends of the annulus, and so the maximum modulus principle implies that there are uniform bounds on $\tilde U_j$ across the annulus $\{|x|\in(\frac{|t_j|}c,c)\}$, and moreover, the bounds are independent of the $t_j$ as $t_j\to0$.  But the metrics $m^{\alpha,s,t}$ are constructed so that uniform convergence of the $\tilde U_j$ is the same as $L^\infty$ convergence with respect to the metrics in the annulus.

For the case of $t_j=0$, the definition of regular cubic differential implies $\tilde U_j$ is bounded near $x=0$. Thus $x=0$ is a removable singularity for $\tilde U_j$. The maximum modulus principle still applies and the convergence is again in $L^\infty$ with respect to the metrics.
\end{proof}

\section{Hyperbolic affine spheres}

\subsection{Relationship to convex $\rp^2$ structures}

For $\Omega\subset\re^n\subset\rp^n$ a properly convex domain, consider the cone $\mathcal C$ in $\re^{n+1}$ given by $\{t(x,1): t\in\re^+, x\in\Omega\}$.  $\Omega$ is the image of $\mathcal C$ under the projection $\pi\!:\re^{n+1}\setminus\{0\} \to\rp^n$. The \emph{proper} convexity of $\Omega$ is equivalent to $\mathcal C$ being properly convex, in that it is convex and contains no lines.  There is a unique (properly normalized) hyperbolic affine sphere $\mathcal H$ asymptotic to the boundary of $\mathcal C$ which is invariant under special linear automorphisms of $\mathcal C$ \cite{cheng-yau77,cheng-yau86}. Then $\pi$ restricts to a diffeomorphism from $\mathcal H$ to $\Omega$, and projective automorphisms of $\Omega$ lift to special linear automorphisms of $\mathcal C$, which act on $\mathcal H$.

In order to define the hyperbolic affine sphere, we introduced the \emph{affine normal}. To any strictly convex smooth hypersurface $\mathcal H$ in $\re^{n+1}$, the affine normal $\xi$ is a smooth transverse vector field which is invariant under unimodular affine transformations of $\re^{n+1}$.  The hyperbolic affine sphere (which we take to be normalized so the center is at the origin and the affine mean curvature is $-1$) can be defined as a convex hypersurface $\mathcal H$ so that at all points, the affine normal is equal to the position vector.   See e.g. \cite{blaschke,calabi72,cheng-yau86,li-simon-zhao,loftin02c,nomizu-sasaki}. A hyperbolic affine sphere is always equivalent to one normalized as above by an affine motion in $\re^{n+1}$.  For the rest of this work, we will always assume all hyperbolic affine spheres are so normalized.

The first natural structure equation on $\mathcal H$ is the following formula of Gauss type:
\begin{equation} \label{gauss-eq}
D_XY = \na_XY + h(X,Y)f,
\end{equation}
where $D$ is the flat connection on $\re^{n+1}$, $X$ and $Y$ are tangent vector fields on $\mathcal H$, $f$ is the position vector of points on $\mathcal H$ (which is transverse to the tangent space).  Then $D_XY$ is split into $\na_XY$, the part in the tangent space, and $h(X,Y)f$, the part in the span of $f$.  $\na$ is a projectively-flat torsion-free connection on $\mathcal H$, while $h(X,Y)$ is a positive-definite symmetric tensor called the \emph{Blaschke metric} or the \emph{affine metric}.  Another import local invariant is the \emph{cubic tensor}, or \emph{Pick form}, which is the difference of the Levi-Civita connection of $h$ and the connection $\na$.  The cubic tensor measures how far a hypersurface is from a hyperquadric, as a general theorem of Maschke, Pick, and Berwald implies
\begin{thm}
$\mathcal H$ is a hyperboloid if and only if its cubic tensor vanishes identically.
\end{thm}
(This is a special case of the more general theorem that any nondegenerate smooth hypersurface is a hyperquadric if and only if its cubic tensor, when defined with respect to the affine normal $\xi$, vanishes identically.)

For a hyperbolic affine sphere $\mathcal H$, the completeness of the Blaschke metric is equivalent to $\mathcal H$ being properly embedded.  In fact, we have the following theorem of Cheng-Yau \cite{cheng-yau77,cheng-yau86} and Calabi-Nirenberg (unpublished), with clarifications by Gigena \cite{gigena81}, Sasaki \cite{sasaki80}, and A.M.\ Li \cite{li90,li92}.
\begin{thm} \label{complete-has}
For $\Omega$ a properly convex domain in $\rp^n$, consider the cone $\mathcal C\subset\re^{n+1}$ over $\Omega$. Then there is a unique properly embedded hyperbolic affine sphere $\mathcal H\subset \mathcal C$ which is centered at the origin, has affine mean curvature $-1$, and which is asymptotic to $\partial \mathcal C$. $\mathcal H$ is invariant under volume-preserving linear automorphisms of $\mathcal C$, and $\mathcal H$ is diffeomorphic to $\Omega$ under projection. The Blaschke metric on $\mathcal H$ is complete.

Conversely, let $\mathcal H$ be an immersed hyperbolic affine sphere normalized to have center 0 and affine mean curvature $-1$. If the Blaschke metric on $\mathcal H$ is complete, then $\mathcal H$ is properly embedded in a proper convex cone $\mathcal C$ centered at the origin and is asymptotic to the boundary $\partial \mathcal C$.
\end{thm}

If $M  = \Gamma\backslash\Omega$ is a properly convex $\rp^n$ manifold, then we may lift the representation $\Gamma$ to $ \pgl{n+1}$ to volume-preserving linear actions $\tilde\Gamma$ on the cone $\mathcal C$ over $\Omega$.  By the invariance of $H$, we find $\tilde\Gamma$ acts on $\mathcal H$, and $M$ is naturally diffeomorphic to $\tilde\Gamma\backslash \mathcal H$.  The invariant tensors on $\mathcal H$ (the Blaschke metric and the cubic form) descend to $M$.

Given a properly convex cone $\mathcal C\subset \re^{n+1}$, the dual cone $\mathcal C^*$ is the cone in the dual vector space $\re_{n+1}$ to $\re^{n+1}$ given by all $\ell\in\re_{n+1}$ so that $\ell(x)>0$ for all $x\in\mathcal C$. Upon projecting to projective space, if $\Omega = \pi(\mathcal C)\subset \rp^n$, the we also have a dual convex projective domain $\Omega^* \subset \rp_n$.  From this formulation, we remark
\begin{lem} \label{proj-dual-inclusion}
If $\Omega_1\subset\Omega_2$ are properly convex domains in $\rp^n$, then $\Omega_1^*\supset\Omega_2^*$.
\end{lem}
There is a related duality result on hyperbolic affine spheres due to Calabi (see \cite{gigena78}, and \cite{loftin10} for an exposition).  For $\mathcal H$ a hyperbolic affine sphere, consider the conormal map $\mathcal H\to \re_{n+1}$ given by $$ x\mapsto \ell, \quad \mbox{where} \quad \ell(x)=1\quad \mbox{and} \quad \ell(T_xH) = 0.$$
\begin{thm} \label{dual-has}
Given a properly convex cone $\mathcal C\subset\re^{n+1}$ with corresponding hyperbolic affine sphere $H\subset \mathcal C$, then the conormal map maps $\mathcal H$ diffeomorphically onto the unique hyperbolic affine  sphere $\mathcal H^*$ corresponding to the dual cone $\mathcal C^*\subset \re_{n+1}$.  The conormal map is an isometry with respect to the Blaschke metrics on $\mathcal H$ and $\mathcal H^*$ and takes the cubic form $C\mapsto -C$.
\end{thm}

We will also use the relationship between the conormal map and the \emph{Legendre transform}. If $v\!:\Omega\to\re$ is a smooth convex function on a convex domain $\Omega\subset\re^n$ with coordinates $x^i$, the we define the Legendre transform function $v^*$ by
$$ v^* + v = x^i\,\frac{\partial v}{\partial x^i}.$$
The function $v^*$ is considered primarily as a function of the variables $\frac{\partial v}{\partial x^i}$, and as such it is an involution on the space of convex functions. If $\mathcal H$ is a hypersurface given by a radial graph of $-\frac1v$ for  a convex function $v$, $$\mathcal H = \left\{-\frac1{v(x)}(x^1,\dots,x^n,1) : x\in\Omega\right\},$$
Then the image of the conormal map of $\mathcal H$ is given by
\begin{equation} \label{leg-conormal}
\left\{\left(-\frac{\partial v}{\partial x^1}, \cdots, -\frac{\partial v}{\partial x^n}, v^*\right)\right\}.
\end{equation}
Therefore, the conormal map essentially (up to a few minus signs) interchanges the radial graph of $-\frac1v$ with the Cartesian graph of the Legendre transform $v^*$.

We also mention here the relationship between hyperbolic affine spheres and a real Monge-Amp\`ere equation, which is due to Calabi \cite{calabi72}.  The formulation here also depends on results in Gigena \cite{gigena81}.

\begin{thm} \label{has-ma-thm}
Given a properly convex domain $\Omega\subset\re^n\subset\rp^n$, the hyperbolic affine sphere asymptotic to the boundary of the cone over $\Omega$ is given by the radial graph of $-\frac1v$ $$\left\{-\frac1{v(x)}(x,1): x\in \Omega\right\},$$ for $v$ the convex solution unique solution of the Dirichlet problem $v$ continuous on $\bar\Omega$, $v=0$ on $\partial\Omega$, and
\begin{equation} \label{ma-eq}
\det \left(v_{ij} \right) = \left(-\frac1v\right)^{n+2},
\end{equation}
for $v_{ij}$ the Hessian matrix of $v$.
The Blaschke metric is $-\frac1v v_{ij} \,dx^i\,dx^j$.
\end{thm}

Loewner-Nirenberg first solved this equation for convex domains with regular boundary in dimension two \cite{loewner-nirenberg}. Cheng-Yau solved this equation in the general case \cite{cheng-yau77}.

\subsection{Benoist-Hulin's convergence of invariant tensors}

For the reader's convenience we provide a new proof of Benoist-Hulin's theorem.  The new proof is different in its treatment of the higher-order estimates of the Monge-Amp\`ere equation.

\begin{thm}\cite{benoist-hulin13} \label{converge-has-potential}
Let $\Omega_j,\Omega_\infty$ be bounded convex domains in $\re^n\subset\rp^n$. Assume $\Omega_j\to\Omega_\infty$ converges in the Hausdorff topology with respect to the Fubini-Study metric on $\rp^n$. Then the solutions $v_j$ to the Dirichlet problem (\ref{ma-eq}) on $\Omega_j$ converge in $C^\infty_{\rm loc}$ to the solutions $v_\infty$ on $\Omega_\infty$.
\end{thm}

Since the projectively-invariant tensors the Blaschke metric and cubic tensor are formed from $v$ and its derivatives, we have the following result of Benoist-Hulin:
\begin{thm}\label{local-converge}
Under the hypotheses of the theorem, the Blaschke metrics and cubic tensors converge in $C^\infty_{\rm loc}$.
\end{thm}

\begin{proof}[Proof of Theorem \ref{converge-has-potential}]
Our new proof of Benoist-Hulin's theorem differs in its treatment of the $C^2_{\rm loc}$ estimates.  For $C^0_{\rm loc}$ estimates, we follow \cite{benoist-hulin13} by using the maximum principle.
Pick an inhomogeneous affine coordinate chart $\re^2\subset\rp^2$ so that $0\in\Omega_\infty$ and $\Omega_\infty$ is bounded in $\re^2$.
This implies there are $\epsilon_j\to0$ so that
$$(1+\epsilon_j)\Omega_j\supset\Omega_\infty \supset (1-\epsilon_j)\Omega_j.$$
For $v_j$ the solutions to the Dirichlet problem for (\ref{ma-eq}) on $\Omega_j$, the corresponding solution on $t\Omega_j$ is $t^{\frac n{n+1}} v_j (t^{-1}x)$, and the maximum principle implies that if $\mathcal O\subset \mathcal U$, then $v_{\mathcal O}\ge v_{\mathcal U}$. In particular, this shows
$$ (1+\epsilon_j)^{\frac n{n+1}} v_j(x/(1+\epsilon_j)) \le v_\infty(x) \le  (1-\epsilon_j)^{\frac n{n+1}} v_j(x/(1-\epsilon_j)),$$
which in turns shows $v_j\to v_\infty$ in $C^0_{\rm loc}$ on $\Omega_\infty$. Define
$$v_j^+(x) = (1+\epsilon_j)^{\frac n{n+1}} v_j(x/(1+\epsilon_j)), \quad v_j^-(x) = (1-\epsilon_j)^{\frac n{n+1}} v_j(x/(1-\epsilon_j)).$$
Then $v_j^-\to v_\infty$ in $C^0_{\rm loc}$, $v_j^+\to v_\infty$ in $C^0_{\rm loc}$, and $v_j^+(x)\ge v_j(x)\ge v_j^-(x)$.

The $C^1_{\rm loc}$ estimates depend only on convexity.  Let $T$ be a large triangle in $\re^2$ which contains $\Omega_\infty$ and all the $\Omega_k$ for $k$ large. Then the solution $w_T$ to the Monge-Amp\`ere equation has a minimum value $-M$, and the maximum principle shows that the solutions $v_\Omega$ and $w_{\Omega_k}$ also must satisfy $w\ge -M$.  Now for such a $w$, let $g=w(y+tv)$, where $t\in\re$, $y$ is a boundary point, and $v$ is a unit vector pointing into the domain.  Then $g$ is continuous, $g(0)=0$, and $g$ is smooth and strictly convex on an interval $(0,R)$. In particular, $g'$ is increasing on $(0,R)$. For $t\in(0,R),$
$$ g'(t) \ge \frac{g(t)-g(0)}{t-0} = \frac{g(t)}t \ge -\frac Mt.$$
Together with the estimate along the same ray traversed in the opposite direction, this shows $|g'|$ is uniformly bounded on any compact set, with the bound depending on the diameter of the domain and the distance to the boundary.

For the interior $C^2$ estimates, we have the following standard result following Pogorelov.  The following theorem we quote is a direct application of Theorem 17.19 in \cite{gilbarg-trudinger}.  We note that we must restrict to a sub-level domain $\{v<-\epsilon\}$, as estimates on the function $v\mapsto \left(-\frac1v\right)^{n+2}$ and its first two derivatives are need to apply Theorem 17.19.

\begin{thm}\label{c2-est}
Consider the solution $v$ to the Dirichlet problem (\ref{ma-eq}) on a bounded convex domain $\Omega\subset\re^n$.  Let $\epsilon>0$, and let $\mathcal O = \{v<-\epsilon\}$. Then there is a constant $C$ depending on $n$, $\epsilon$, $\|v\|_{C^1(\mathcal O)}$ and the diameter of $\mathcal O$ so that if $\mathcal O' \subset\subset \mathcal O$, then
$$ \sup_{\mathcal O'} |D^2v|\le \frac C{{\rm dist}(\mathcal O',\partial \mathcal O)}.$$
\end{thm}

By the arguments above, we have $v_k^+ \to v$,  $v_k^- \to v$, and $v_k^+\ge v\ge v_k^-$.  On each compact set, convexity shows that the convergence is uniform (this follows from the $C^1$ estimates above and Ascoli-Arzela). Since $v_k^+\ge v_k \ge v_k^-$, we have that $v_k\to v$ uniformly on compact subsets of $\Omega_\infty$. Now let $K$ be a compact subset of $\Omega_\infty$.  For large $k$, $K\subset\Omega_k$ as well.  By continuity, $\max_K v= -3\epsilon$ for some $\epsilon>0$.  The set $\{v\le -2\epsilon\}$ is also compact, and by uniform convergence, we can see that for large enough $k$, we have
$$\{v_k<-\epsilon\}\supset \{v < -2\epsilon\} \supset \{v \le -3\epsilon\} \supset K.$$
So if we define $\mathcal O_k = \{v_k<-\epsilon\}$,
$${\rm dist}(K,\partial \mathcal O_k) >{\rm dist}(K,\{v=-2\epsilon\})>0.$$
Moreover, the diameter of $\mathcal O_k$ is bounded by that of $ \Omega_k$, which is uniformly bounded.  Thus we have estimates for Theorem \ref{c2-est} which are independent of $k$, and on $K$, we have uniform $C^2$ estimates on $v_k$.

Finally, we use the Evans-Krylov estimates to find interior $C^{2,\alpha}$ estimates.  See Theorem 17.14 in \cite{gilbarg-trudinger}. In particular, on any compact subset $K$ of $\Omega$, there is an $\alpha\in(0,1)$ so that the $C^{2,\alpha}$ estimates on a slightly smaller compact subset $K'$ depends only on the distance ${\rm dist}(K',\partial K)$, the $C^2$ estimates of $v$ on $K$, and bounds on the eigenvalues of the Hessian matrix of $v$.  These estimates are similar to but easier to apply than the Pogorelov estimates above.  The main new ingredient is to bound the eigenvalues of the Hessian matrix of $v_k$. The largest eigenvalue is bounded by Pogorelov's bounds on the second derivatives, while the smallest eigenvalue is bounded away from 0 by using the Monge-Amp\`ere equation $\det(v_{k,ij}) = (-v_k)^{-n-2}$ and the bounds away from 0 on $v_k$.

We have shown so far that on every compact $K\subset \Omega_\infty$, there are uniform $C^{2,\alpha}$ bounds on $v$ and $v_k$ for $k$ large, and also $v_k\to v$ uniformly.  This means that the Ascoli-Arzela Theorem applies to show that (subsequentially at least), $v_k\to v$ in $C^2$ on $K$. But every subsequence of $\{v_k\}$ then has a subsequence converging to $v$ in $C^2$, and so we see $v_k\to v$ in $C^2$ on $K$.

Higher-order interior estimates and convergence are standard once $C^{2,\alpha}$ estimates are in place, and so $v_k\to v$ in $C^\infty_{\rm loc}$ on $\Omega_\infty$.
\end{proof}

Benoist-Hulin's $C^\infty_{\rm loc}$ convergence of affine invariants is quite strong.  However, the conformal structure at the end of a surface is not quite local in this sense, and so we will need to expend a more effort to compute the conformal structures at the ends.

\subsection{An estimate on Blaschke metrics}

We begin this subsection with a quantitative version of the following theorem of Cheng-Yau \cite{cheng-yau86} (as clarified by Li \cite{li90}): A hyperbolic affine sphere $\mathcal H$ with complete Blaschke metric is properly embedded in $\re^{n+1}$ and is asymptotic to the boundary of the convex cone given by the convex hull of $\mathcal H$ and its center.  We consider a quantitative version involving geodesic balls of large radius.

\begin{prop} \label{cone-quant}
Let $\mathcal H$ be a hyperbolic affine sphere given by the radial graph of $-\frac1v$, where $v$ solves the Dirichlet problem (\ref{ma-eq}) over a convex domain $\Omega$. Let  $v$ be normalized so that $v(0)=-1$ and $dv(0)=0$.  Assume there are positive constants $\gamma,\delta,\epsilon$, and consider the ball in $\Omega$ the ball $\{x:|x|<\epsilon\}$. Assume the Blaschke arc-length $\ell(0,x)$ of radial paths $\{tx:t\in[0,1]\}$ satisfies
\begin{equation}\label{arc-length-lower-bound}
\ell(0,x)>\delta \quad \mbox{for}\quad |x|=\epsilon.
\end{equation}
Also assume
\begin{equation} \label{v-grad-bound-below}
v(x)>-1+\gamma,\quad \frac1\epsilon x^i\frac{\partial v}{\partial x^i}(x)>\gamma \quad \mbox{for}\quad |x|=\epsilon.
\end{equation}
Then there are constants $C=C(n)$ and $A=A(C,\gamma,\delta)$ so that
if $B^h_0(Q)$ is the geodesic ball centered at $x=0$ of radius $Q$, then
$$
 \{x : v(x) < - e^{(A-Q)/C} \} \subset B_0^h(Q).
$$
\end{prop}

\begin{proof}
Consider a hyperbolic affine sphere $\mathcal H$ normalized with its center at that origin, affine mean curvature $-1$. Moreover, assume $(0,\dots,0,1)\in\mathcal H$ and that the tangent plane at that point is horizontal. Consider the dual affine sphere $\mathcal H^*$ can be written as the graph $\{(y,p(y))\}$.  $\mathcal K = \{(-y,p(y))\}$, as the image of $\mathcal H^*$ under a volume-preserving linear map, is also a hyperbolic affine sphere.

We follow a suggestion in \cite{li-simon-zhao}, Remark 2.7.2.6(ii). The height function $p$ on $\mathcal K$ is a positive eigenfunction of the Laplacian with respect to the Blaschke metric $h$, which is complete and has Ricci curvature uniformly bounded below.  A gradient estimate of Yau \cite{schoen-yau94} then applies to show that there is a uniform constant $C$ depending only on the dimension so that
\begin{equation} \label{grad-estimate}
\|d(\log p)\|_h \le C
\end{equation}
Consider $v$ the Legendre transform of $p$.   Theorem \ref{dual-has} and  (\ref{leg-conormal}) show that $\mathcal H = \{-\frac1{v(x)}(-x,1):x\in\Omega\}$ is essentially the radial graph of $-\frac1v$. Recall the Legendre transform is given by
\begin{equation} \label{leg-transform}
p+v = x^iy_i,\qquad y_i = \frac{\partial v}{\partial x^i},\quad x^i = \frac{\partial p}{\partial y_i}.
\end{equation}
We primarily consider $p=p(y)$ and $v=v(x)$. Choose coordinates on $\re^{n+1}$ so that $v(0)=-1$ and $dv(0)=0$.   Since $v$ is convex, it has its minimum at $x=0$.
Differentiating (\ref{leg-transform}) shows that
\begin{equation} \label{dp-dx-equation}
 \frac{\partial p}{\partial x^i} = x^j\, \frac{\partial^2v}{\partial x^i\partial x^j}.
\end{equation}

We follow the proof of Theorem 2.7.1.9 in \cite{li-simon-zhao}. Use the expression for the Blaschke metric in Theorem \ref{has-ma-thm} above to compute for $\bar x\in\Omega$ the Blaschke length $\ell$ of the path $\mathcal P = \{t\bar x: 0\le t\le1\}$ to be
$$ \ell(0,\bar x) = \int_0^1 \left(-\frac1v\, \frac{\partial^2v}{\partial x^i \partial x^j} \bar x^i\bar x^j \right)^{\frac12}\,dt.$$

Assume (\ref{arc-length-lower-bound}) and (\ref{v-grad-bound-below}) and use (\ref{grad-estimate}), (\ref{leg-transform}) and (\ref{dp-dx-equation}). Let $v^{ij}$ denote the inverse matrix of $\frac{\partial^2v}{\partial x^i\partial x^j}$. Compute
\begin{eqnarray*}
\ell(0,\bar x) &\le& \delta + \int_{\epsilon/|\bar x|}^1 \left(-\frac1v\, \frac{\partial^2v}{\partial x^i \partial x^j} \bar x^i\bar x^j \right)^{\frac12}\,dt \\
&=& \delta + \int_{\epsilon/|\bar x|}^1 \left(-\frac1{vt^2} \, \frac{\partial^2v}{\partial x^i \partial x^j} (t\bar x^i)(t\bar x^j)\right)^{\frac12}\,dt \\
&=& \delta + \int_{\epsilon/|\bar x|}^1 \left(-\frac1{vt^2} \, v^{ij}\,\frac{\partial p} {\partial x^i} \,\frac{\partial p} {\partial x^j} \right)^{\frac12}\,dt \\
&=& \delta + \int_{\epsilon/|\bar x|}^1 -\frac1{vt} \, \|dp\|_h \,dt \\
&\le& \delta + \int_{\epsilon/|\bar x|}^1 -\frac C{vt} \, p \,dt \\
&=& \delta + \int_{\epsilon/|\bar x|}^1 -\frac C{vt}\left(t \bar x^j\frac{\partial v}{\partial x^j} - v\right) \,dt \\
&=& \delta - C\int_{\epsilon/|\bar x|}^1 d(\log|v|)(\mathcal P(t))  + \int_{\epsilon/|\bar x|}^1 \frac Ct\,dt \\
&=& \delta + C\left[-\log|v(\bar x)| + \log\left|v\left(\frac{\epsilon\bar x}{|\bar x|}\right)\right| - \log\left|\frac{\epsilon}{|\bar x|}\right|\right]\\
&\le& \delta - C\log|v(\bar x)| + C\log(1-\gamma) - C\log\epsilon + C \log |\bar x|.
\end{eqnarray*}
Now the convexity of $v$, together with (\ref{v-grad-bound-below}) and  $\Omega =\{v<0\}$, implies $|\bar x| < \frac\epsilon\gamma$, and so there is a constant $A = A(C,\gamma,\delta)$ so that
$$
d_h(0,\bar x)\le \ell(0,\bar x) \le A - C\log|v(\bar x)|
$$
for $d_h$ the Blaschke distance.
\end{proof}

Now we use Proposition \ref{cone-quant} to show that sequences of points in convex domains must be separated from each other if their Blaschke distance approaches $\infty$.  In fact, the set of sequences of points in $\Omega_j$ which converge in the Benz\'ecri sense may be partitioned into equivalence classes according to whether their Blaschke distances remain bounded.

\begin{prop} \label{separate-O}
Let $\Omega_j\to\mathcal O$ be a convergent sequence of properly convex domains in $\rp^n$ with respect to the Hausdorff topology.  Assume $\rho_j\Omega_j\to\mathcal U$ for $\rho_j\in\Sl{n+1}$. Assume $q_j\to q$ for $q_j\in\Omega_j$ and $q\in\mathcal O$, and $r_j\to r$ for $r_j\in\rho_j\Omega_j$ and $r\in\mathcal U$.  Assume that the Blaschke distance $d_{\Omega_j}(q_j,\rho_j^{-1}r_j)\to \infty$. Then there does not exist a sequence $z_j\in\Omega_j$ so that $z_j\to z\in\mathcal O$ and $\rho_j z_j\to w\in\mathcal U$.
\end{prop}

\begin{proof}
Choose coordinates on $\mathcal O$ so that $q=0$ and so that the hyperbolic affine sphere $\mathcal H_{\mathcal O}$ is given by the radial graph of $-\frac1v$.  Also assume  $v(0)=-1$ and $dv(0)=0$.  In these coordinates,  Theorem \ref{converge-has-potential} shows that the corresponding functions $v_j$ on $\Omega_j$ converge to $v$ in $C^\infty_{\rm loc}$. In particular, we can change coordinates on $\Omega_j$ to assume $q_j=0$, $v_j(0)=-1$, and $dv_j(0)=0$, while still maintaining $\Omega_j\to\mathcal O$. The $C^\infty_{\rm loc}$ convergence of $v_j\to v$ implies that there are positive constants $\gamma,\delta,\epsilon$ so that the hypotheses of Proposition \ref{cone-quant} are satisfied.  If $d_{\Omega_j}(q_j,\rho_j^{-1}r_j) \ge Q_j$, then $v_j(\rho_j^{-1}r_j)\ge - e^{A-Q_j}$. Thus $v_j(\rho^{-1}r_j)\to 0$ and $\rho_j^{-1}r_j$ has no limit in $\mathcal O$.

To prove the result, assume there is such a sequence $z_j$. Note that $z$ and $q$ are of finite Blaschke distance from each other in $\mathcal O$, as are $r$ and $w$ in $\mathcal U$.  Thus Theorem \ref{local-converge} implies there is a constant $C$ so that for all $j$ large enough, the Blaschke distances $d_{\Omega_j}(z_j,q_j)\le C$ and $d_{\Omega_j}(z_j,\rho^{-1}_jr_j) = d_{\rho_j\Omega_j}(\rho_jz_j,r_j)\le C$.  Then the triangle inequality implies that $d_{\Omega_j}(z_j,z_j)\to\infty$, which provides a contradiction.
\end{proof}

\begin{cor}
Proposition \ref{separate-O} holds with the Hilbert metric in place of the Blaschke metric.
\end{cor}
\begin{proof}
Benoist-Hulin \cite{benoist-hulin13} prove that these two metrics are uniformly bi-Lipschitz.
\end{proof}

\subsection{Wang's developing map}
For a two-dimensional hyperbolic affine sphere $\mathcal H$, one may consider the conformal structure with respect to the Blaschke metric to give $\mathcal H$ the structure of a simply-connected open Riemann surface.  As the geometry is derived from the elliptic Monge-Amp\`ere equation (\ref{ma-eq}), it should not be surprising that holomorphic data on the Riemann surface comes into play.  In particular, with respect to a local conformal coordinate $z$, the cubic form, upon lowering an index with the metric, is of the form $C = U\,dz^3 + \bar U d\bar z ^3$, for $U\,dz^3$ a holomorphic cubic differential.

C.P.\ Wang worked out the developing map for hyperbolic affine spheres in $\re^3$ \cite{wang91}.  (Much earlier, \c{T}i\c{t}eica analyzed a slightly different case of non-convex proper affine spheres \cite{tzitzeica08,tzitzeica09}.)  Below, we present a synopsis of Wang's theory, as presented in \cite{loftin02c}.

Let $\mathcal D$ be a simply-connected domain in $\co$, and let $f\!:\mathcal D \to \re^3$ represent an immersed surface so that $f$ is conformal with respect to the Blaschke metric. Let $z$ be a conformal coordinate on $\mathcal D$. Let $\mathcal H=f(\mathcal D)$ be the immersed surface.  Then $f_z,f_{\bar z}$ span the complexified tangent space $T_f\mathcal H\otimes_\re\co$, and thus $\{f,f_z,f_{\bar z}\}$ is a frame of $\re^3\otimes_{\re}\co$ at each point of $\mathcal H$.

Consider $f,f_z,f_{\bar z}$ as column vectors, and form the frame matrix
$$ F = (f,f_z, f_{\bar z}).$$
Let $e^\psi|dz|^2$ and $U$ be a conformal metric and cubic differential respectively on $\mathcal D$.  Then the structure equation (\ref{gauss-eq}) is equivalent to the following first-order system
\begin{equation} \label{evolve-frame}
F_z = F \left(\begin{array}{ccc} 0&0&\frac12e^\psi \\ 1&\psi_z &0 \\ 0 &Ue^{-\psi}&0 \end{array} \right), \qquad
F_{\bar z} = F\left(\begin{array}{ccc} 0&\frac12e^\psi &0 \\ 0 & 0&\bar U e^{-\psi} \\ 1& 0 & \psi_{\bar z} \end{array} \right).
\end{equation}
This system of equations is integrable if and only if the following two conditions hold:
\begin{eqnarray} \label{psi-eq}
\psi_{z\bar z} + |U|^2 e^{-2\psi} -\sfrac12 e^\psi &=& 0, \\
U_{\bar z} &=& 0.
\end{eqnarray}
In this case, if at a point $z_0\in\mathcal D$  initial conditions
\begin{equation} \label{initial-frame}
f(z_0) \in \re^3, \qquad f_{\bar z}(z_0) = \overline{f_z(z_0)}, \qquad
\det F(z_0) = \sfrac12ie^{\psi(z_0)}
\end{equation}
are satisfied, then the frame $F$ can be uniquely defined on all of $\mathcal D$ and $f$ is an immersion of a hyperbolic affine sphere in $\re^3$ with Blaschke metric $e^\psi$ and cubic form (with index lowered by the metric) $U\,dz^3 + \bar U d\bar z^3$.  Note the integrability conditions can be thought of as the flatness condition for the connection $D$ in (\ref{gauss-eq}). The map $f\!:\mathcal D \to \re^3$ is the \emph{developing map} for the affine sphere, and $[f]$ is a developing map for the corresponding $\rp^2$ structure (see e.g.\ \cite{loftin-mcintosh-survey}).  We will use more details of this developing map below: Given appropriate initial conditions as above, the equations (\ref{evolve-frame}) become a linear system of ODEs along any path from $z_0$ in $\mathcal D$. The integrability conditions ensure that the solution to the ODE initial value problem is independent of paths in a given homotopy class.

The initial value problem is particularly useful in the case the Blaschke metric $e^\psi|dz|^2$ is complete, as then Theorem \ref{complete-has} above shows the developed image $f(\mathcal D)$ is a hyperbolic affine sphere asymptotic to the boundary of a convex cone $\mathcal C\subset\re^3$ and diffeomorphic  to a properly convex domain $\Omega\subset \rp^2$.

We also consider a Riemann surface $\Sigma$ with universal cover $\mathcal D$.  Assume $\Sigma$ is a Riemann surface equipped with a holomorphic cubic differential $U$ and background conformal metric $g$. Then if $e^\psi|dz|^2 = e^ug$, the elliptic equation (\ref{psi-eq}) becomes
\begin{equation} \label{u-eq}
\Delta u + 4e^{-2u}\|U\|^2 -2e^u - 2\kappa = 0,
\end{equation}
where $\Delta$ is the Laplacian, $\|\cdot\|$ is the induced norm on cubic differentials, and $\kappa$ is the Gauss curvature, all with respect to $g$.
\begin{thm}\label{complete-metric-convex-rp2}
Let $\Sigma$ be a Riemman surface equipped with a holomorphic cubic differential $U$ and a conformal background metric $g$.  Let $u$ be a solution to (\ref{u-eq}) so that $e^ug$ is complete.  For a basepoint $z_0\in\tilde\Sigma$ and initial frame $F(z_0)$ satisfying (\ref{initial-frame}), we have a complete hyperbolic affine sphere $f(\tilde\Sigma)$ asymptotic to the boundary of a properly convex cone in $\re^3$. Different choices of $z_0$ and $F(z_0)$ lead to moving $f(\tilde\Sigma)$  by a motion of an element of $\Sl3$ acting on $\re^3$.

Upon projection to $\rp^2$, the universal cover $\tilde\Sigma$ is identified with a convex domain $\Omega$. Holomorphic deck transformations of $\tilde\Sigma$ correspond to orientation-preserving projective automorphisms of $\Omega$.  In this way, the triple $(\Sigma,U,e^ug)$ corresponds to a convex $\rp^2$ structure on $\Sigma$.
\end{thm}

\begin{rem}
In the case of $\Sigma$ closed, this theorem is due independently to Labourie and the author \cite{labourie97,labourie07,loftin01}. See also \cite{wang91}.  In this case, existence of solutions to (\ref{u-eq}) is also proved and uniqueness is a straightforward application of the maximum principle, and thus a properly convex $\rp^2$ structure on a closed surface $S$ of genus at least two is equivalent to a pair $(\Sigma,U)$ of a conformal structure and holomorphic cubic differential.
\end{rem}

\subsection{\c{T}i\c{t}eica's example} \label{triangle-has}
Consider the first octant in $\re^3$ as a convex cone.  The hyperbolic affine sphere associated to this cone was discovered by \c{T}i\c{t}eica \cite{tzitzeica09} (and generalized to any dimension by Calabi \cite{calabi72}). As we will use this example below, we summarize the basics of its construction.  See e.g.\ \cite{loftin02c} or \cite{dumas-wolf14} for justification.

The hyperbolic affine sphere $\mathcal H$ is equal to the set $\{(x^1,x^2,x^3) : x^i>0, x^1x^2x^3 = 3^{-\frac32}\}$. With respect to the induced Blaschke metric, $\mathcal H$ is conformally equivalent to $\co$.  If $z$ is a complex coordinate on $\co$, the Blaschke metric is given by $h=2\,|dz|^2$ and the cubic differential is $U = 2\,dz^3$.  If $z=\sigma+i\tau$, an embedding $f$ of $\mathcal H$ is given by
$$ f(z) = \sfrac1{\sqrt3} \left( e^{2\sigma}, e^{-\sigma+\sqrt3\tau}, e^{-\sigma-\sqrt3\tau}\right).$$

\section{Regular cubic differentials to regular convex $\rp^2$ structures} \label{cubic-to-rp2-sec}

\subsection{The regular limits}
We recall our earlier work in \cite{loftin02c}. For every pair $(\Sigma,U)$ for $\Sigma$ a noded Riemann surface and $U$ a regular cubic differential, a regular convex $\rp^2$ structure is constructed on $\Sigma^{\rm reg}$.  In particular, if we view $\Sigma^{\rm reg}$ as a punctured Riemann surface, then at each puncture, the residue of the cubic differential determines the local $\rp^2$ geometry of the end.  In particular, we have the following
\begin{thm} \label{thm04}
Let $\Sigma=\bar\Sigma-\{p_1,\dots,p_n\}$ be a Riemann surface of finite hyperbolic type, and let $U$ be a cubic differential on $\Sigma$ with poles of order at most 3 and residue $R_i$ at each puncture $p_i$. Then there is a background metric $g$ on $\Sigma$ and a solution $u$ to (\ref{u-eq}) so that $e^ug$ is complete. Then for the corresponding convex $\rp^2$ structure on $\Sigma$ provided by Theorem \ref{complete-metric-convex-rp2}, the $\rp^2$ holonomy and developing map of each end is determined in the following way:

For a residue $R\in\co$, consider the roots $\lambda_1,\lambda_2,\lambda_3$ of the cubic equation
$$\lambda^3 - 3(2^{-\frac23})|R|^{\frac23}\lambda - {\rm Im}\,R = 0.$$ Then the eigenvalues of the holonomy of the $\rp^2$ structure along an oriented loop around $p_i$ are given by $\alpha^i = \exp(2\pi\lambda_i)$.  When there are repeated eigenvalues, the Jordan blocks are all maximal.  In the cases where the eigenvalues are distinct (hyperbolic holonomy), the bulge parameter is $\pm\infty$, with the sign coinciding with the sign of ${\rm Re}\,R$.

More specifically, there are four cases.  First of all if $R=0$, then all $\alpha_i=1$ and the $\rp^2$ holonomy is parabolic, conjugate to $$\left(\begin{array}{ccc} 1&1&0 \\ 0&1&1 \\ 0&0&1 \end{array} \right).$$  If ${\rm Re}\,R=0$ but $R\neq0$, then there are two positive repeated eigenvalues $\alpha_1=\alpha_2$, and the holonomy is quasi-hyperbolic, conjugate to $$\left(\begin{array}{ccc} \alpha_1 & 1 & 0 \\ 0 & \alpha_1 & 0 \\ 0&0& \alpha_3\end{array} \right), \qquad \alpha_1^2\alpha_3=1.$$ If ${\rm Re}\,R\neq0$, then the eigenvalues $\alpha_1,\alpha_2,\alpha_3$ are positive and distinct, and so the holonomy is hyperbolic.  The holonomy matrix is conjugate to $$\left(\begin{array}{ccc} \alpha_1&0&0 \\ 0&\alpha_2&0 \\ 0&0&\alpha_3 \end{array} \right), \qquad \alpha_1\alpha_2\alpha_3=1.$$ The bulge parameter is $\pm\infty$, with the same sign as ${\rm Re}\,R$.
\end{thm}
\begin{rem}
In \cite{loftin02c}, the bulge parameter is called the vertical twist parameter.
\end{rem}
\begin{rem}
On a noded Riemann surface equipped with a regular cubic differential, the $\rp^2$ structures of the ends pair up to form regular separated necks.  Consider each node as two punctures glued together.  Then a regular cubic differential near the node has residues around the punctures which sum to zero.  Then we may apply the dictionary in Theorem \ref{thm04} to show the $\rp^2$ structures near each puncture satisfy the conditions for a regular separated neck as in Subsection \ref{sep-neck-subsection} above.  In particular, if the residue changes $R\mapsto -R$, then the eigenvalues of the holonomy change by $\{\alpha_i\}\mapsto \{\alpha_i^{-1}\}$. In the cases above, if $R=0$ at one puncture, the other puncture also has residue 0, and both holonomies are parabolic.  If ${\rm Re}\,R=0$ but $R\neq0$, then both holonomies are quasi-hyperbolic with inverse holonomy type.  And finally if ${\rm Re}\,R\neq0$, then the holonomies are hyperbolic inverses of each other, and the bulge parameter $\pm\infty \mapsto \mp\infty$ under $R\mapsto-R$.
\end{rem}

In the context of the present work, we may summarize the main results of \cite{loftin02c} in
\begin{thm}
Let $(\Sigma,U)$ be a pair of a noded Riemann surface $\Sigma$ and a cubic differential $U$. Then there is a corresponding regular $\rp^2$ structure on $\Sigma^{\rm reg}$ the type of whose regular separated necks is determined by the residue of $U$ at each node.  This map $\Phi$ from $(\Sigma,U)$ to regular $\rp^2$ structures is injective, and, if a mild technical condition is satisfied, the local invariants of the regular separated necks (the holonomy and bulge parameters) depend continuously on $(\Sigma,U)$ with the topology described above.
\end{thm}
Our present work improves this result in many ways: A natural topology is described on the space of regular convex $\rp^2$ structures, for which $\Phi$ is shown to be a homeomorphism.  $\Phi$ is onto. In the case of residue $0$, this is due to  Benoist-Hulin \cite{benoist-hulin13}, and recently the surjectivity of $\Phi$ is due to Nie \cite{nie15}. We also construct $\Phi^{-1}$ to prove $\Phi$ is onto. Under the topology on the space of convex $\rp^2$ structures, the local invariants addressed in \cite{loftin02c} also vary continuously, and so we have a better understanding of the geometry of the regular convex $\rp^2$ structures.  We can also remove the technical hypothesis needed in \cite{loftin02c}.  The continuity of $\Phi^{-1}$ is novel.

\subsection{Uniform estimates}
One of the main steps to construct the $\rp^2$ structures in \cite{loftin02c} is to find sub- and super-solutions to (\ref{u-eq}) which are quite precise near the punctures.  These sub- and super-solutions work well in most degenerating families of $(\Sigma,U)$, except for those in which the residue at a node is 0.  In this paper, we take a different tack: We find sub- and super-solutions which are not particularly precise but which have the virtue of being uniformly bounded in the universal curve away from singularities.  This allows us to take limits of families of solutions without the restrictions above.  Then we use a uniqueness theorem of Dumas-Wolf \cite{dumas-wolf14} to show that the limiting Blaschke metric is the one constructed in \cite{loftin02c}.  We record Dumas-Wolf's result here:

\begin{thm} \label{unique-thm}
Let $\Sigma$ be a Riemann surface which may or may not be compact, and let $U$ be a holomorphic cubic differential on $\Sigma$.  Let $g$ be a conformal background metric on $\Sigma$. Then there is at most one solution $u$ to (\ref{u-eq}) so that $e^ug$ is complete.
\end{thm}
We remark the proof of this theorem closely follows Wan \cite{wan92}, who studies similar equations for quadratic differentials.  The theorems in \cite{dumas-wolf14,wan92} are phrased in terms of differential equations on domains in $\co$.  The theorem as we state it here follows from passing to the universal cover of $\Sigma$.

Recall the finite cover of $\overline{\mathcal M}_g$ by $\{V^\alpha\}$ from Subsection \ref{plumbing-subsec} above, and consider the universal curve $\pi\!:\overline{\mathcal C}_g \to \overline{\mathcal M}_g$. Let $K^\alpha = \pi^{-1}V^\alpha$, and let $K^{\alpha,{\rm reg}}$ denote $K^\alpha$ with the nodes removed.  Recall each $V^\alpha$ consists of an $(s,t)$ neighborhood of a noded Riemann surface $\Sigma$, where the $s$ parameters represent by Beltrami differentials $\nu(s)$ which are supported away from the nodes and which preserve hyperbolic cusp coordinates, and the $t$ parameters open the nodes by taking $\{zw=0\}$ to $\{zw=t\}$.

\begin{thm} \label{u-limit-thm}
Let $(\Sigma_j,U_j)\to(\Sigma_\infty,U_\infty)$ in the total space of regular cubic differentials.  Then the corresponding Blaschke metrics $h_j$ converge in $C^\infty$ in the following sense:  We may assume the elements of the sequence all lie in one $V^\alpha\subset \overline{\mathcal M}_g$.  Then the Blaschke metrics $h_j$ converge in the same manner that the cubic differentials $U_i$ do:

In particular, there is a fixed noded $\Sigma$ so that $\Sigma_j = \Sigma_{s_j,t_j}$ with respect to the hyperbolic-metric plumbing coordinates.  On the thick part of each Riemann surface, we require that the Blaschke metrics converge upon being pulled back by the quasi-conformal diffeomorphisms induced by the Beltrami coefficients so that $[(\chi^{s_j})^{-1}]^*h_i \to [(\chi^{s_\infty})^{-1}]^*h_\infty$ in $C^\infty$. On the cusp and collar neighborhoods, we require that the Blaschke metric converges in $C^\infty_{\rm loc}$ with respect to the cusp coordinates $z$ and $w$.
\end{thm}

\begin{proof}
We use the method of sub- and super-solutions.
Consider the hyperbolic metric $k_j$ on $\Sigma_j$ as a background.  In this case, the equation for the conformal factor (\ref{u-eq}) becomes
\begin{equation} \label{u-hyp-eq}
L_j(u)\equiv\Delta_j u +4e^{-2u}\|U_j\|_j^2 - 2e^u + 2 = 0.
\end{equation}
Note that $L_j(u)\ge0$ always, and we use the hyperbolic metric as a sub-solution for our equations.  In order to find a family of super-solutions $S_j$, we need to ensure $$ S_j\ge0, \qquad L(S_j)\le 0.$$ Then we will always be able to find a solution $u_j$ satisfying $0\le u_j\le S_j$ everywhere.  Note the method of sub- and super-solutions works on non-compact Riemann surfaces (see e.g.\ \cite{loftin02c}), and it is not necessary to have an $L^\infty$ bound on the difference  $S_j-0$ of the super- and sub-solutions.

To construct a family of super-solutions, recall that with respect to the metrics $m_j$ on $\Sigma_j^{\rm reg}$, the convergence of $(\Sigma_j,U_j)$ implies there is a uniform constant $C$ so that $\|U_j\|_{m_j}\le C$. Write our metric $m_j = e^{\phi_j} k_j$. Then for a constant $B$,
$$ L_j(\phi_j+B) = e^{\phi_j}(  4\|U_j\|^2_{m_j} e^{-2B} - 2e^B - 2 \kappa_{m_j}).$$
Moreover,  (\ref{compare-g-k}) implies the grafting metric $g_j$ and the hyperbolic metric $k_j$ are uniformly comparable. By the construction of $m_j$ above, $\phi_j$ is uniformly bounded on the region of interpolation, $\frac {k_j}{g_j}\ge 1$ where $k_j$ is flat, and $\kappa_{m_j}$ is uniformly bounded.  In particular, this shows $\phi_j$ has a uniform lower bound, and so for $B$ large enough, $\phi_j+B\ge0$. Therefore, $S_j=\phi_j+B$ is a super-solution.

$S_j$ is a smooth function on each $\Sigma_j^{\rm reg}$.  Note that the $S_j$ can be chosen to vary continuously as $\Sigma_j$ changes for $t_j$ small (within our coordinate neighborhood $V^\alpha\subset \overline{\mathcal M}_g$), but within each $N^{t_i}$ neighborhood, $S_j$ varies discontinuously for values of $|t_j|$ large enough.  This is not a concern, however, as the there are still uniform bounds.  This follows since for $|t_j|$ bounded away from zero, the $g_j$ metrics on the hyperbolic collars are uniformly equivalent (depending on the bound on $|t_j|$) to the flat metrics we glue in to form $m_j$. In other words, there is a uniform positive constant $C$ so that $Cg_j\le m_j\le C^{-1}g_j$.

With the sub-solution 0 and super-solution $S_j$ in place, there is a solution $u_j$ to (\ref{u-hyp-eq}) satisfying $0\le u_j \le S_j$. This implies the Blaschke metric $h_j=e^{u_j}k_j$ is complete, since the hyperbolic metric $k_j$ is complete.

Now consider the sequence of Blaschke metrics $h_j = e^{u_j}k_j$ on $\Sigma_j^{\rm reg}$. It is a theorem of Bers that the hyperbolic metrics $k_j$ vary smoothly on compact sets of $K^{\alpha,{\rm reg}}$ \cite{bers74}. The uniform local bounds on $S_j$ on $K^{\alpha,{\rm reg}}$, together with interior elliptic estimates, imply that the $u_j$ have locally uniform $C^{2,\beta}$ estimates on each $\Sigma_j^{\rm reg}\subset K^{\alpha,{\rm reg}}.$ (See e.g. \cite{loftin02c} for the elliptic regularity argument.) This implies by Ascoli-Arzela that there is a limit (up to a subsequence) in $C^2_{\rm loc}$: $u_{j_k}\to w$ on $\Sigma_\infty^{\rm reg}$.  Since each $u_{j_k}\ge0$, $w\ge0$ as well, and $e^w k_\infty$ is  complete.  Since the convergence is $C^2_{\rm loc}$, $w$ satisfies (\ref{u-hyp-eq}) and so $e^w k_\infty$ is a complete Blaschke metric on $\Sigma_\infty$.  But Dumas-Wolf's uniqueness Theorem \ref{unique-thm} above then   shows that $w=u_\infty$.  Moreover, the same argument shows that every subsequence of $u_j$ has a subsequence which converges to $u_\infty$ in $C^2_{\rm loc}$.  This shows $u_j\to u_\infty$ in $C^2_{\rm loc}$ (and in $C^\infty_{\rm loc}$ by elliptic regularity).
\end{proof}

Since $0$ is a sub-solution to (\ref{u-hyp-eq}), we have the following
\begin{prop} \label{blaschke-dominates-hyperbolic}
Let $\Sigma$ be a Riemann surface equipped with a complete conformal hyperbolic metric.  Let $U$ be a cubic differential over $\Sigma$, and let $h=e^uk$ be a complete Blaschke metric for which $u$ satisfies (\ref{u-hyp-eq}). Then $h\ge k$ on $\Sigma$.
\end{prop}

\subsection{Convergence of the holonomy}

\begin{thm} \label{hol-limits}
Let $(\Sigma_j,U_j)\to (\Sigma_\infty,U_\infty)$ be a convergent family of pairs of (possibly) noded Riemann surfaces and regular cubic differentials.  Consider a family of parametrized smooth loops $\mathcal L_j$ based at points $p_j$ in $\Sigma_j^{\rm reg}$ which converge uniformly to $(\mathcal L_\infty,p_\infty)$ on $\Sigma_\infty^{\rm reg}$.  Then given a choice of a continuous family of initial frames at the basepoints, consider the $\rp_2$ holonomy ${\rm hol}_j\in \Sl3$. Then ${\rm hol}_j\to {\rm hol}_\infty$ up to conjugation.
\end{thm}

\begin{proof}
The idea of the proof is that since all the loops avoid the nodes, we remain in the thick part of the universal curve $\overline{\mathcal C}_g$, where we have uniform estimates, and convergence, of the cubic differentials $U_j$, the conformal factors $u_j$ and the hyperbolic metrics.  Then the convergence of the holonomy will follow from the theory of linear ODEs with parameters.

Consider the Poincar\'e disk $\mathcal D$ as a universal cover of $\Sigma_1$ so that $\tilde p_1=0$ is a lift of $p_1$. We may assume all $\Sigma_j$ for $j$ large enough are represented in a single neighborhood $V^\alpha\subset \overline{\mathcal M}_g$. Then for each Beltrami differential $\nu$ in the definition of the hyperbolic metric plumbing coordinates, consider the diffeomorphism $\tilde\chi^\nu$ of $\mathcal D$ which fixes three points (such as $1,i,-1$) on the boundary $\partial\mathcal D$. Each $\tilde\chi^\nu$ is a lift of the quasi-conformal diffeomorphism $\tilde\chi^\nu$ as in Subsection \ref{plumbing-subsec} above.  In this way, we can choose lifts $\tilde p_j\in\mathcal D$ of $p_j$ so that $\tilde p_j\to\tilde p_\infty$. Also consider lifts $\tilde{\mathcal L}_j$  based at $\tilde p_j$ of the loop $\mathcal L_j$. Let $\iota_j$ denote the deck transformation corresponding to $\mathcal L_j$.  Then $\tilde {\mathcal L}_j$ has endpoints $\tilde p_j$ and $\iota_j\tilde p_j$.

Let $\hat{\mathcal L_j}\!:[0,1]\to \mathcal D$ be the hyperbolic-geodesic constant-speed path from $\tilde p_j$ to $\iota_j\tilde p_j$. $\hat{\mathcal L}_j$ and $\tilde{\mathcal L}_j$ are homotopic, but $\hat{\mathcal L}_j$ enjoys better convergence properties: $\hat{\mathcal L}_j\to \hat {\mathcal L}_\infty$ in $C^\infty$, while $\tilde {\mathcal L}_j \to {\mathcal L}_\infty$ only uniformly.  Upon projecting to $\Sigma_j^{\rm reg}$, the image of $\hat{\mathcal L}_j$ is a hyperbolic geodesic typically with a corner at the basepoint $p_j$. The angle $\theta_j$ at this corner at $p_j$ satisfies $\theta_j\to\theta_\infty$, as $\theta_j$ is the angle at $\iota_j\tilde p_j$ between $\hat{\mathcal L}_j$ and $\iota_j\hat{\mathcal L}_j$.

Along any path (\ref{evolve-frame}) becomes a linear system of ODEs.
In the special case (which is possible on some local complex coordinate chart, but not on $\mathcal D$) in which there is a deck transformation of the form $z\mapsto z+c$,  the solution of an appropriate initial-value problem along the path is the $\rp^2$ holonomy. See \cite{loftin02c}. The reason for this is that the frame $F = (f,\,f_z,\,f_{\bar z})$ remains a frame under the holomorphic coordinate change $w=z+c$, since $f_z=f_w$ and $f_{\bar z}=f_{\bar w}$.    For more general paths, (\ref{evolve-frame}) defines a flat connection on a rank-3 vector bundle. Around each loop, the inverse of the parallel transport map of this connection is the $\rp^2$ holonomy. See  \cite{goldman90b,loftin-mcintosh-survey} for details.

To compute the $\rp^2$ holonomy along $\mathcal L_j$, first lift to the loop $\tilde{\mathcal L}_j$ as above.  Since the holonomy is independent of the choice of path in the homotopy class, we consider the geodesic path $\hat{\mathcal L}_j$.  We first solve an initial-value problem for (\ref{evolve-frame}) along $\hat{\mathcal L}_j$. Then the values of the evolved frame  at $\tilde p_j$ and $\iota_j \tilde p_j$ are not compatible, as we need a second and third contribution.  The second contribution is this: for $F=(f,\,f_z,\,f_{\bar z})$, consider the conformal coordinate $w=\iota_j z$. There is a discrepancy given by the  diagonal matrix $D(1,\frac{\partial \iota_j}{\partial z}, \frac{\partial \bar\iota_j} {\partial \bar z})$.  For the third contribution, the angle $\theta_j$ shows we must have a similar transformation for a conformal coordinate rotation of angle $\theta_j$.  We will show the $\rp^2$ holonomy converges by showing these three contributions each converge as $j\to\infty$.

First of all, the fact that each loop $\mathcal L_j$ avoids the nodes in the universal curve $\overline{\mathcal C}_g$ shows that there is an $\epsilon>0$ so that $\mathcal L_j$ lies in Thick$_\epsilon$ inside the noded Riemann surface $\Sigma_j$.  Since $(\Sigma_j,U_j)\to (\Sigma_\infty, U_\infty)$, we have $U_j\to U_\infty$ uniformly on the paths $\hat {\mathcal L}_j\to\hat{\mathcal L}_\infty$.  The conformal factors $u_j$ for each Blaschke metric, and their first derivatives, also converge in this way by Theorem \ref{u-limit-thm}.  Recall the Blaschke metric $h_j=e^{u_j}k_j$ for $k_j$ the hyperbolic metric.  On $\mathcal D$, the union of the $\hat{\mathcal L}_j$ for $j=1,2,\dots,\infty$ is compact, and so there is a $\delta>0$ so that they all lie in $\{z:|z|<1-\delta\} \subset \mathcal D$.  This shows the hyperbolic metric and $|dz|^2$ are uniformly bounded in terms of each other on the $\hat{\mathcal L}_j$, and so for $h_j=e^{\psi_j}|dz|^2$, $\psi_j$ and its first derivative converge uniformly as $j\to\infty$. As these cover all the terms in the coefficients in (\ref{evolve-frame}),
the theory of linear systems of ODE's with parameters shows that the solutions to the appropriate initial-value problems converge as $j\to\infty$.

For the second contribution, note that $\iota_j$ is determined by the hyperbolic geodesic between the endpoints $\tilde p_j$ and $\iota_j\tilde p_j$.  The convergence of $\iota_j\to\iota_\infty$ then follows since $\tilde p_j \to \tilde p_\infty$ and $\iota_j\tilde p_j \to \iota_j \tilde p_\infty$.  Finally, for the third contribution, we have already shown $\theta_j\to\theta_\infty$.
Thus the holonomy converges as $j\to\infty$.
\end{proof}

\subsection{Convergence of the developing map}
The convergence of the developing map as $(\Sigma_j,U_j)\to (\Sigma_\infty, U_\infty)$ is a little trickier, as it is necessary to consider paths that go to the boundary of the universal cover, and thus the standard theory of linear systems of ODE's with parameters does not directly apply.  Our proof is similar to Dumas-Wolf (\cite{dumas-wolf14}, the proof of Theorem 8.1).

\begin{thm}\label{conv-dev-map}
Let $(\Sigma_j,U_j)\to (\Sigma_\infty,U_\infty)$ be a convergent family of pairs of (possibly) noded Riemann surfaces and regular cubic differentials.  Let $p_j\in\Sigma_j^{\rm reg}$ and let $p_j\to p_\infty\in \Sigma_\infty^{\rm reg}$.  Consider the connected component of $\Sigma_j^{\rm reg}$ containing $p_j$, and take a universal cover of this component to be the unit disk $\mathcal D$, with a lift of $p_j$ placed at 0.  Let $\Omega_j\subset \rp^2$ be the convex domain determined by projecting the complete affine sphere determined by the initial value problem (\ref{evolve-frame}) with a fixed initial frame from $\re^3$ to $\rp^2$. Then $\Omega_j\to\Omega_\infty$ in the Hausdorff sense.
\end{thm}

\begin{proof}
We need to prove that for each $\epsilon>0$, there is a  $J$ so that for all $j\ge J$, $\Omega_n\subset N_\epsilon(\Omega_\infty)$ and $\Omega_\infty \subset N_\epsilon(\Omega_j)$, where $N_\epsilon$ is an $\epsilon$-neighborhood with the respect to the Fubini-Study metric on $\rp^2$.

Let $F_j(z)$ denote the frame for $z\in\mathcal D$ corresponding to $(\Sigma_j,U_j)$ as above. Then the component $f_j(z)$ is the parametrization of the hyperbolic affine sphere in $\re^3$, and $[f_j(z)]$ is the projection to $\rp^2$.  For $R<1$, consider the closed disk centered at the origin $\overline{B(R)}\subset \mathcal D$.  Choose $R$ so that $\Omega\subset N_{\epsilon/2} ([f_\infty](\overline{B(R)})).$

Now for $z\in\overline{B(R)}$, $F_j(z)$ can be determined by a linear system of ODE's given by integrating the frame along a radial path path from $0$ to $z$.  As above in the proof of Theorem \ref{hol-limits}, the coefficients of these ODE systems on the compact set $\overline{B(R)}$ converge uniformly as $j\to\infty$. Therefore, the theory of linear ODE systems with parameters shows $F_j\to F_\infty$ (and thus $[f_j] \to [f_\infty]$) uniformly on $\overline{B(R)}$. So there is a $J$ so that for $j\ge J$, $$[f_n](\overline{B(R)}) \subset N_{\epsilon/2}([ f_\infty](\overline{B(R)})) \subset N_{\epsilon/2}([f_\infty](\mathcal D)) = N_{\epsilon/2}(\Omega_j),$$
and thus $\Omega_\infty \subset N_\epsilon(\Omega_j)$.

To prove the opposite inclusion, we consider the dual hyperbolic affine sphere, which has the same metric $e^{u_j}k_j$ on $\Sigma_j$ and the opposite cubic differential $-U_j$, by Theorem \ref{dual-has} above.  Now we can lift the data to the universal cover $\mathcal D$ as above, and consider an appropriate initial frame to form the dual hyperbolic affine sphere and projective dual convex domain.  Then the previous case implies there is a $J$ so that if $j\ge J$, then
$$ \Omega^* \subset N_\epsilon(\Omega_j^*).$$ But then Lemma \ref{proj-dual-inclusion} implies that
$$ \Omega \supset (N_\epsilon(\Omega_j^*))^* \supset N_{\epsilon'}(\Omega_j),$$ where $\epsilon'\to0$ if and only if $\epsilon\to0$. This follows from the continuity under the Hausdorff distance of the projective duality of uniformly bounded convex domains which contain a fixed ball. This in turn implies (for convex domains) that there is an $\epsilon''$ which approaches 0 if and only if $\epsilon$ does so that
$$\Omega_j\subset N_{\epsilon''}(\Omega_\infty)$$ for $j\ge J$. This is enough to prove the theorem.
\end{proof}

\section{Regular convex $\rp^2$ structures to regular cubic differentials} \label{rp2-to-cubic-sec}

\subsection{The singular limit cases}
In this subsection, we show the regular convex $\rp^2$ structures each correspond to a pair $(\Sigma,U)$ of a noded Riemann surface $\Sigma$ and regular cubic differential $U$ on $\Sigma$.  It suffices to consider each connected component of $\Sigma^{\rm reg}$ separately. Consider a connected oriented properly convex $\rp^2$ surface each of whose ends is regular.  Then use the hyperbolic affine sphere to construct a Riemann surface of finite type and regular cubic differential so that the $\rp^2$ geometry of each end corresponds to the residue of the cubic differential as in Theorem \ref{thm04} above.  The results in this subsection are also recently due to Nie \cite{nie15}, using similar techniques.  We include our version, as we find the material both short and instructive.

Consider a single end $\mathcal E$ of $X$.  We proceed by considering the four cases of regular ends separately.

\begin{thm}\cite{benoist-hulin13} \label{parabolic-case}
Let $\mathcal E$ be an end of parabolic type. Then with respect to the Blaschke metric, $\mathcal E$ can be conformally compactified by adding a single point. The cubic differential $U$ has at worst a pole of order 2 at this point, and so the residue is 0.
\end{thm}

This case is settled by Benoist-Hulin \cite{benoist-hulin13}, who prove that the conformal structure at the end can be compactified by adding a single point, and that the corresponding cubic differential $U$ has a pole of order at most 2.  In our language, this corresponds to the residue's being 0.  To be more specific, Benoist-Hulin consider convex $\rp^2$ surfaces with finite area with respect to the Hilbert metric, which Marquis \cite{marquis12} has proved are equivalent to having a finite number of ends each with parabolic holonomy.  Thus  \cite{benoist-hulin13} is concerned with convex $\rp^2$ surface \emph{all} of whose ends are parabolic.  But the techniques used to analyze each end are essentially local, and apply to each end separately, and indeed they prove that each such end has finite conformal type and has cubic differential with residue 0.

\begin{prop}
Let $\mathcal E$ be a regular end of quasi-hyperbolic type, or of hyperbolic type with bulge parameter $\pm\infty$. Then there is a family of loops $L_s$ around $\mathcal E$ which depend on a parameter $s\to0^+$ so that
\begin{itemize}
\item $L_s$ uniformly approaches the end as $s\to0^+$. More precisely, represent $\mathcal E$ as homeomorphic to a closed half-cylinder $[0,\infty)\times S^1$. Then for every compact $K\subset \mathcal E$, there is an $\epsilon>0$ so that if $s<\epsilon$, $L_s\cap K = \emptyset$.
\item There is a family of elements $M_s\in\Sl3$ so that $M_s\,\Omega\to T$ in the Hausdorff topology as $s\to0^+$ and $M_s\,L_s$ lies in a compact subset of $T$ for $s$ small enough. Here $T$ is a triangle in $\rp^2$.
\end{itemize}
\end{prop}

\begin{proof}
The proof is broken into 3 cases.

First, we consider the case in which $\mathcal E$ is of hyperbolic type with bulge parameter $-\infty$. Choose a based loop $\mathcal L$ in $X$ freely homotopic to a loop around $\mathcal E$, and coordinates on $\rp^2$ so that the $\Sl3$ holonomy along a lift $\tilde{\mathcal L}$ of $\mathcal L$ is represented by $H =  D(\lambda,\mu,\nu)$ so that $\lambda>\mu>\nu>0$ and $\lambda\mu\nu = 1$. Let $T$ denote the principal triangle given by the projection of the first octant in $\re^3$ to $\rp^2$. Note the vertices of $T$ are the fixed points of $H$, and since the bulge parameter is $-\infty$, we may assume $\Omega$ the image of the developing map of $X$, is contained in $T$ and the boundary of $\Omega$ contains the principal geodesic $\tilde\ell=\{[t,0,1-t] : 0\le t\le 1\}$.
For $p=[1,s,1]$ as $s\to0^+$, consider the lift of a loop
$L_s = \{H^t\,p: t\in[0,1)\}$.
Let $M_s = D(s^{\frac13},s^{-\frac23}, s^{\frac13})$ so that $M_s$ acts on the hyperbolic affine sphere $\mathcal H$ by sending $ p$ to $[1,1,1]$.  For $t\in[0,1)$, we have $M_s\,H^t\,p = [1,1,1]\, H^t$. Thus the limit of of $M_s\,L_s$ lies in a bounded neighborhood of $[1,1,1]$ in $T$.

Since $\partial\Omega$ contains the principal geodesic $\tilde\ell$, $\Omega\subset T$, and $\Omega$ is convex, $M_s\,\Omega \to T$ in the Hausdorff topology as $s\to0$.  One can see this by noting $\tilde\ell$ is fixed by $M_s$, and all other points in $\bar\Omega$ approach $[0,1,0]$ as $s\to0^+$. The interior of the convex hull of $\tilde\ell$ and $[0,1,0]$ is $T$.  In fact, $M_s\,\Omega $ increases to $T$ as $s\to0^+$. See Figure \ref{bulge-minus-figure}.

\begin{figure}
\begin{center}
\scalebox{.3}{\includegraphics{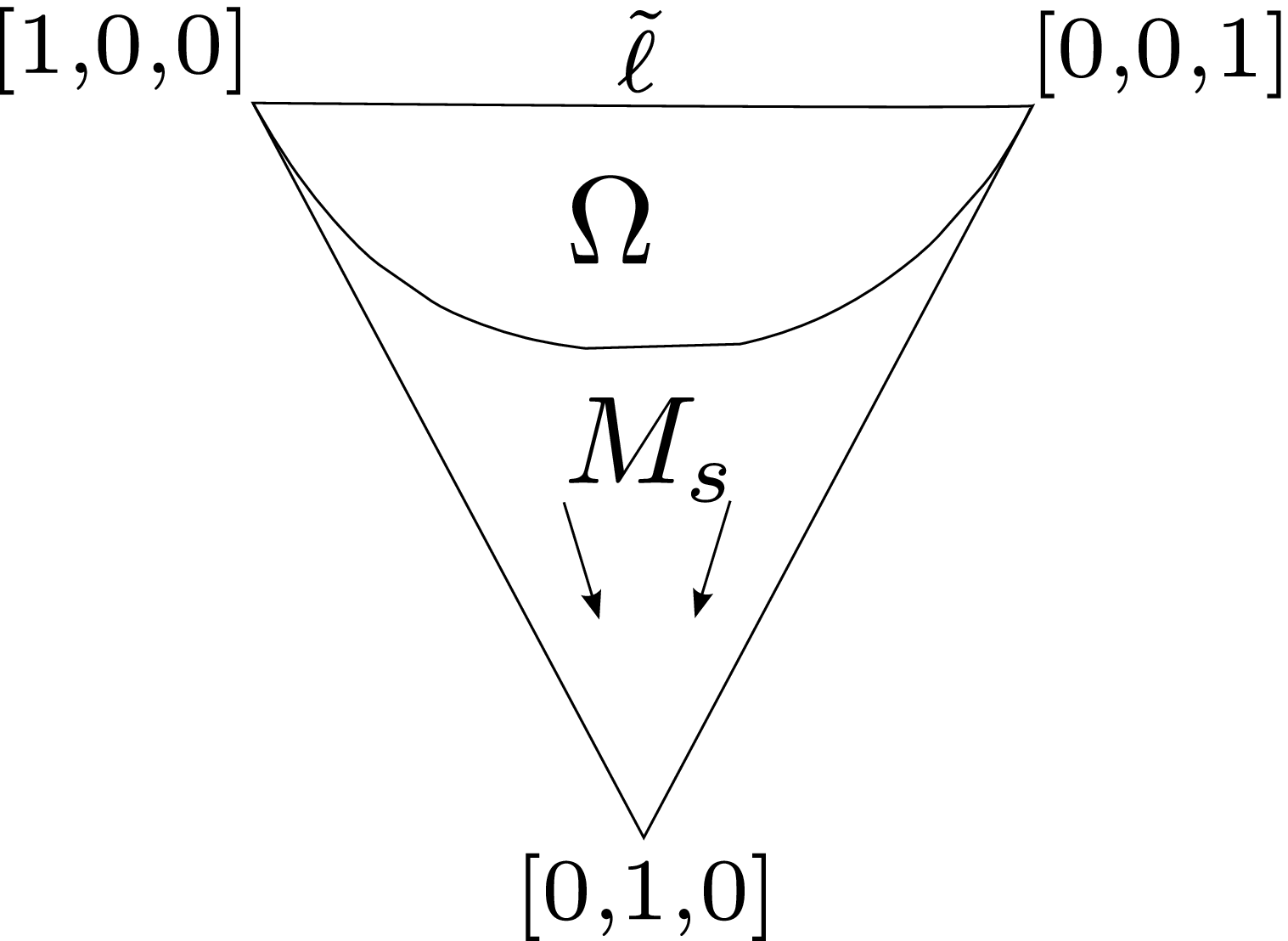}}
\end{center}
\caption{}
\label{bulge-minus-figure}
\end{figure}

The second case is of hyperbolic holonomy $H$ with bulge parameter $+\infty$. In this case, we choose coordinates so that the convex domain $\Omega$ contains $T$, $\tilde \ell$, and a proper nontrivial subset of $\bar T$.  Consider the point $p=[s,1,s]\in T$ as $s\to0^+$. Consider the map $M_s = D(s^{-\frac13},s^{\frac23},s^{-\frac13})$, which sends $p$ to $[1,1,1]$. As $s\to0$, the action of $M_s$ is essentially a bulge parameter approaching $-\infty$. Since $\Omega\cap \bar T$ is bounded away from  $[0,1,0]$, we see that $M_s\,\Omega \to T$ as $s\to 0$.

Moreover, as $s\to 0$, $p\to [0,1,0]\in\partial \Omega$, and the points in the lift of the  loop $L_s = \{H^t\,p:t\in[0,1)\}$ approach $H^t\,[0,1,0] = [0,1,0]$.  Thus the family of loops do approach the end as $s\to0$. Also,
$M_sH^t p  \to H^t\,[1,1,1]$.  Since $t\in[0,1)$, $\lim_{s\to0^+}M_s\,L_s$ lies in a compact subset $T$. See Figure \ref{bulge-plus-figure}.

\begin{figure}
\begin{center}
\scalebox{.3}{\includegraphics{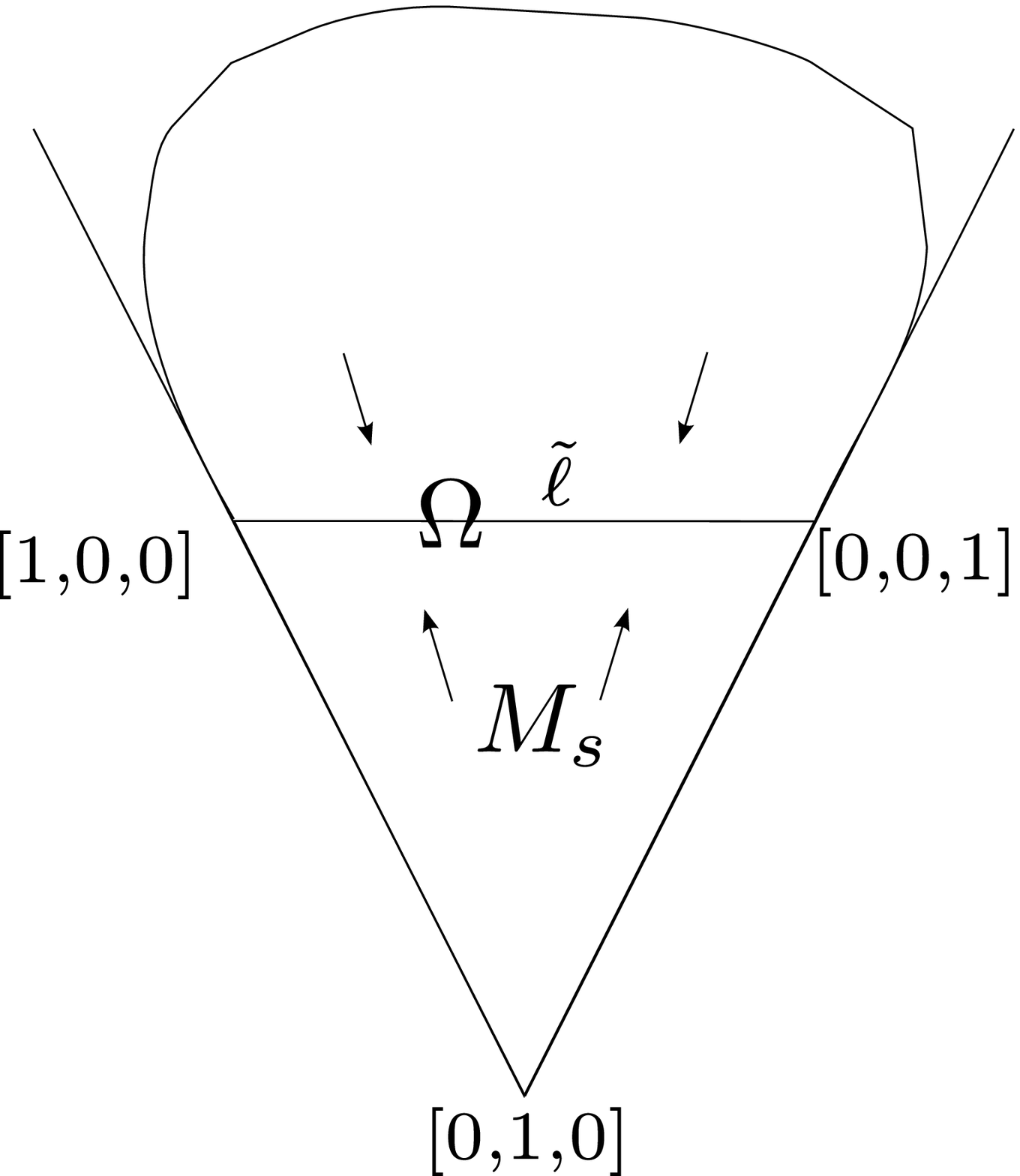}}
\end{center}
\caption{}
\label{bulge-plus-figure}
\end{figure}

The remaining case is that of quasi-hyperbolic holonomy.  It is a bit different, in that the dynamics do not involve a principal triangle.  Nevertheless, we analyze this case in terms of $T$  as well.  We may assume the holonomy matrix $H = \left(\begin{array}{ccc} \lambda & 1 & 0 \\ 0 & \lambda & 0 \\ 0&0&\mu \end{array} \right)$, with $\lambda,\mu$ positive and $\lambda^2\mu=1$. Assume without loss of generality that $\lambda>\mu$ (otherwise, we could analyze $H^{-1}$ similarly).  Then the fixed points of $H$ are the attracting fixed point $[1,0,0]$ and the repelling fixed point $[0,0,1]$.  Consider the geodesic $\tilde\ell = \{[t,0,1-t]: t\in[0,1]\}$. The proper convexity of $\Omega$ and a simple analysis of the dynamics of $H$ imply that we can choose coordinates so that $\partial\Omega\cap\{[x,y,z]:y=0\} = \tilde\ell$ and $\Omega \subset \{[x,y,z]: y,z>0\}$.

As above, we consider $p=[1,s,1]\in\Omega$ as $s\to0^+$, and note $\lim_{s\to0}p=[1,0,1]\in\partial\Omega$. Then the map $M_s = D(s^{\frac13},s^{-\frac23},s^\frac13)$ takes $p$ to $[1,1,1]$. Moreover, as $s\to0^+$, $M_s\,\Omega\to T$ in the Hausdorff sense. This can be seen because $\tilde\ell$ is fixed under the action of $M_s$.  Moreover, for every $q=[x,y,z]\in\rp^2$ with $y\neq0$, the orbit $M_s\,q$ is a straight line approaching $[0,1,0]$ as $s\to0$. This implies that any point $q\in\Omega$ which is $\epsilon$-close to $\tilde\ell$ remains $\epsilon$-close to $T$ under the action of $M_s$ as $s\to0^+$.  On the other hand, there is a $\sigma>0$ so that if $0<s<\sigma$, and $r\in\Omega$ is not $\epsilon$-close to $\tilde\ell$, then $M_s\,r$ is $\epsilon$-close to $[0,1,0]$.  Thus all points of $ M_s\,\Omega$ are within $\epsilon$ of $T$ for $s$ small enough.

Consider a family of loops $L_s = \{H^t\,p:t \in[0,1)\}$ which uniformly approaches $\partial\Omega$ as $s\to0^+$. Compute $$M_s\,H^t\,p = [\lambda^t+st\lambda^{t-1},\lambda^t,\mu^t] \to [\lambda^t,\lambda^t,\mu^t]$$ as $s\to0$. This shows that the closure of $\lim_{s\to0}M_s\,L_s$ is compactly contained in $T$. See Figure \ref{qh-figure}.

\begin{figure}
\begin{center}
\scalebox{.3}{\includegraphics{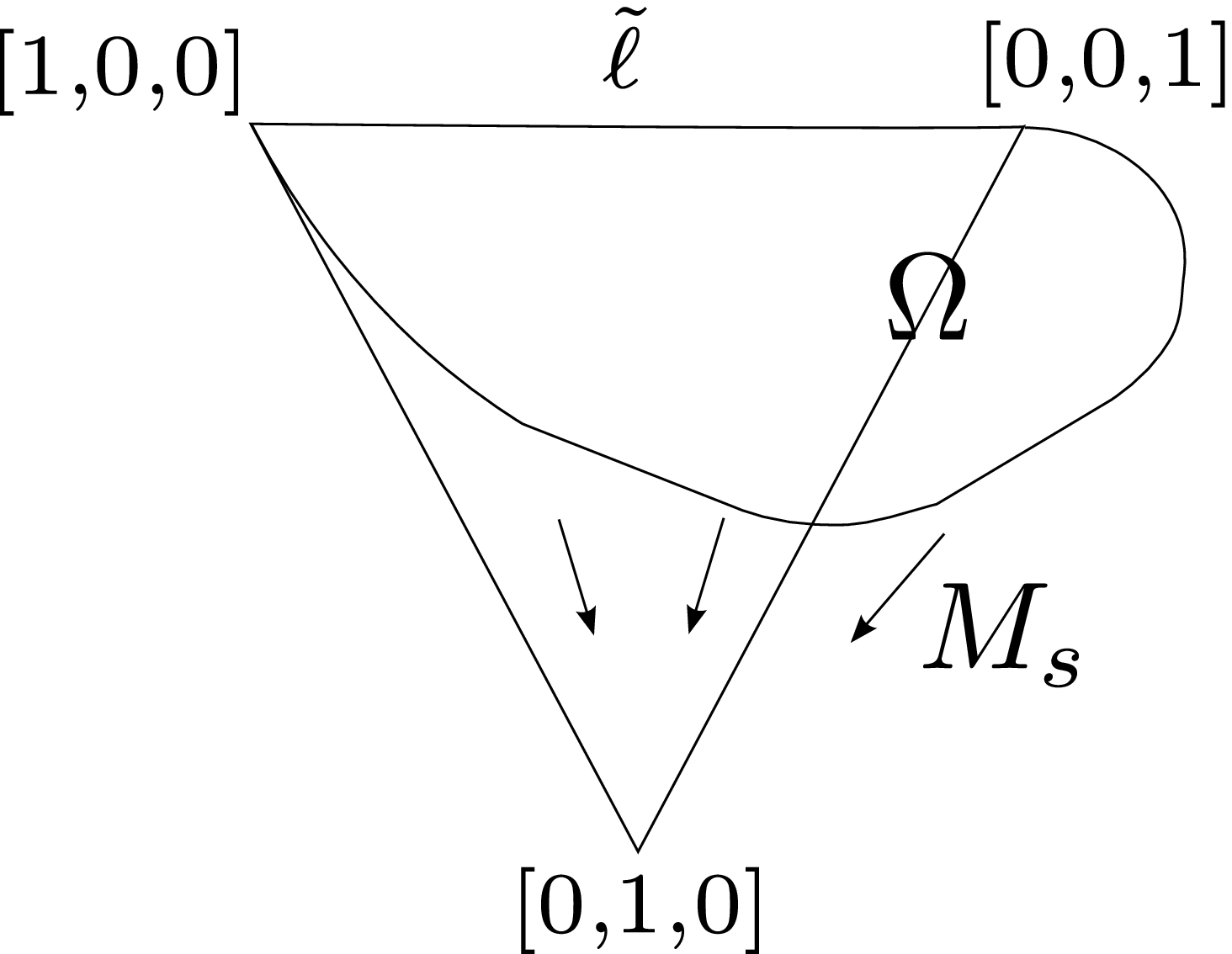}}
\end{center}
\caption{}
\label{qh-figure}
\end{figure}

\end{proof}

\begin{prop} Let $X$ be a convex $\rp^2$ surface. Let $\mathcal E$ be a regular end of $X$ of quasi-hyperbolic type, or of hyperbolic type with bulge parameter $\pm\infty$.
For the cubic differential $U$ and Blaschke metric $h$, the norm squared $\|U\|^2_h$ approaches the constant $\frac12$ uniformly at the end $\mathcal E$.  Moreover, the metric $|U|^{\frac23}$ is a flat Riemannian metric which is complete and has bounded circumference on $\mathcal E$. The induced conformal structure of the end can be compactified by adding a single point.
\end{prop}

\begin{proof}
Under the previous proposition, we know that $M_s\,\Omega\to T$ in the Hausdorff topology.  Then Theorem \ref{local-converge} applies to show the Blaschke metric and cubic tensor on $M_s\,\Omega$ converge to those on $T$ in $C^\infty_{\rm loc}$.  We know by Subsection \ref{triangle-has} above that the norm squared $\|U\|^2_h = \frac  12$ on $T$, and so on each compact subset of $T$, the same quantities on $M_s\,\Omega$ satisfy $\lim_{s\to0}\|U\|^2_{h} = \frac12$. By our construction of $M_s$, such a compact set $K\subset T$ pulls back to $M_s^{-1}\,K$, which approaches a lift of the end $\mathcal E$ in $\Omega$.  This shows $\|U\|_2^h\to\frac12$ uniformly approaching the end $\mathcal E$ on $X$.

This implies that there are no zeros of $U$ in a neighborhood of $\mathcal E$. Also, the Blaschke metric $h$ is complete by Theorem \ref{complete-has}. Then $\|U\|^2_h = |U|^2h^{-3}\to\frac12$, which shows that $|U|^{\frac23}$ is complete at $\mathcal E$. Away from the zeros of $U$, $|U|^\frac23$ is a flat metric, and thus $|U|^{\frac23}$ is a flat metric on $\mathcal E$ which is complete at the end.

Since the loops $L_s$ from the previous proposition converge to the end and the $M_s\,L_s$ lie in a compact subset of $T$ for $s$ near 0, the length with respect to the flat metric $|U|^\frac23$ (or the Blaschke metric $h$) of $L_s$ remains bounded as $s\to0$. This shows the $|U|^\frac23$-circumference of the $\mathcal E$ is bounded.  Since the metric is flat and complete, the circumference must be constant.  In fact, by considering the Euclidean developing map of $(\mathcal E,|U|^{\frac23})$, we find it must be a flat half-cylinder. This flat half-cylinder can be conformally compactified by adding one point, as one can choose a conformal coordinate $z$ so that $|U|^\frac23 = C\,|z|^{-2}\,|dz|^2$ for a constant $C$ and $0<|z|\le1$.
\end{proof}

\begin{prop}
Let $\mathcal E$ be a quasi-hyperbolic end, or hyperbolic end with bulge parameter $\pm\infty$ on a convex $\rp^2$ surface.  At the puncture on the Riemann surface induced by the Blaschke metric $h$, the cubic differential $U$ has a pole of order exactly three.
\end{prop}

\begin{proof}
Dumas-Wolf show that the completeness of $|U|^{\frac23}$ implies $U$ cannot have an essential singularity at $z=0$ \cite[Lemma 7.6]{dumas-wolf14} (see also Osserman \cite{osserman86}).  Moreover, the completeness implies $U$ has a pole of order at least 3.  The finite circumference then shows that the pole order of $U$ is at most 3.
\end{proof}

\begin{thm} \label{regular-to-cubic}
Let $X$ be a connected oriented properly convex $\rp^2$ surface of genus $g$ and $n$ ends, so that $2g+n>2$. Assume the $\rp^2$ structure of each end is regular.  Then conformal structure $\Sigma$ induced by the Blascke metric on $X$ is of finite type, and the induced cubic differential $U$ has poles of order at most 3 at each puncture of $\Sigma$.  The residue of $U$ at each puncture corresponds to the $\rp^2$ structure of the end as in Theorem \ref{thm04} above.
\end{thm}

\begin{proof}
The case of parabolic holonomy was settled by Benoist-Hulin's Theorem \ref{parabolic-case} above.

Denote the conformal structure by $\Sigma$.
For the other cases, we have shown that they each lead to a regular cubic differential $U$ of pole order 3.  We proved in \cite{loftin02c} that given such a pair $(\Sigma,U)$, we can construct from a background metric the complete Blaschke metric $\tilde h$, and also integrate the equations to determine an $\rp^2$ structure $\tilde X$ of the surface. On $\tilde X$, the residue of the corresponding cubic differential determines the holonomy and bulge parameters of the end as in Theorem \ref{thm04}.  Theorem \ref{unique-thm} shows $\tilde h=h$ and, as the $\rp^2$ structure is determined by $(\Sigma,U,h)$, the ends of the $\rp^2$ structure $X$ conform to Theorem \ref{thm04}.
\end{proof}

\subsection{Convergence in families} This paragraph contains an overview of the present subsection.
In order to show convergence of $(\Sigma_j,U_j)$ given the convergence of the corresponding regular convex $\rp^2$ structures, we rely on the thick-thin decomposition of hyperbolic surfaces.  We consider convergent Benz\'ecri sequences $(\Omega_j,x_j)\to(\mathcal O,x)$.  If the $x_j$ (subsequentially) converge to the thin part of moduli, then the hyperbolic length from $x_j$ to a bounded sequence in the thick part must diverge to infinity.  Thus the Blaschke length does the same by Proposition \ref{blaschke-dominates-hyperbolic}. Thus Proposition \ref{separate-O} shows $\mathcal O$ must remain ``disjoint" from the limits of the thick part.  On the other hand, if $x_j$ remains in the same component the thick part, we use the fact that each such component has bounded diameter.  We also derive uniform estimates on the norm of the cubic differentials $U_j$ in this case, and so we also get uniform bounds on the Blaschke metric and its derivatives on the thick part, including near $x_j$. These uniform bounds show that we can use the ODE theory as in Theorem \ref{conv-dev-map} to control the $\rp^2$ developing map near $x_j$.  These sequences will converge subsequentially, and we can use the uniqueness results to show that all convergent subsequences must converge to the same limit $(\Sigma_\infty,U_\infty)$.

$\mathcal G_S^{\rm aug}$ is a stratified space with strata $\mathcal G_S^c$ for $c\in C(S)$.  The closure of each stratum
$$\overline{\mathcal G_S^c} = \bigsqcup_{d\supset c} \mathcal G_S^d.$$
The lowest strata $\mathcal G_S^c$, in which $c$ splits $S-c$ into a disjoint collection of pairs of pants, are closed.  Thus, by considering subsequences, we may assume the limit either remains within one stratum or the limit point is on a smaller stratum (by separating one or more necks).  This means that in considering limits of sequences in $\mathcal G_S^{\rm aug}$, we may consider, by taking subsequences if necessary, only the case of elements of $\mathcal G_S^c$ approaching a limit in $\mathcal G_S^{d\cup c}$ for $d$ and $c$ fixed disjoint sets of free homotopy classes of simple loops satisfying $d\cup c \in C(S)$.  In particular, we may focus precisely on separating the necks in $d$.

For the case of families, consider a family of regular convex $\rp^2$ structures converging in $\mathcal R_S^{\rm aug}$.  The associated unmarked conformal structures must converge (subsequentially) in $\overline{\mathcal M}_g$, as it is compact.  Our first task is to show that this subsequential convergence can be extended to the regular cubic differentials as well.  For a convergent sequence in $\mathcal R_S^{\rm aug}$, we consider by Lemma \ref{orbit-sequence} convergent sequences of the form $$\oplus_k (\Omega_j,\Gamma_j|_{S_k}).$$  The induced conformal structures given by the Blaschke metric associates to each pair $(\Omega_j,\Gamma_j|_{S_k})$ a conformal structure.  We then attach the separated necks by adding a node to attach the ends of the Riemann surfaces.  Thus we have a sequence $\Sigma_j$ of noded Riemann surfaces.  Since $\overline{\mathcal M}_S$ is compact, there is a convergent subsequence $\Sigma_{j_\ell}$.

For the cubic differentials, we have by Theorem \ref{local-converge} that the norm-squared of the cubic tensor with respect to the Blaschke metric converges in $C^\infty_{\rm loc}$.  We can also prove the following universal bound

\begin{prop}
Let $S$ be an oriented surface of genus $g\ge2$.
Consider a convergent sequence in $\mathcal R_S^{\rm aug}$. Assume, by possibly taking a subsequence, that the associated conformal structures $\Sigma_j$ converge to a limit $\Sigma_\infty$ in $\overline{\mathcal M}_g$.  Assume $\Sigma_\infty$ is an element of the chart $V^\alpha$ as above in Subsection \ref{plumbing-subsec}. In terms of the metrics $m_j^\alpha$, there is a constant $C$ so that the cubic differentials $U_j$ satisfy $\|U_j\|_{m_j^\alpha}\le C$ for all $j$ large.
\end{prop}

\begin{proof}
First of all, the convergence in $\mathcal R_S^{\rm aug}$ implies by Lemma \ref{orbit-sequence} that we may lift to a convergent sequence in $\mathcal G_S^{\rm aug}$.  Let $n$ be the number of connected components of the regular limit $\rp^2$ surface.  Assume, by taking subsequences, that all the surfaces in the sequence lie in the same stratum of $\mathcal G_S^{\rm aug}$.  In particular, assume that the surface $S$ is the disjoint union of a set of loops $c\in C(S)$ and open subsurfaces $S_1,\dots,S_n$. Along each loop in $c$, there is a regular separated neck, and for $k=1,\dots,n$, there are pairs $(\Omega^k_j,\Gamma^k_j)$ of properly convex domains and discrete subgroups of $\Sl3$ acting on the domains so that the quotient $\Gamma^k_j\backslash\Omega^k_j$ is diffeomorphic to $S_k$. Moreover, the induced projective structure at each end of an $S_k$ is regular and is paired appropriately with another end of an $S_{\tilde k}$ to form a regular separated neck.

In the limit as $j\to\infty$, we may have more necks being separated.  Consider a set of homotopy classes of loops $d$ so that $d$ and $c$ are disjoint and $d\cup c\in C(S)$.  We will separate the necks along $d$.  For simplicity, we only consider the case of a single loop in $d$ which separates $S_1$ into two pieces $\tilde S_0$ and $\tilde S_1$.  In this case, we have  $\rho_j,\sigma_j\in\Sl3$ so that $\rho_j(\Omega^1_j,\Gamma^1_j|_{\tilde S_0}) \to (\mathcal O,G)$ and $\sigma_j(\Omega^1_j,\Gamma^1_j|_{\tilde S_1}) \to (\mathcal U, H)$, so that $\tilde S_0$ and $\tilde S_1$ are diffeomorphic to $G\backslash\mathcal O$ and $H\backslash \mathcal U$ respectively.

In particular, $\rho_j\Omega^1_j\to \mathcal O$ in the Hausdorff topology.  Theorem \ref{regular-to-cubic} implies that $(\mathcal O,G)$ is topologically conjugate to a non-elementary finitely-generated Fuchsian group of the first kind. In particular, there is a diffeomorphism $\phi\!:\mathcal O \to \mathcal D$ for the Poincar\'e disk $\mathcal D$ so that $\phi\circ G\circ \phi^{-1}$. Moreover, the Riemann surface $(\phi\circ G\circ \phi^{-1})\backslash \mathcal D$ has finite hyperbolic area. Consider the Dirichlet domain, which a convex ideal polygonal fundamental region $\mathcal P$ for $\phi\circ G\circ \phi^{-1}$ with finitely many sides and for which each ideal vertex corresponds to an end of the quotient surface $\tilde S_0$; see e.g.\ \cite{beardon83}.  Let $K\subset \mathcal O$ be a compact set large enough so that all of $\mathcal P$ outside  neighborhoods of the ideal vertices is in the interior of $\phi(K)$.  Theorem \ref{local-converge} implies the Blaschke metrics and cubic tensors of $\rho_j\Omega^1_j$ converge on $K$ in $C^\infty$ to those on $\mathcal O$.

Upon passing to the quotient surface $\tilde S_0$, the convergence on $K$ descends to the quotient surface to show $C^\infty_{\rm loc}$ convergence of the Blaschke metrics and cubic tensors on $\tilde S_0$ outside the ends (which are topological annuli).  The same sort of convergence is true on $\tilde S_1$ and all the other connected components of $S-(d\cup c)$. On all of $S$, then, there exist disjoint annular neighborhoods $\mathcal A_k$, one for each homotopy class of loops in $d\cup c$, so that the Blaschke metrics and cubic tensors converge in $C^\infty$ on $S-\cup_k\mathcal A_k$.

By our assumption, the necks in $d$ are precisely those which are conformally pinched as $\Sigma_j\to\Sigma_\infty$. So we may assume each $\mathcal A_k$ contains the thin part of each collar neighborhood in $\Sigma_j$. In other words, there is an $\epsilon>0$ so that $\Sigma_j-\cup_k\mathcal A_k \subset {\rm Thick}_\epsilon$.  The Blaschke metric, the hyperbolic metric, and the modified metrics $m^\alpha$ are thus all uniformly equivalent (depending on $\epsilon$) on $\Sigma_j-\cup_k\mathcal A_k$.  Therefore, the uniform convergence of the cubic tensors and Blaschke metrics on $\Sigma_j-\cup_k\mathcal A_k$ implies the uniform convergence of $\|U_j\|_{m^{\alpha,j}}$ on $\Sigma_\infty-\cup_k\mathcal A_k$.  So for large enough $j$, there is a uniform bound on  $\|U_j\|_{m^{\alpha,j}}$ when restricted to $\Sigma_j-\cup_k\mathcal A_k$.

The next lemma shows this uniform bound can be extended to a uniform bound of $\|U_j\|_{m^{\alpha,j}}$ on all of $\Sigma_j^{\rm reg}$.
\end{proof}

\begin{lem} \label{cubic-diff-bound}
Let $\Sigma$ be a noded Riemann surface represented by a point in $V^\alpha\subset \overline{\mathcal M}_g$ with metric $m^\alpha$. Let $U$ be a regular cubic differential on $\Sigma$.  Let $\mathcal A_k$ be a collection of disjoint sets of the following forms: either 1) an annular subset of $\Sigma$ or 2) a neighborhood of a node which is homeomorphic to $\{zw=0:|z|,|w|<1\}$ with respect to the plumbing coordinates.  Assume  $\mathcal A_k$ contains a component of the locus where the $m^\alpha$ metric is flat.   Then there is a constant $C$ depending only on the genus so that for all $x\in\Sigma^{\rm reg}$, $$\|U(x)\|_{m^\alpha} \le C \sup\{\|U(z)\| _{m^\alpha} : z\in\Sigma^{\rm reg} - \cup_k\mathcal A_k\}.$$
\end{lem}

\begin{proof}
See e.g.\ \cite{wolpert10, wolpert12}. For simplicity, we consider the case of a single domain $\mathcal A$.

We consider two cases.  First of all, let $\mathcal A$ be an annulus.  If $\mathcal A$ is equal to $\mathcal{F}\equiv\{\ell : m^\alpha = (2\log c)^{-2}|d\ell|^2\}$, then the $m^\alpha$ metric is flat on $\mathcal A$.  For the quasi-coordinate $\ell = \log z$, we have $m^\alpha = 2(\log c)^{-2} |d\ell|^2$.  Thus $\|U\|_{m^\alpha}$ is, up to a constant, the same as $|\tilde U|$, for $U $ represented locally as $\tilde U\,d\ell^3$.  Thus  the maximum modulus principle implies that $\sup\{\|U(x)\|_{m^\alpha} : x \in \mathcal A\} \le \sup\{\|U(x)\|_{m^\alpha} : x \in \partial\mathcal A \}$.

On the other hand, if the annulus $\mathcal A$ is not contained in $\mathcal F$, then outside this set, the metric $m^\alpha$ is uniformly equivalent to the hyperbolic metric.   Thus if we attempt to extend the flat metric $(2\log c)^{-2}|d\ell|^2$ to all of $\mathcal A$, the hyperbolic metric differs from the flat metric by a conformal discrepancy whose size depends is bounded by a function only of the hyperbolic distance to the flat part, as the metric is given by (\ref{hyp-met-collar}) above.  The universal bound on the hyperbolic diameter on the thick part (see e.g.\ p.\ 9 in Wolpert \cite{wolpert10}) then provides the constant $C$ as needed.

The remaining case is in which $\mathcal A$ is a neighborhood of the node the regular part of which is two punctured disks. If $\mathcal A$ is exactly the locus $\mathcal F$ in which $m^\alpha = 2(\log c)^{-2} |d\ell|^2$, the maximum of $\|U\|_{m^\alpha}$  must occur at the boundary.  Moreover,  the asymptotic value $\|U\|_{m^\alpha}$ at the node when $z=w=t=0$  is equal to $|R|\cdot|\log c|^3$, where $R$ is the residue of $U$ and $c$ is a uniform constant. But $R$ is determined by a Cauchy integral formula for $\tilde U$ integrated along the boundary of the disk. Thus in this case, we have the same sort of bounds as above.  The analysis involving the hyperbolic distance is also valid by (\ref{hyp-met-cusp}) above, and we may produce the uniform constant $C$ needed.
\end{proof}

Now as the cubic differentials remain uniformly bounded in the $m^{\alpha,j}$ metrics, they subsequentially converge to a regular limit $(\Sigma_j,U_j)\to(\Sigma_\infty,U_\infty)$ (ignoring the subsequence in the notation). Then Theorems \ref{hol-limits} and \ref{conv-dev-map} above imply that $(\Omega_j,\Gamma_j|_{S_k}) \to (\mathcal O_k,G_k)$ for $k=1,\dots,n$, where $n$ is the number of components of $\Sigma_\infty^{\rm reg}$.  Let $\{\Sigma^k_\infty\}$ denote the corresponding components of   $\Sigma_\infty^{\rm reg}$.

We investigate these regular limits of convex $\rp^2$ structures. Assume again that the regular convex $\rp^2$ structures lie in a single stratum, in which the surface $S$ is already separated into pieces as $S-c$ for $c\in C(S)$.  Then additional necks may be separated by choosing $d$ disjoint from $c$ and $d\cup c\subset C(S)$.  We consider a single connected component $S_1$ of $S-c$, and let $\tilde S_k$ be a connected component of $S_1-d$. In this case, we have $(\Omega_j,\Gamma_j)$ so  that $\Gamma_j \backslash \Omega_j$ is diffeomorphic to $S_1$ and $(\Omega_j,\Gamma_j|_{\tilde S_k}) \to (\mathcal O, G)$ with $G$ acting properly discontinuously and discretely on $\mathcal O$ so that the quotient is diffeomorphic to $\tilde S_k$.  We will prove any other limit is equivalent up the the action of $\Sl3$ and the mapping class group.

First of all, recall that, given a basepoint in $\tilde S_k$, the fundamental group of the open subsurface $\tilde S_k$ of $S_1$ can be represented as a conjugacy class of subgroups of $\pi_1S_1$. We have shown above that we may pick one element of this conjugacy class to represent $\Gamma_j|_{\tilde S_k}$ as a sub-representation of $\Gamma_j$.  If we further characterize the boundary of $\tilde S_k\subset S_1$ to be given by a collection of principal geodesics, then we may choose the image of the developing map $\Omega_j^k$ as a subset of $\Omega_j$ on which $\Gamma_j|_{\tilde S_k}$ acts so that the quotient $\Gamma_j|_{\tilde S_k}\backslash\Omega_j^k$ is diffeomorphic to $\tilde S_k$ with principal geodesic boundary.  Goldman's Theorem \ref{glue-rp2} then shows that the surface $S_1$ is reconstructed from gluing the subsurfaces $\tilde S_k$ together in a standard combinatorial way, which we detail in the next three paragraphs.
In the three paragraphs which follow, we suppress the dependence on the index $j$ in our sequence of domains.

We can represent $S_1$ as the disjoint union of open subsurfaces $\tilde S_k$, $k=1,\dots,m$ and free homotopy class of loops $\ell_i\in c$. Combinatorially, we may represent $S_1$ as a connected graph with nodes $\tilde S_k$ and connected by edges $\ell_i$.  Now consider an image of the developing map $\Omega^k$ for each $\tilde S_k$.  Then we follow Goldman \cite{goldman90a} to reconstruct the image $\Omega$ of the developing map of $S_1$ from many copies of the $\Omega^k$.  Begin by analyzing a single loop $\ell_1$ which connects $\tilde S_1$ to $\tilde S_2$.  Fix $\Omega^1$ and pick a lift $b\subset \partial \Omega$ of $\ell_1$. Then choose $\gamma\in\Sl3$ which acts on $\Omega^1$ by a hyperbolic action on the principal segment $b$.  Similarly, there is a $\rho\in \Sl3$ so that the closures $\overline{\Omega^1} \cap\overline{\rho\Omega^2} = \bar b$ and $\overline{\Omega^1}\cup \overline{\rho\Omega^2}$ is convex.  We may repeat this attaching process along all copies of $\ell_1$ in order to glue $\tilde S_1$ and $\tilde S_2$ along $\ell_1$. Then this process can be repeated for all the other copies of the same principal segment, which can be enumerated by $\delta b$ for $\delta$ in the coset space $\Gamma(\pi_1S_1)/\langle \gamma\rangle$.  (We have assumed in our notation that $\tilde S_1\neq \tilde S_2$.  The case in which $\tilde S_1 = \tilde S_2$, and thus $\tilde S_1$ is attached to itself across $\ell_1$, is essentially the same.)

We repeat this process with other loops in $d$, and then describe $\Omega$ as a disjoint union of copies of $\Omega^1,\dots,\Omega^m$ and lifts of loops in $d$.  The combinatorial structure of this union can be described by an infinite-valence tree, with each vertex corresponding to a copy of $\Omega^k$ and each edge corresponding to a lift of a principal geodesic segment across which the two domains represented by the vertices are attached.  The fact that this graph is a \emph{tree} is a consequence of the injectivity of the developing map \cite{goldman90a}. For $\Omega^1$ as in the previous paragraph, there is one adjacent edge for each $\delta\in\Gamma(\pi_1S_1)/\langle\gamma\rangle$, which corresponds to the principal geodesic segment $\delta b$. The other vertex of this edge corresponds to the domain $\delta\rho\Omega^2$.  (If there are other loops in $c$ which border $\tilde S_1$, then there will be corresponding edges from $\Omega^1$ as well.)  Denote the domains represented by vertices in the graph by $\mathcal O_i$. Each $\mathcal O_i = \sigma\Omega^k$ for $\sigma\in\Sl3$ and $1\le k\le m$.

Now we consider the action of $\Gamma^1$ on $\Omega$.  $\Gamma^1$ acts on the sub-domain $\Omega^1$. For $I$ the identity matrix, we have
\begin{lem} \label{action-glued-domains}
\begin{itemize}
\item
 $\gamma\in\Gamma^1(\pi_1S_1)-\{I\}$ acts on the boundary segment $b$ of $\Omega^1$ if and only if $\gamma$ is a hyperbolic action on the principal geodesic segment $b$.
\item $\gamma\in\Gamma^1(\pi_1S_1)-\{I\}$ acts on $\mathcal O_i$ if and only if $\mathcal O_i$ is adjacent to  $\Omega^1$ and $\gamma$ is a hyperbolic action on the principal geodesic segment separating $\Omega^1$ and $\mathcal O_i$.
\end{itemize}
\end{lem}

With this combinatorial picture set up, we assume $(\Omega_j,\Gamma_j|_{S_k})\to (\mathcal O, G)$ so that $G\backslash \mathcal O$ is diffeomorphic to $S_k$.

Recall Benz\'ecri's compactness theorem \cite{benzecri60}, that for every sequence $(\Omega_j,x_j)$ for $\Omega_j$ properly convex and $x_j\in\Omega_j$, that upon passing to a subsequence, there are $\rho_j\in\Sl3$ so that $\rho_j(\Omega_j,x_j)\to (\mathcal O, x)$ in the Hausdorff topology.  We analyze our limits of $(\Omega_j,\Gamma_j|_{S_k})$ in terms of these Benz\'ecri limits of pointed convex domains.  Recall that the surface $S_1$ has genus $\tilde g$ and $n$ punctures, where $\tilde g+2n\le g$ the genus of $S$.
In order to analyze the limits, we consider the conformal structure induced by the Blaschke metric on $\Gamma_j\backslash \Omega_j$ as an element of the compact space $\overline{\mathcal M}_{\tilde g,n}$, and then consider sequences of points in the corresponding universal curve.

Consider a convergent sequence in Benz\'ecri's sense $(\Omega_j,x_j)\to (\mathcal O,x)$.  By taking a subsequence if necessary, assume $x_j$ converges in the universal curve $\overline{\mathcal C}_{\tilde g,n}$.  Denote by $R_j$ the noded Riemann surface containing $x_j$.

\begin{prop} \label{unique-regular-convex-limit}
Up to the actions of $\Sl3$ and the mapping class group, there is exactly one limit of the sequence of pairs  $(\Omega_j,\Gamma_j|_{S_k})$ for each $k$.
\end{prop}

\begin{proof}
Consider a convergent Benz\'ecri sequence $(\Omega_j,x_j)\to(\mathcal U,x)$.  By choosing a subsequence if necessary, we consider two cases, as in Lemma \ref{thick-lemma} above.

First of all, consider the case in which $[x_j]$ converges to a node or a puncture in $\overline{\mathcal C}_{\tilde g,n}$, where $[x_j]$ denotes the image of $x_j$ in the Riemann surface conformal to the quotient $\Gamma_j\backslash\Omega_j$ equipped with the Blaschke metric. In this case, as follows from Subsection \ref{plumbing-subsec} above, hyperbolic geodesic balls of fixed radius centered at $[x_j]$ are all contained in cusp/collar neighborhoods in $R_j$.  Since the Blaschke metric is bounded from below by the hyperbolic metric  (Proposition \ref{blaschke-dominates-hyperbolic}), the same is true of geodesic balls with respect to the Blaschke metrics on $\Gamma_j\backslash\Omega_j$. In fact, the Blaschke distance from $x_j$ to any point in the thick part of $R_j$ has an infinite limit as $j\to\infty$.  Then Lemma \ref{action-glued-domains} implies that, for every $\delta_j\in\Gamma_j|_{S_k}(\pi_1S_k)-\langle\gamma_j\rangle$,  the Blaschke distance from $x_j$ to $\delta x_j$ diverges to infinity, where $\gamma_j$ is the projective holonomy around the neck determined by cusp/collar neighborhood.  Proposition \ref{separate-O} implies that $\Gamma_j|_{S_k}$ cannot converge to act on the limiting domain $\mathcal U$.  This rules out this case.

Second, consider the case in which $[x_j]$ converges to a limit in $\overline{\mathcal C}_{\tilde g,n}$ which is not a node or puncture, then for all large $j$, the $[x_j]$ lie uniformly in the thick part of the Riemann surfaces $R_j$.  Now if the $[x_j]$ lie in a different connected component of the thick part of $R_j$ from $S_k$, then the same considerations as in the previous paragraph apply to rule this out.

Therefore, we may assume that $[x_j]$ converges to a limit in $\overline{\mathcal C}_{\tilde g,n}$ so that $[x_j]$ is in a component of the thick part of the $R_j$ which overlaps with $S_k$. Recall that we have already taken a subsequence to show $(\Sigma_j,U_j)\to(\Sigma_\infty,U_\infty)$ and that this convergence by Theorems \ref{hol-limits} and \ref{conv-dev-map} implies the convergence of $(\Omega_j,\Gamma_j|_{S_k})\to (\mathcal O,G)$. The proofs of Theorems \ref{hol-limits} and \ref{conv-dev-map} show that we may fix diffeomorphisms $\phi_j\!:\Omega_j\to \mathcal D$ so that $\mathcal D$ is the Poincar\'e disk, $\phi_j$ is conformal with respect to the Blaschke metric, $\phi_j^{-1}(0)$ lies in each $\Omega_j$, and $\phi_\infty^{-1}(0)\in\mathcal O$. We may rephrase our assumption to state that $[x_j]$ lies in the same component of the thick part as $[\phi_j^{-1}(0)]$.

But there are uniform bounds on the hyperbolic diameter of connected components of the thick part of Riemann surfaces. See e.g.\ \cite{wolpert10}, page 9.  In particular, the hyperbolic distance from $[x_j]$ to $[\phi_j^{-1}(0)]$ is uniformly bounded by a constant $C$.  Therefore, we may consider a lift $\tilde x_j$ of $[x_j]$ so that the hyperbolic distance from 0 to $\phi_j(\tilde x_j)$ is bounded by $C$.  Passing from $x_j$ to $\tilde x_j$ corresponds to the action of an element $\rho_j\in \Gamma_j|_{S_k}$.  See Subsection \ref{sep-neck-subsection} above.  Lemma \ref{cubic-diff-bound} above shows the cubic differentials $U_j$ on $R_j$ are uniformly bounded with respect to the $m^{\alpha,j}$ metric.  Moreover, on the thick part, the hyperbolic metric and the $m^{\alpha,j}$ metrics are uniformly equivalent, and the conformal factors $u_j$ of the Blaschke metrics are uniformly bounded in the $C^1$ norm.

Choosing an appropriate initial frame, we may integrate the structure equations for the affine sphere as in Theorem \ref{conv-dev-map} to show that the limit $\tilde x_\infty\in\mathcal O$ and, as the previous paragraph shows the coefficients of the relevant ODE system are uniformly bounded, $\tilde x_j\to\tilde x_\infty$ (up to a subsequence).  Now we already have assumed that $(\Omega_j,x_j)$ converges to $(\mathcal U,x)$ in the space of pointed convex domains in $\rp^2$ modulo the action of $\Sl3$.  We have just shown that a subsequence $(\Omega_{j_i},x_{j_i})$ converges to $(\mathcal O,\tilde x_\infty)$.  The Benz\'ecri space of pointed properly convex domains in $\rp^2$ modulo $\Sl3$ is Hausdorff \cite{benzecri60}; thus there is a $\rho\in\Sl3$ so that $(\mathcal U,x)=\rho(\mathcal O,\tilde x_\infty)$.  Moreover, every subsequence of $(\Omega_j,x_j)$ itself has a subsequence converging to $(\mathcal U,x)$ in the Benz\'ecri sense, and so we find $(\Omega_j,x_j)\to(\mathcal O, \tilde x_\infty)$ up to the action of $\Sl3$.

To address the convergence of the representations $\Gamma_j|_{S_k}$, recall that an element of $\Sl3$ is determined by its action on 4 points in general position in $\rp^2$.  Luckily, the estimates we have proved are strong enough to control the geometry of a uniformly large neighborhood of $\tilde x_\infty$, and points in this neighborhood will serve as our 4 points in general position.  In particular, as $j\to\infty$, we have a neighborhood $\mathcal N$ of $\phi_\infty^{-1}(\tilde x_\infty)$ in the Poincar\'e disk $\mathcal D$, and uniform estimates on $\mathcal N$ of the cubic differentials $U_j$, the conformal factors $u_j$, and their derivatives. (Proof: We have shown above that there is a uniform $\epsilon$ so that $[x_j]\in {\rm Thick}_\epsilon$ for $j=1,2,\dots,\infty$. This shows there is a uniformly large neighborhood $\mathcal N$ around $\tilde x_j$ for $j$ large so that the projection of $\mathcal N$ to $\Sigma_j^{\rm reg}$ is contained in Thick$_{\epsilon/2}$.  See e.g.\ Lemma 1.1 in \cite{wolpert10} for a justification.  On the thick part, we have uniform bounds on $U_j$, $u_j$ and $du_j$.)

Upon choosing a suitable initial frame, the diffeomorphism $\phi_j^{-1}\!: \mathcal D \to \Omega_j$ is constructed by solving the ODE system (\ref{evolve-frame}), choosing the component $f$ of the frame $F$, and finally projecting from $\re^3\to\rp^2$.  The uniform estimates on $\mathcal N$ imply that there are open sets $\mathcal A$ and $\mathcal B$ so that $\phi_\infty^{-1}(\tilde x_\infty) \in \mathcal A \subset \mathcal N$, $\tilde x_\infty \in \mathcal B \subset \phi_\infty(\mathcal N)$, and for all $j$ large, $\tilde x_j\in \mathcal B$, $\phi_j(\tilde x_j)\in \mathcal A$, $\phi_j$ and its derivatives are bounded on $\mathcal A$, and $\phi_j^{-1}$ and its derivatives are bounded on $\mathcal B$.  (This is just a quantitative version of the Inverse Function Theorem.)

We assume $\rho_j(\Omega_j,\Gamma_j|_{S_k})\to (\mathcal O,G)$, and we have shown that there is a sequence $x_j\in\Omega_j$ so that $\rho_j(x_j)\to x\in\mathcal O$ (this $x$ is referred to as $\tilde x_\infty$ above), and $x_j$ is in the same connected component of the thick part of $\Gamma_j\backslash \Omega_j$ as $S_k$ is (this follows from the uniform estimates on $\mathcal A$ and $\mathcal B$ in the previous paragraph).  Absorb the $\rho_j$ into the notation for $\Omega_j$ and $\Gamma_j|_{S_k}$, so that $(\Omega_j,\Gamma_j|_{S_k})\to (\mathcal O,G)$ and $x_j\to x$.  Let $\gamma^1,\dots,\gamma^m$ be generators of $G$.

Let $x^a=x,x^b,x^c,x^d$ be in general position in $\mathcal O$, and assume that they are in a small neighborhood of $x$.  In particular, for $p\in\{a,b,c,d\}$, let $K^p$ be the convex hull of $\{x^a,x^b,x^c,x^d\}-\{x^p\}$. Assume for a choice of an affine coordinate patch in $\rp^2$ that there are six  open disks $D_a,D_b,D_c,D_d,D_2,D_3$ so that
\begin{itemize}
\item
The closure  $\bar D_p$ is contained in the interior of  $K^p$,
\item $\bar D_2\subset\mathcal O$ and each $x^p\in D_2$, and
\item $\bar{\mathcal O}\subset D_3$.
\end{itemize}

Then for $j$ large,  all these points $x^a,x^b,x^c,x^d$  will be in the same connected component as $x_j$ of the thick part of $\Gamma_j\backslash \Omega_j$ (this is a consequence of the Inverse Function Theorem argument above). As $\Gamma_j|_{S_k}\to G$, let $\gamma_j^i\in\Gamma_j|_{S_k}$ converge to $\gamma^i$ for $i=1,\dots,m$ as $j\to\infty$. For large $j$, the $\gamma_j^1,\dots,\gamma_j^m$ still generate $\Gamma_j|_{S_k}$. Then the set $\{\gamma_j^ix^p : i=1,\dots,m; p=a,b,c,d\}$ determines the generators of $\Gamma_j|_{S_k}$ and thus also the group $\Gamma_j|_{S_k}$ itself.

Now we prove the uniqueness of $G$ (up to a possible $\Sl3$ action). Recall we assume that $(\Omega_j,\Gamma_j|_{S_k})\to (\mathcal O, G)$.  We have established there is an $x_j$ so that $\rho_j(\Omega_j,x_j)\to(\mathcal O, x)$. Now consider another sequence $\sigma_j(\Omega_j,\Gamma_j|_{S_k})\to (\mathcal O, H)$.  Consider the points $x^a,x^b,x^c,x^d$ in general position in $\mathcal O\subset \rp^2$.  Then, as above (recalling that $x^a,x^b,x^c,x^d$ and their convex hull are uniformly contained in the thick part of $\Gamma_j\backslash\Omega_j$, for large $j$), $\{\sigma_j x^p\}$ remains in general position, and there are still uniformly large ellipses $D_a,D_b,D_c,D_d,D_2,D_3$ as above (this follows from a transversality argument based on the Inverse Function Theorem analysis above).  The existence of these bounding ellipses shows that the $\sigma_j(x^p)$ remain uniformly in general position in $\rp^2$, and that the family $\sigma_j$ lies in a compact subset of $\Sl3$.  Thus there is a limit $\sigma_j\to\sigma$ (upon taking a subsequence).

For $i=1,\dots,m$, let $\eta^i=\sigma\gamma^i\sigma^{-1}$. These $\eta^i$ generate $H$. Similarly, define $\eta_j^i = \sigma_j\gamma_j^i\sigma_j^{-1}\in \sigma_j(\Gamma_j|_{S_k})\sigma_j^{-1}$. Then for large $j$, $\eta_j^i$ generate $\sigma_j(\Gamma_j|_{S_k})\sigma_j^{-1}$ and $\lim_{j\to\infty} \eta_j^i=\eta^i$. This implies $H=\sigma G\sigma^{-1}$.  Moreover, since $\mathcal O$ has already been fixed, $\sigma\in\Sl3$ is a projective automorphism of $\mathcal O$.  Therefore, the two limits $(\mathcal O, G)$ and $(\mathcal O,H)$ are equivalent up to the action of $\Sl3$. More precisely, for all subsequences of $(\Omega_j,\Gamma_j|_{S_k})$, there is a further subsequence and an element $\sigma\in\Sl3$ so that $(\mathcal O,G) = \sigma(\mathcal O,H)$.  But then these two objects are the same modulo the action of $\Sl3$.
\end{proof}

Now we complete the proof of Theorem \ref{main-thm}. Consider a convergent sequence of regular convex $\rp^2$ structures $\oplus_j( \Omega_j,\Gamma_j)\to \oplus_m(\mathcal O_m,G_m)$, and their corresponding sequence $(\Sigma_j,U_j) = \Phi^{-1}[\oplus_j( \Omega_j,\Gamma_j)]$ of noded Riemann surfaces and regular cubic differentials.  Then there is a convergent subsequence $(\Sigma_{j_\ell},U_{j_\ell}) \to (\Sigma_\infty, U_\infty)$. Moreover, $\oplus_k( \Omega_{j_\ell},\Gamma_{j_\ell})\to \oplus_m(\mathcal O_m,G_m)$, and the regular convex $\rp^2$ structure corresponding to $(\Sigma_\infty,U_\infty)$ is $\Phi(\Sigma_\infty,U_\infty)=\oplus_m(\mathcal O_m,G_m)$.  But then  Proposition \ref{unique-regular-convex-limit} shows that every subsequence of $(\Sigma_j,U_j)$ has a subsequence which converges to the same limit.  Recall $\mathcal R_S^{\rm aug}$ is first countable. This is enough to show that $(\Sigma_j,U_j)\to (\Sigma_\infty, U_\infty)$. Therefore, $\Phi^{-1}$ is continuous, and Theorem \ref{main-thm} is proved.

\input{neck-pinch.bbl}

\end{document}

%% file: neck-pinch.bbl
\def\cprime{$'$}